\documentclass[12pt,reqno,a4paper]{amsart}

\usepackage[breaklinks,colorlinks,plainpages,hypertexnames=false,plainpages=false]{hyperref}
\hypersetup{urlcolor=blue, citecolor=blue, linkcolor=blue}

\usepackage[utf8]{inputenc}
\usepackage{amsmath}
\usepackage{amsthm}
\usepackage{amssymb}
\usepackage{amsfonts}
\usepackage{enumerate}
\usepackage{subcaption}
\usepackage{mathtools}
\usepackage{tikz-cd}
\usepackage[colorinlistoftodos,prependcaption]{todonotes}
\usepackage{regexpatch}
\usepackage{hhline}
\usepackage{stmaryrd}
\usepackage{enumitem}
\usepackage{centernot}

\usepackage{mathrsfs}
\usepackage{comment}
\usepackage{listings}
\usepackage{multicol}
\usepackage[capitalise, noabbrev]{cleveref} %
\usepackage{algpseudocode}
\usepackage{algorithm}

\usepackage{mathdots}
\usepackage{cleveref} %
\DeclareMathAlphabet{\dutchcal}{U}{dutchcal}{m}{n}

\usepackage{array}
\newcolumntype{L}[1]{>{\raggedright\let\newline\\\arraybackslash\hspace{0pt}}m{#1}}
\newcolumntype{C}[1]{>{\centering\let\newline\\\arraybackslash\hspace{0pt}}m{#1}}
\newcolumntype{R}[1]{>{\raggedleft\let\newline\\\arraybackslash\hspace{0pt}}m{#1}}

\usepackage{tikz}
\usetikzlibrary{arrows}
\usetikzlibrary{decorations.pathreplacing}
\usetikzlibrary{matrix}
\usetikzlibrary{calc}
\usetikzlibrary{shapes}
\usetikzlibrary{patterns}
\usetikzlibrary{fit,backgrounds,scopes}
\newtheoremstyle{theoremstyle}
{10pt}      %
{5pt}       %
{\itshape}  %
{}          %
{\bfseries} %
{}         %
{ }      %
{}          %

\newtheoremstyle{algorithmstyle}
{10pt}      %
{5pt}       %
{}  %
{}          %
{\bfseries} %
{}         %
{ }      %
{}          %

\newtheoremstyle{examplestyle}
{10pt}      %
{5pt}       %
{}          %
{}          %
{\bfseries} %
{}         %
{ }      %
{}          %

\makeatletter
\newtheorem*{rep@theorem}{\rep@title}
\newcommand{\newreptheorem}[2]{%
\newenvironment{rep#1}[1]{%
 \def\rep@title{#2 \ref{##1}}%
 \begin{rep@theorem}}%
 {\end{rep@theorem}}}
\makeatother

\makeatletter %
\newcommand{\subalign}[1]{%
  \vcenter{%
    \Let@ \restore@math@cr \default@tag
    \baselineskip\fontdimen10 \scriptfont\tw@
    \advance\baselineskip\fontdimen12 \scriptfont\tw@
    \lineskip\thr@@\fontdimen8 \scriptfont\thr@@
    \lineskiplimit\lineskip
    \ialign{\hfil$\m@th\scriptstyle##$&$\m@th\scriptstyle{}##$\hfil\crcr
      #1\crcr
    }%
  }%
}
\makeatother

\newreptheorem{theorem}{Theorem}
\newreptheorem{corollary}{Corollary}
\newreptheorem{proposition}{Proposition}

\makeatletter
\@namedef{subjclassname@2020}{%
  \textup{2020} Mathematics Subject Classification}
\makeatother

\makeatletter
\xpatchcmd{\@todo}{\setkeys{todonotes}{#1}}{\setkeys{todonotes}{inline,#1}}{}{}
\makeatother

\theoremstyle{theoremstyle}
\newtheorem{theorem}{Theorem}[section]
\newtheorem{lemma}[theorem]{Lemma}
\newtheorem{proposition}[theorem]{Proposition}
\newtheorem{corollary}[theorem]{Corollary}
\theoremstyle{examplestyle}
\newtheorem{example}[theorem]{Example}
\newtheorem{definition}[theorem]{Definition}

\newtheorem*{notation*}{Notation}

\newtheorem{remark}[theorem]{Remark}

\newcommand{\C}{\mathbb{C}}
\newcommand{\R}{\mathbb{R}}

\newcommand{\Q}{\mathbb{Q}}

\newcommand{\Z}{\mathbb{Z}}

\newcommand{\F}{\mathbb{F}}
\mathchardef\mhyphen="2D

\newcommand{\suchthat}{\;\ifnum\currentgrouptype=16 \middle\fi|\;}

\newcommand{\an}{{\mathrm{an}}}

\newcommand{\unr}{\mathrm{unr}}

\newcommand{\SL}{\mathrm{SL}}
\newcommand{\PSL}{\mathrm{PSL}}
\newcommand{\GL}{\mathrm{GL}}
\newcommand{\un}{\mathrm{un}}
\newcommand{\bs}{\backslash}

\newcommand{\pr}{\mathrm{pr}}

\DeclareMathOperator{\Spec}{Spec}
\DeclareMathOperator{\Spf}{Spf}

\definecolor{ududff}{rgb}{0.30196078431372547,0.30196078431372547,1}
\definecolor{qqqqff}{rgb}{0,0,1}
\definecolor{qqccqq}{rgb}{0,0.8,0}
\definecolor{ffqqtt}{rgb}{1,0,0.2}
\definecolor{wwzzff}{rgb}{0.4,0.6,1}
\definecolor{ffxfqq}{rgb}{1,0.4980392156862745,0}
\definecolor{ffqqqq}{rgb}{1,0,0}
\definecolor{ududff}{rgb}{0.30196078431372547,0.30196078431372547,1}
\definecolor{zzttqq}{rgb}{0.6,0.2,0}
\definecolor{aureolin}{rgb}{0.99, 0.93, 0.0}

\definecolor{ffffff}{rgb}{1,1,1}%
\newcommand\restr[2]{{\left.\kern-\nulldelimiterspace #1 \right|_{#2}}}

\begin{document}

\title[Toric ranks and component groups of modular curves]
{{Toric ranks and component groups\\ 
 of modular curves}}

\author{Paul Alexander Helminck}
\address{Department of Mathematics, Tohoku University, Japan.
}


\date{\today}


\begin{abstract}
Let $p\neq{2,3}$ be a prime number and let $\Gamma \subset \SL_{2}(\Z)$ be a congruence subgroup with modular curve $X_{\Gamma}/K$ and Jacobian $J(X_{\Gamma})$. In this paper we give an explicit group-theoretic description of the semistable toric rank and component group of $J(X_{\Gamma})$ at the finite places of $K$ lying over $p$. We first produce a suitable deformation retract of the minimal Berkovich skeleton of $X_{\Gamma}$ in terms of Hecke-Iwahori double coset spaces. We call this deformation retract the pruned skeleton of the curve. 
Our description of this skeleton includes a group-theoretic formula for the edge lengths, allowing us to give the component group of the modular curve as the quotient of a lattice using the monodromy pairing. For $X_{0}(N)$, $X_{1}(N)$, $X_{sp}(N)$ and $X_{sp}^{+}(N)$, we explicitly determine the pruned skeleta using a set of coset schemes over $\Z$. This in particular recovers results by Deligne-Rapoport, Edixhoven, Coleman-McMurdy and Tsushima on the semistable reduction type of $X_{0}(p^{n})$ for $n\leq{4}$. Finally, we determine the geometric Tamagawa number and the prime-to-$2$ structure of the component group of $X_{0}(N)$ over the extension given by Krir's theorem.

\end{abstract}

\maketitle

\vspace{-1cm}
\section{Introduction}\label{sec:Introduction}

Let $\Gamma\subset \SL_{2}(\Z)$ be a congruence subgroup with associated modular curve $X_{\Gamma}$ defined over a number field $K$. Let $\mathfrak{p}$ be a finite place of $K$ not lying over $2$ or $3$ and let $K_{\mathfrak{p}}$ be the associated completion. %
In this paper, we give an explicit group-theoretic description of the semistable toric rank and component group associated to the Jacobian of the curve $X_{\Gamma}/K_{\mathfrak{p}}$. That is, let $K_{\mathfrak{p}}\subset L_{\mathfrak{p}}$ be a finite extension over which $X_{\Gamma}$ attains semistable reduction and let $\mathcal{J}$ be the N\'{e}ron model of the Jacobian of $X_{\Gamma}$ over the ring of integers of $L_{\mathfrak{p}}$. We then express the toric rank of the special fiber of $\mathcal{J}$ and the geometric component group $\Psi:=\mathcal{J}_{s}/\mathcal{J}^{0}_{s}$ in terms of glued double coset spaces over a finite metric tree. We explicitly find these semistable invariants for various modular curves, including $X_{0}(N)$, $X_{1}(N)$, $X_{sp}(N)$ and $X^{+}_{sp}(N)$. For $X_{0}(N)$, we give a complete description of the prime-to-$2$ structure of $\Psi(\overline{\F}_{p})$ over the extension given by a theorem by Krir. We moreover give a formula for the geometric Tamagawa number $|\Psi(\overline{\F}_{p})|$, which includes the $2$-adic factors. 

To explain our method for finding these invariants, %
let $\C_{p}$ be the completion of the algebraic closure of $K_{\mathfrak{p}}$ and let $X^{\an}_{\Gamma}$ be the Berkovich analytification of 
$X_{\Gamma}$ over $\C_{p}$. The minimal skeleton $\Sigma(\Gamma)$ of $X^{\an}_{\Gamma}$ is a finite metric graph that can be seen as a topological representation of the dual intersection graph of a semistable model of $X_{\Gamma}$. It   %
completely determines the tropical invariants mentioned above, as the toric rank is the first Betti number of $\Sigma(\Gamma)$, and the geometric component group is the discrete tropical Jacobian associated to the $\Lambda_{L_{\mathfrak{p}}}$-valued points of $\Sigma(\Gamma)$, where $\Lambda_{L_{\mathfrak{p}}}$ is the value group of $L_{\mathfrak{p}}$. One does not need the full minimal skeleton for these, as it suffices to find a \emph{pruned} version of the skeleton. That is, one only needs the metric graph obtained by 
contracting all maximal $1$-connected trees of $\Sigma(\Gamma)$. 

To reconstruct this pruned skeleton for a modular curve, we introduce the notion of a \emph{monodromy labeling} on a finite metric tree $\mathcal{T}$. Let $\mathcal{P}_{g}(G)$ be the set of subgroups of a profinite group $G$. A $G$-monodromy labeling of $\mathcal{T}$ is a function $D: \mathcal{T}\to \mathcal{P}_{g}(G)$ such that for any normal open subgroup $H\subset G$, the induced quotients of the $D_{x}$ are upper-semicontinuous and locally constant on the complement of finitely many points, see \cref{def:MonodromyLabeling2}. 
Each monodromy labeling of a finite metric tree automatically corresponds to a tower of coverings $\mathcal{T}_{H}\to \mathcal{T}$ of finite metric graphs indexed by the open subgroups $H$ of $G$. To reconstruct $\mathcal{T}_{H}$ from the data, one takes the associated double coset spaces $D_{x}\bs G/H$ and glues these over line segments where $D_{x}$ is constant. We note here that the group $D_{x}$ up to conjugacy would only determine the fiberwise behavior of the covering, see \cref{rem:FiberwiseBehaviorCovering} and \cite[Figure 1]{Helminck2023}.

For modular curves, we take the metric tree $\mathcal{T}_{can}$ to be the \emph{canonical supersingular tree} in $X(1)^{\an}=\mathbb{P}^{1,\an}$. This consists of the Gauss vertex $\zeta_{G}$ with respect to the $j$-coordinate, together with an edge for each %
supersingular $j$-invariant over $\overline{\F}_{p}$. The tree for $p=37$ can be found in %
\cref{fig:CanonicalSupersingularTree2}. %
\vspace{-0.8cm}
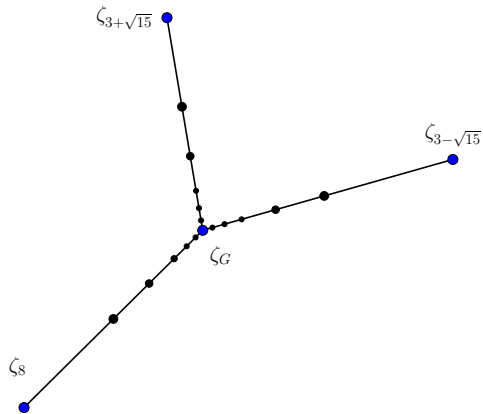
\begin{figure}[h]
\scalebox{0.47}{
\begin{tikzpicture}[line cap=round,line join=round,>=triangle 45,x=1cm,y=1cm]
\clip(-1,-1) rectangle (13,12);
\draw [line width=1.3pt] (5,5)-- (4,11);
\draw [line width=1.3pt] (5,5)-- (12,7);
\draw [line width=1.3pt] (5,5)-- (0,0);
\begin{scriptsize}
\draw [fill=qqqqff] (4,11) circle (4pt);
\draw [fill=qqqqff] (12,7) circle (4pt);
\draw [fill=qqqqff] (0,0) circle (4pt);
\draw [fill=black] (2.5,2.5) circle (3.5pt);
\draw [fill=black] (3.5,3.5) circle (3pt);
\draw [fill=black] (4.2,4.2) circle (2.5pt);
\draw [fill=black] (4.55,4.55) circle (2pt);
\draw [fill=black] (4.8,4.8) circle (2pt);
\draw [fill=black] (4.418761806280903,8.487429162314578) circle (3.5pt);
\draw [fill=black] (4.651257083768542,7.092457497388747) circle (3pt);
\draw [fill=black] (4.814003778009889,6.115977331940666) circle (2pt);
\draw [fill=black] (4.895377125130563,5.627737249216624) circle (2pt);
\draw [fill=black] (4.953500944502473,5.278994332985162) circle (2pt);
\draw [fill=black] (4.94,4.94) circle (0.05pt);
\draw [fill=black] (4.986050283350745,5.083698299895528) circle (0.05pt);
\draw [fill=black] (5.081588012438052,5.023310860696586) circle (0.05pt);
\draw [fill=black] (5.27196004146005,5.0777028689885855) circle (2pt);
\draw [fill=black] (5.611910093285086,5.174831455224311) circle (2pt);
\draw [fill=black] (6.0878401658401415,5.310811475954326) circle (2pt);
\draw [fill=black] (7.039700310950257,5.5827715174143595) circle (3pt);
\draw [fill=black] (8.399500518250425,5.971285862357265) circle (3.5pt);
\draw [fill=qqqqff] (5,5) circle (4pt);
\draw[color=black] (5.5,4.3) node {\Large{$\zeta_{G}$}};
\draw[color=black] (-0.1840040207491881,1.1510481922414615) node {\Large{$\zeta_{8}$}};
\draw[color=black] (2.8,11) node {\Large{$\zeta_{3+\sqrt{15}}$}};
\draw[color=black] (12,7.696978298682645) node {\Large{$\zeta_{3-\sqrt{15}}$}};
\end{scriptsize}
\end{tikzpicture}
}
\caption{\label{fig:CanonicalSupersingularTree2}
The canonical supersingular tree for $p=37$. The three line segments correspond to the supersingular $j$-invariants over $\overline{\F}_{37}$. %
The central point is the Gauss point $\zeta_{G}$, and the other $\zeta_{j}$ %
lie at distance $p/(p+1)$ from $\zeta_{G}$. The lengths of the intermediate line segments are given by the theory of canonical subgroups. %
}
\end{figure}

 The canonical supersingular tree inherits a $\PSL_{2}(\hat{\Z})$-monodromy labeling from the theory of canonical subgroups and the tame structure of the modular tower. %
 We illustrate the induced monodromy labeling for the group $\PSL_{2}(\Z_{p})$ here. We assign the group $\PSL_{2}(\Z_{p})$ to the outside vertices and the Borel subgroup $B$ of upper triangular matrices to the central Gauss point. For the intermediate points, consider the inverse image $I_{n}$ of the subgroup of upper-triangular matrices in $\PSL_{2}(\Z/p^{n}\Z)$ under the map $\PSL_{2}(\Z_{p})\to \PSL_{2}(\Z/p^{n}\Z)$. We monotonically decrease $D_{x}$ along the line segment using the inclusions 
 \begin{equation*}
 \PSL_{2}(\Z_{p})\supset I_{1}\supset I_{2} \supset ... \supset B.
 \end{equation*}
 The jumps are exactly at the vertices in \cref{fig:CanonicalSupersingularTree2}. For $G=\PSL_{2}(\hat{\Z})$, one modifies the $D_{x}$ using the ramification data coming from the tame part of the tower of modular curves.  We show that this monodromy labeling reconstructs the pruned skeleton of any modular curve. 

\vspace{1cm}
\begin{reptheorem}{thm:MainThm3}%
Let $\mathcal{T}_{can}\subset X(1)^{\an}$ be the canonical supersingular tree and let $\phi_{H}: X^{\an}_{H}\to X(1)^{\an}$ be the morphism of modular curves corresponding to %
an open subgroup $H$ of $\PSL_{2}(\hat{\Z})$. Let $\mathcal{T}_{can,H}$ be the metric graph induced by the monodromy labeling in \cref{def:PSL2Labeling}.   %
Then $\phi^{-1}_{H}(\mathcal{T}_{can})\simeq \mathcal{T}_{can,H}$. Moreover, $\mathcal{T}_{can,H}$ deformation retracts onto the pruned skeleton of $X^{\an}_{H}$. 
\end{reptheorem}
To prove this, we first show that the coverings $X(M)^{\an}\to X(1)^{\an}$ for $(M,p)=1$ and $p>3$ are residually tame, so that their structure can easily be deduced from the geometry of modular curves over $\C$. We then consider a covering %
$X(p^{n}M)^{\an}\to X(M)^{\an}$ for $M$ suitably large with $(M,p)=1$. The results in \cite{WS16} give a semistable covering of the local supersingular part of $X(p^{n}M)^{\an}$ in terms of a quotient of the infinite Lubin-Tate tower. We use this description %
to calculate the stabilizers of the points in the corresponding Bruhat-Tits tree over $\SL_{2}(\Z_{p})$, which allows us to pinpoint their images in $X(M)^{\an}$ and $X(1)^{\an}$. By a group-theoretic argument on normalizers of Borel subgroups, these local pictures glue uniquely. We then combine the tame and wild pictures to obtain the final global monodromy labeling. To show that this gives the pruned skeleton, we complete the labeling to a monodromy labeling on a larger tree $\mathcal{T}_{can,mod}$ which recovers the full skeleton. Over the newly attached parts, the corresponding groups are increasing towards the tree $\mathcal{T}_{can}$, so that the inverse images are disjoint unions of trees. This gives the desired statement.  

We apply \cref{thm:MainThm3} to various classical modular curves to obtain explicit formulas for their tropical invariants. For instance, we give a general formula (see \cref{thm:MainThm2v2}) for the toric rank of the Jacobian of a modular curve associated to a \emph{decomposable subgroup} of $\PSL_{2}(\hat{\Z})$, which is a subgroup that can be suitably decomposed into a $p$-part and a prime-to-$p$ part in $\SL_{2}(\hat{\Z})$. This immediately gives a group-theoretic criterion for the potential good reduction of modular abelian varieties, see \cref{pro:PotentialGoodReduction}. %
We also explicitly determine the pruned skeleta for the modular curves $X_{0}(N)$, $X_{1}(N)$, $X_{sp}(N)$ and $X^{+}_{sp}(N)$ %
using coset schemes over $\Z$. The result for $X_{0}(N)$ can be found in \cref{fig:X0(pn)}. We define a suitable ladder-like basis of $H_{1}(\Sigma^{pr}(X_{0}(N)))$ whose associated monodromy matrix is almost tridiagonal. We then use continuants to determine the group-theoretic structure of the prime-to-$2$ part of the component group of $J_{0}(N)$ over the finite extension of $\Q_{p}$ given by Krir's theorem and local class field theory, see \cref{thm:MainThmComponentGroup} for the final result.

\begin{figure}[h]
\scalebox{.32}{
\begin{tikzpicture}[line cap=round,line join=round,>=triangle 45,x=1cm,y=1cm]
\clip(-12.133952880447442,-1) rectangle (37.5356197917536,17.9);
\draw [line width=3pt] (0,2)-- (10,2);
\draw [line width=3pt] (10,2)-- (16,2);
\draw [line width=3pt] (16,2)-- (22,2);
\draw [line width=3pt] (22,2)-- (28,2);
\draw [line width=1.2pt,loosely dashed] (28,2)-- (34,2);
\draw [shift={(0.6016091062078582,18.02630510355656)},line width=3pt]  plot[domain=3.9898696955643285:4.674867745634293,variable=\t]({1*16.032034496116204*cos(\t r)+0*16.032034496116204*sin(\t r)},{0*16.032034496116204*cos(\t r)+1*16.032034496116204*sin(\t r)});
\draw [line width=3pt] (10,6)-- (16,6);
\draw [line width=3pt] (16,6)-- (22,6);
\draw [line width=3pt] (22,6)-- (28,6);
\draw [line width=3pt] (10,10)-- (16,10);
\draw [line width=3pt] (16,10)-- (22,10);
\draw [line width=3pt] (22,10)-- (28,10);
\draw [line width=1.2pt,loosely dashed] (28,10)-- (34,10);
\draw [line width=1.2pt,loosely dashed] (28,6)-- (34,6);
\draw [line width=3pt] (22,12)-- (28,12);
\draw [line width=1.2pt,loosely dashed] (28,12)-- (34,12);
\draw [line width=1.2pt,loosely dashed] (28,14)-- (34,14);
\draw [line width=1.2pt,loosely dashed] (28,16)-- (34,16);
\draw [line width=3pt] (22,16)-- (28,16);
\draw [line width=3pt] (22,14)-- (28,14);
\draw [line width=3pt] (16,14)-- (22,14);
\draw [line width=3pt] (16,12)-- (22,12);
\draw [line width=3pt] (16,16)-- (22,16);
\draw [line width=3pt] (22,17)-- (28,17);
\draw [line width=3pt] (22,15)-- (28,15);
\draw [line width=1.1pt,loosely dashed] (28,17)-- (34,17);
\draw [line width=1.5pt,loosely dashed] (28,15)-- (34,15);
\draw [line width=1.1pt,loosely dashed] (28,17.5)-- (34,17.5);
\draw [line width=1.1pt,loosely dashed] (28,16.5)-- (34,16.5);
\draw [shift={(0.48131731964547825,-6.010331488915593)},line width=3pt]  plot[domain=1.6008501952631267:2.2883175079627254,variable=\t]({1*16.017564757076926*cos(\t r)+0*16.017564757076926*sin(\t r)},{0*16.017564757076926*cos(\t r)+1*16.017564757076926*sin(\t r)});
\draw [shift={(10.487486772902743,21.67140675901655)},line width=3pt]  plot[domain=3.9803687242420236:4.681292242492923,variable=\t]({1*15.691052053518503*cos(\t r)+0*15.691052053518503*sin(\t r)},{0*15.691052053518503*cos(\t r)+1*15.691052053518503*sin(\t r)});
\draw [shift={(10.607321137612411,-2.2954661829160425)},line width=3pt]  plot[domain=1.608048417320959:2.282618392375518,variable=\t]({1*16.306779482250665*cos(\t r)+0*16.306779482250665*sin(\t r)},{0*16.306779482250665*cos(\t r)+1*16.306779482250665*sin(\t r)});
\draw [shift={(15.580447273063626,7.351200176211828)},line width=3pt]  plot[domain=1.522324388461345:2.26905526688955,variable=\t]({1*8.658970082096259*cos(\t r)+0*8.658970082096259*sin(\t r)},{0*8.658970082096259*cos(\t r)+1*8.658970082096259*sin(\t r)});
\draw [shift={(15.880033184837796,21.67140675901655)},line width=3pt]  plot[domain=4.058419452030125:4.724792621941399,variable=\t]({1*9.665674933344205*cos(\t r)+0*9.665674933344205*sin(\t r)},{0*9.665674933344205*cos(\t r)+1*9.665674933344205*sin(\t r)});
\draw [line width=3pt] (10,14)-- (16,14);
\draw [shift={(21.74588908011716,31.6832529242532)},line width=3pt]  plot[domain=4.361204836827085:4.72761929983459,variable=\t]({1*16.70268432339862*cos(\t r)+0*16.70268432339862*sin(\t r)},{0*16.70268432339862*cos(\t r)+1*16.70268432339862*sin(\t r)});
\draw [shift={(21.871751420321186,-0.7376194416904277)},line width=3pt]  plot[domain=1.5635661375728356:1.9081940114385858,variable=\t]({1*17.73808307445942*cos(\t r)+0*17.73808307445942*sin(\t r)},{0*17.73808307445942*cos(\t r)+1*17.73808307445942*sin(\t r)});
\draw [line width=3pt] (0,10)-- (10,10);
\draw [shift={(27.84704,-17.34934)},line width=3pt]  plot[domain=1.566407175730383:1.7394029943393858,variable=\t]({1*34.84967568281232*cos(\t r)+0*34.84967568281232*sin(\t r)},{0*34.84967568281232*cos(\t r)+1*34.84967568281232*sin(\t r)});
\draw [shift={(27.970544440327117,49.923046841223986)},line width=3pt]  plot[domain=4.5329901913130985:4.71327027502509,variable=\t]({1*33.46004205352034*cos(\t r)+0*33.46004205352034*sin(\t r)},{0*33.46004205352034*cos(\t r)+1*33.46004205352034*sin(\t r)});
\draw [line width=1.5pt,loosely dashed] (28,17.5)-- (34,17.8);
\draw [line width=1.5pt,loosely dashed] (28,17.5)-- (34,17.2);
\begin{scriptsize}
\draw [fill=ududff] (0,2) circle (8pt);
\draw [fill=ududff] (10,2) circle (8pt);
\draw [fill=ududff] (16,2) circle (8pt);
\draw [fill=ududff] (22,2) circle (6.5pt);
\draw [fill=ududff] (28,2) circle (5.5pt);
\draw [fill=ududff] (-10,6) circle (10pt);
\draw [fill=ududff] (0,10) circle (8pt);
\draw [fill=ududff] (10,10) circle (8pt);
\draw [fill=ududff] (10,6) circle (8pt);
\draw [fill=ududff] (10,14) circle (8pt);
\draw [fill=ududff] (16,10) circle (8pt);
\draw [fill=ududff] (16,6) circle (8pt);
\draw [fill=ududff] (22,10) circle (6.5pt);
\draw [fill=ududff] (22,6) circle (6.5pt);
\draw [fill=ududff] (28,10) circle (5.5pt);
\draw [fill=ududff] (28,6) circle (5.5pt);
\draw [fill=ududff] (16,14) circle (8pt);
\draw [fill=ududff] (22,14) circle (6.5pt);
\draw [fill=ududff] (28,14) circle (5.5pt);
\draw [fill=ududff] (16,12) circle (8pt);
\draw [fill=ududff] (22,12) circle (6.5pt);
\draw [fill=ududff] (28,12) circle (5.5pt);
\draw [fill=ududff] (16,16) circle (8pt);
\draw [fill=ududff] (22,16) circle (6.5pt);
\draw [fill=ududff] (28,16) circle (5.5pt);
\draw [fill=ududff] (22,15) circle (6.5pt);
\draw [fill=ududff] (22,17) circle (6.5pt);
\draw [fill=ududff] (28,17) circle (5.5pt);
\draw [fill=ududff] (28,15) circle (5.5pt);
\draw [fill=ududff] (28,16.5) circle (5.5pt);
\draw [fill=ududff] (28,17.5) circle (5.5pt);
\draw [fill=ududff] (21.74588908011716,31.6832529242532) circle (6.5pt);
\draw [fill=ududff] (27.84704,-17.34934) circle (6.5pt);
\draw[color=ududff] (8.079113128531233,-7.682178973904463) node {$B_2$};
\draw [fill=ududff] (27.970544440327117,49.923046841223986) circle (6.5pt);
\draw [fill=ududff] (32.75172,53.55781) circle (6.5pt);
\draw[color=ududff] (10.195081471068946,23.890296031854913) node {$F_2$};
\draw [fill=ududff] (32.71283,-17) circle (6.5pt);
\draw[color=ududff] (11.253065642337802,-7.904912483645271) node {$G_2$};
\end{scriptsize}
\end{tikzpicture}
}
\caption{\label{fig:X0(pn)}The local picture of the pruned skeleton of $X_{0}(p^{n})$ over $\mathbb{C}_{p}$, which is a quotient of the Bruhat-Tits tree over $\mathbb{Q}_{p}$. If we remove the dashed lines, then this picture gives the local pruned skeleton of $X_{0}(p^{5})$. The global skeleton is obtained by gluing $|\mathcal{S}|$ copies of this graph at the endpoints, where $|\mathcal{S}|$ is the number of supersingular $j$-invariants over $\overline{\F}_{p}$. %
By further truncating this graph, one finds the graphs in \cite{Tsushima2015}, \cite{CM2010}, \cite{Edixhoven90} and finally \cite{DR73}.  }%
\end{figure}
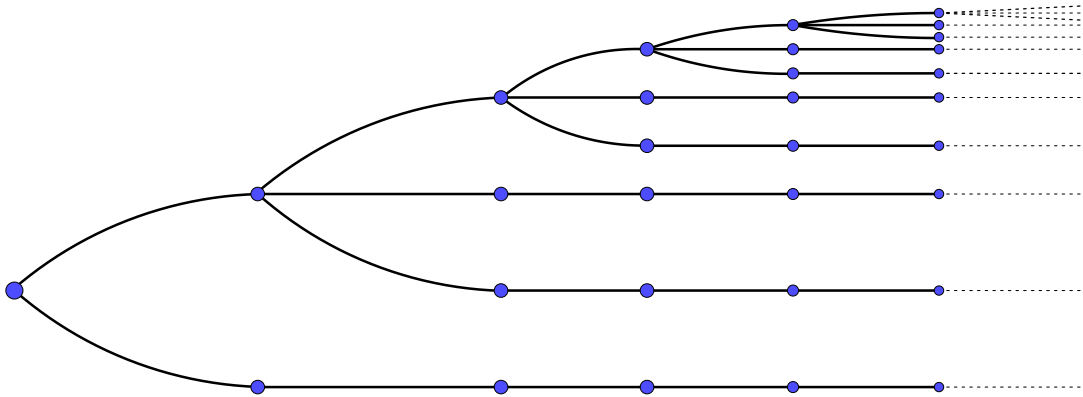

\subsection{Connections to the literature}

We provide some additional context in the form of a brief summary of related results. %
Our main results heavily rely on the results in \cite{WS16}, where semistable models for the modular curves $X(p^{n}N)$ are produced in terms of the infinite Lubin-Tate tower. More specifically, one first constructs a model for subspace of non-CM-points in the infinite Lubin-Tate tower, and one uses the explicit form of the local Langlands correspondence in terms of types for $\GL_{2}(\Q_{p})$ (see \cite{BH06}, \cite{WS10}, \cite{Strauch2008}, \cite{Carayol1990}) %
 to show that this exhausts the cohomology of the modular curves at a finite level, so that the induced model is automatically semistable. In \cite[Section 7]{WS16}, a sketch is given of how one might obtain the dual intersection graph of such a semistable model: one takes the graph defined in \cite{WS16}, one then quotients this graph by the corresponding subgroup (up to some small modifications), and this gives the desired dual intersection graph. It seems that this process has not led to any explicit graphs however in the literature, see the introduction of \cite{EP21} for instance. The current paper can be seen as a first step towards explicitly determining the various semistable invariants using the techniques in \cite{WS16}.   

A paper that is related to ours in certain aspects is \cite{Fargues2007}. There, a {quotient} of the infinite Lubin-Tate tower is shown to contain a copy of the Bruhat-Tits tree that maps to our canonical tree. It might be possible to directly obtain the decomposition groups and section from this result, but the author wasn't able to fully complete the argument due to the interference of the quotient. On the other hand, in this paper we obtain new proofs of these results and we give somewhat stronger statements: %
 this graph gives the pruned skeleton, and we can pinpoint the points in the tree where the non-trivial parts of the rest of the full skeleton attach. Moreover, we give a group-theoretic argument to show that the local parts glue, and we show how to deal with the auxiliary level structure in a purely group-theoretic way. 

Let us briefly mention the missing parts in our description of skeleta of modular curves: %
\begin{itemize}
\item The edge lengths for the outer maximal $1$-connected trees.
\item An extension over which the modular curve attains semistable reduction (for $X_{0}(N)$, we can use a theorem by Krir).
\item Explicit algebraic equations for the residue curves (which in principle can be obtained from the equations in \cite{WS16} by taking suitable quotients.)   
\end{itemize}
We discuss some of these in further detail in  \cref{sec:FutureDirections}.  
 
 \subsection{Short outline of the paper}\label{sec:Outline}
 
 We start in \cref{sec:BerkovichSkeleta} by defining the notion of a pruned skeleton. In \cref{sec:MonodromyLabelingSection}, we define monodromy labelings and in \cref{sec:SimplicialStructureCoverings} we show that they recover the local structure of a morphism of Berkovich analytifications of curves. 
 
 In \cref{sec:CoveringsModularCurves}, we prove \cref{thm:MainThm3}. We start by setting up our notation in Sections  \ref{sec:GroupNotation} and \ref{sec:ModularTower}. We then prove that the tame part of the tower of modular curves is residually tame in \cref{sec:TameStructure}. We review elliptic curves over arbitrary valued fields and define the canonical supersingular tree in \cref{sec:EllipticCurvesValuedFields}. In \cref{sec:pAdicTower} we use the residual tameness of the tame part and the results in \cite{KM85} to find the correct decomposition groups over the central vertex of the canonical supersingular tree. The rest of \cref{sec:CoveringsModularCurves} is then concerned with the tower over the remaining parts of the supersingular tree.  In \cref{sec:ReviewWS}, we give a review of the results in \cite{WS16}. We then use these in \ref{sec:WildTower} and \ref{sec:PrunedSkeleton} to prove the main theorem. We conclude in \cref{sec:PotentialGoodReduction} with a formula for the toric rank when the subgroup is decomposable, and we use this to prove the potential good reduction of various Jacobians.
 
\cref{sec:CosetSchemes} focuses on applying these techniques to the modular curves $X_{0}(N)$, $X_{1}(N)$, $X_{sp}(N)$ and $X^{+}_{sp}(N)$. For each of the corresponding congruence subgroups $\Gamma$, we define a coset scheme $\mathcal{F}_{\Gamma}$ over $\Z$ which classifies the left-coset spaces associated to $\Gamma$, see \cref{sec:QuotientFunctors}. We then determine the double coset spaces associated to $\Gamma$ and $\Gamma_{0}$ in \cref{sec:DCSCentralVertex}, which determines the pruned skeleton by the considerations in \cref{sec:DCSOuterParts}. In \cref{sec:CompX0N}, we determine the structure of the component groups for $X_{0}(N)$ over Krir's extension. We conclude the paper with a short list of future directions in \cref{sec:FutureDirections}.  
\vspace{-1cm}
\tableofcontents

\subsection{Acknowledgements}

The author would like to thank Kazuhiko Yamaki and Joseph Rabinoff for their suggestions and comments on monodromy labelings. The author would also like to thank Yuji Odaka, Omid Amini, Matthew Baker, Akio Tamagawa and Alexander Betts for helpful and enlightening discussions. The author was supported by UK Research and
Innovation under the Future Leaders Fellowship programme MR/S034463/2  %
as a Postdoctoral Fellow at Durham University, and the JSPS Postdoctoral Fellowship with ID No. 23769 and KAKENHI 23KF0187 as a Postdoctoral Fellow at Tohoku University and the University of Tsukuba. There is no additional computational data associated to this paper.

 \subsection{Notation}\label{sec:Notation}
 
 We will generally use the following notation in this paper:
 \begin{itemize}
 \item The greatest common divisor of two integers $N$ and $M$ is $(N,M)$. The Euler totient function is denoted by $\phi(N):=|(\mathbb{Z}/N\Z)^{*}|$. 
 \item $K$ will denote a non-archimedean field with valuation $v_{K}:K\to \mathbb{R}\cup\{\infty\}$, valuation ring $\mathcal{O}_{K}$, residue field $k$ and uniformizer $\pi_{K}$. 
 \item $\C_{p}$ will denote the completion of an algebraic closure of $\Q_{p}$. We write $v$ for the valuation on $\C_{p}$ with %
 $v(p)=1$. 
 \item Curves $X$ over a field $L$ are smooth, proper and geometrically irreducible, unless mentioned otherwise. Their function fields are denoted by $L(X)$.
 \item We call a finite separable dominant morphism of normal connected Noetherian schemes $X\to{Y}$ a covering. Here being separable means that the induced extension of function fields is separable.    
 \item The Jacobian of a curve $X$ is denoted by $J=J(X)$. If $X$ is defined over a local field $K$, we write $\mathcal{J}$ for the N\'{e}ron model of $J$ over $\mathcal{O}_{K}$.
 \item Finite metric graphs as in \cite[Section 2.1]{ABBR2015} are denoted by $\Sigma$. Unless mentioned otherwise, these are connected and compact. 
 \item Edges of finite metric graphs are subsets homeomorphic to intervals in $\mathbb{R}\cup\{\infty\}$. Here intervals are allowed to be open and closed on either side. The type of interval will usually be clear from context. We write $\ell(e)$ for the length of an edge $e\subset \Sigma$. 
 \end{itemize}

\section{Preliminaries}\label{sec:Preliminaries}

In this section, we provide some of the preliminary considerations needed to find skeleta of modular curves. We start by defining the pruned skeleton in \cref{sec:BerkovichSkeleta}, whose metric structure determines the tropical Jacobian of a curve. In \cref{sec:MonodromyLabelingSection} we define the notion of a monodromy labeling, and in \cref{sec:SimplicialStructureCoverings} we give a general Galois-theoretic reconstruction algorithm for coverings of metric trees in terms of these monodromy labelings. This tool will allow us to study coverings of graphs in terms of varying double coset spaces.         

\subsection{Pruned skeleta}\label{sec:BerkovichSkeleta}

Let $\Sigma$ be a finite connected metric graph as in \cite[Section 2.1]{ABBR2015} with $\beta_{1}(\Sigma)\neq{0}$, and let $\Sigma'\subset \Sigma$ be a connected closed subspace that is a tree (i.e., whose first Betti number is zero). If $\Sigma'$ meets the closure of its complement in only one point, then we say that $\Sigma'$ is $1$-connected to $\Sigma$. If the two $1$-connected trees intersect, then their union is again a tree, and it is again $1$-connected. We can thus consider maximal $1$-connected trees. By successively retracting all maximal $1$-connected trees, we obtain a finite metric graph $\Sigma^{pr}$. We call this the pruned skeleton of $\Sigma$.   %

\begin{definition}
Let $X$ be a proper smooth and connected algebraic curve over $K$ with $g(X)\geq{1}$ and minimal Berkovich skeleton $\Sigma(X)$. If $\Sigma(X)$ is a tree, then we define $\Sigma^{pr}(X)$ to be any point in $\Sigma(X)$. Otherwise, %
the pruned skeleton $\Sigma^{pr}(X)$ is the metric graph obtained by retracting all maximal $1$-connected trees.  %
\end{definition}

\begin{example}
A leaf is a $1$-connected tree that consists of a single line segment. 
Note that a leaf is part of the minimal Berkovich skeleton if and only if the attached point of valence one has a strictly positive weight, so that the corresponding residue curve has strictly positive genus. These outer leaves with positive genus are removed in the pruned skeleton.  
\end{example}
We now recall how one can obtain the group of connected components of the N\'{e}ron model of the Jacobian of $X$ from the pruned skeleton. Let $K$ be a local field with embedding $\mathbb{Q}_{p}\subset K\subset \mathbb{C}_{p}$, and assume that $X$ admits a semistable model over ${K}$. We write $v_{K}:K\to \mathbb{R}\cup\{\infty\}$ for the normalized valuation with $v(\pi_{K})=1$, where $\pi_{K}$ is a uniformizer of $\mathcal{O}_{K}$. We also use this normalization on $\C_{p}$ %
in this section. %
We write $\Sigma(X)(\Lambda_{K})$ for the finite graph induced by the $\Lambda_{K}$-valued points of $\Sigma(X)$, where $\Lambda_{K}=\Z\subset \R$ is the value group of $K$. 

Let $\mathcal{J}/\mathcal{O}_{K}$ be the N\'{e}ron model of the Jacobian $J(X)/K$, and let $\mathcal{J}^{0}_{s}$ be the identity component of the special fiber $\mathcal{J}_{s}$ of $\mathcal{J}$. The quotient scheme $\mathcal{J}_{s}/\mathcal{J}^{0}_{s}$ is finite \'{e}tale over the residue field $k$ of $K$; we denote it by $\Psi_{K}$ if $X$ is understood from context. 
\begin{definition}
The group $\Psi_{K}(\overline{\F}_{p})$ is the component group or tropical Jacobian of $X$ over $K$.  
\end{definition}
Consider a $\mathbb{Z}$-basis $\gamma_{i}$ of the first homology group of $\Sigma(X)(\Lambda_{K})$. The monodromy matrix $A=(a_{i,j})$ is given by 
\begin{equation*}
a_{i,j}=\langle \gamma_{i}, \gamma_{j}\rangle,
\end{equation*}%
where $\langle \cdot , \cdot \rangle$ denotes the normalized length pairing over $K$. We then have that 
\begin{equation*}
\Psi_{K}(\overline{\mathbb{F}}_{p})\simeq \mathbb{Z}^{n}/\mathrm{im}(A),
\end{equation*}
see \cite[Lemma 2.22]{DDMM2023}. %
Alternatively, one can also define this group in terms of a divisor class group on the graph, see \cite{Baker2008}.  %

\begin{lemma}
The pruned minimal Berkovich skeleton $\Sigma^{pr}(X)$ completely determines the structure of $\Psi_{K}(\overline{\F}_{p})$. %
\end{lemma}
\begin{proof}
This follows since we can ignore $1$-connected trees in our choice of the $\gamma_{i}$. 
\end{proof}

\begin{remark}
Suppose that $X$ is a curve of genus $\geq{1}$ over a local field $K$ with $\Q_{p}\subset K\subset \C_{p}$ and let $\Sigma^{pr}(X)$ be the pruned minimal skeleton, with edge lengths induced by the normalized valuation $v(p)=1$ on $\C_{p}$. Assume for simplicity that $\Sigma^{pr}(X)$ is not a cycle. %
Then %
$\Sigma^{pr}(X)$ %
has a unique minimal semistable vertex set $V(\Sigma^{pr}(X))$ with edge set $E(\Sigma^{pr}(X))$. More explicitly, $V(\Sigma^{pr}(X))$ is the set of all vertices of valence  %
greater than or equal to $3$, and $E(\Sigma^{pr}(X))$ is the induced set of open line segments. We can then consider the smallest $m\in \mathbb{N}$ such that $m\ell(e)\in\mathbb{Z}$ for all $e\in E(\Sigma^{pr}(X))$. %
Let $L\supset K$ be the minimal extension over which $X$ attains semistable reduction and define $m_{min}=[L^{unr}:\Q_{p}^{unr}]$ to be the corresponding ramification degree.   
Note that $m|m_{min}$. %
We can define a virtual monodromy matrix by taking the normalization of the valuation with respect to $m$; we define the virtual component group to be the associated quotient. 
 Note that the virtual component group gives the component group over any extension over which semistable reduction is attained.

Using the results in this paper, we can find the virtual component group of any modular curve for $p\neq{2,3}$. To obtain the actual component groups, one needs to know the ramification degree over $\Q_{p}$ of an extension over which $X$ attains semistable reduction.    
\end{remark}

\begin{remark}
For $X=X_{0}(N)$, the results in \cite{Krir96} give an extension of $K:=\mathbb{Q}_{p}$  over which $J_{0}(N)=J(X)$ attains semistable reduction. %
Namely, let $K=\mathbb{Q}_{p}(\sqrt{-p},\sqrt{-Dp})$ for $D$ a non quadratic residue and let $n=v_{p}(N)$. Let $L_{n}$ be the field associated to the subgroup 
\begin{equation*}
U_{n,\pm}=\{\alpha \in \mathcal{O}^{*}_{K}:\alpha^2\in 1+\sqrt{p^{n-1}}\mathcal{O}_{K}\}
\end{equation*}
by local class field theory. 
Then $J_{0}(N)$ has semistable reduction over $L_{n}$. %
We recall here from \cite[Th\'{e}or\`{e}me 2]{Krir96} that 
\begin{equation*}
[L^{\unr}_{n}:\mathbb{Q}^{\unr}_{p}]=p^{2(n-2)}(p^{2}-1)
\end{equation*}
for $n\geq{2}$. 
For $n=2$, the %
minimal extension $L_{\mathrm{min}}$ of $\Q^{\mathrm{unr}}_{p}$ over which $X_{0}(p^{2})$ attains semistable reduction is  %
the unique tamely ramified extension of degree $(p^{2}-1)/2$ over $\Q^{\un}_{p}$ by \cite{Edixhoven90}, so Krir's extension is not necessarily the minimal one. Our results suggest that the minimal extension is much smaller in general, see \cref{thm:MainThmComponentGroup}. 
\end{remark}

\subsection{Monodromy labelings}\label{sec:MonodromyLabelingSection}

In this section we introduce the notion of a monodromy labeling on a finite metric tree, which gives a tower of coverings of the metric tree in terms of glued double coset spaces. In the next section, we will see that this notion locally reconstructs the Berkovich analytification of any finite separable morphism of smooth curves, and any profinite tower of these coverings.  

\begin{definition}\label{def:MonodromyLabeling2}(Monodromy labeling)
Let $G$ be a finite group and let $\mathcal{P}_{g}(G)$ be the set of subgroups of $G$. A $G$-monodromy labeling on a finite metric tree $\mathcal{T}$ is a function
\begin{equation*}
D:\mathcal{T}\to \mathcal{P}_{g}(G)
\end{equation*}
with the property that %
 for every $x\in\mathcal{T}$, there is an open neighborhood $U_{x}$ of $x$ such that
\begin{enumerate}
\item $D$ is constant on every connected component of $U_{x}\bs \{x\}$,
\item $D$ is increasing towards $x$, in the sense that %
\begin{equation*}
D_{y}\subset D_{x}%
\end{equation*}
for $y\in U_{x}$. 
\end{enumerate}
We refer to the $D_{x}$ as the local monodromy or decomposition groups. 
If $G$ is a profinite group, then a $G$-monodromy labeling on $\mathcal{T}$ is a function $D:\mathcal{T}\to \mathcal{P}_{g}(G)$ such that for every open normal subgroup $H\subset G$, the function 
\begin{equation*}
D_{H,x}:=\pi_{H}(D_{x})
\end{equation*}
defines a $G/H$-monodromy labeling as above. Here $\pi_{H}:G\to G/H$ is the quotient map. 
\end{definition}

\begin{remark}
Let $D$ be a monodromy labeling for a finite group $G$. Then there is a finite set of points $S\subset \mathcal{T}$ such that the induced function on $\mathcal{T}\bs {S}$ is constant on connected components. Indeed, this follows since $\mathcal{T}$ is compact. For infinite profinite groups however, the local monodromy groups $D_{x}$ can accumulate around a point, as is the case for modular curves, see \cref{def:PSL2Labeling}. %
\end{remark}

\begin{definition}\label{def:MonodromyTower}(Tower of metric graphs)
Let $G$ be a finite group and let 
\begin{equation*}
\mathcal{T}'_{G}=\bigsqcup_{g\in{G}} \mathcal{T}_{g}.
\end{equation*} 
Write $(x,g)\sim (x',g')$ if $x=x'$ and $gD_{x}=g'D_{x}$. This defines an equivalence relation on $\mathcal{T}'_{G}$. We define $\mathcal{T}_{G}$ to be the quotient by this equivalence relation. The left $G$-action on $G/D_{x}$ induces a left $G$-action on $\mathcal{T}_{G}$. %
For a subgroup $H\subset G$, we define $\mathcal{T}_{H}$ to be the quotient by this action. The points lying over a fixed $x$ in the quotient are given by the $H$-orbits of the cosets $gD_{x}$.    %

We can endow $\mathcal{T}_{G}$ and $\mathcal{T}_{H}$ with the structure of a finite metric graph as follows. Let $e\subset \mathcal{T}$ be an open edge over which $D$ is constant. The inverse image of this edge is then a disjoint union of $G/D_{x}$ copies of $e$. We define the edge length of any of these edges to be $\ell(e)/|D_{x}|$. If $e$ has infinite length, then we do this locally on the edge. To glue these edges, let $x$ be an endpoint of an open edge $e$. By definition, we have $D_{x}\supset D_{y}$ for $y\in e$.
We then glue the edge corresponding to $(y,g)$ to the point given by $(x,g)$ for $g\in{G}$. This is well defined by the inclusion of groups $D_{x}\supset D_{y}$. Similarly, if $(y,g)$ is a representative for an edge $e$ in $\mathcal{T}_{G}$, then an edge in $\mathcal{T}_{H}$ is given by the $H$-orbit of $(y,g)$. The stabilizer of $e$ in $H$ is $gD_{y}g^{-1}\cap H$. We assign the length $\ell(e)/|gD_{y}g^{-1}\cap H|$ to the edge $e\subset \mathcal{T}_{G}$ corresponding to $(y,g)$. 

   For $G$ a profinite group and $H$ an open normal subgroup, we define $\mathcal{T}_{H}$ to be the finite metric graph associated to $G/H$. If $H_{1}\subset H_{2}$ are two open normal subgroups, then $\mathcal{T}_{H_{1}}/\pi_{H_{1}}(H_{2})\simeq \mathcal{T}_{H_{2}}$. For any open subgroup $H_{0}$, consider an open normal subgroup $H\subset H_{0}$\footnote{For instance, the kernel of $G\to \mathrm{Aut}(G/H_{0})$ induced by multiplication on the left.}. We define $\mathcal{T}_{H_{0}}$ to be $\mathcal{T}_{H}/\pi_{H}(H_{0})$. This is independent of the open normal subgroup $H$.  %
 From this, we obtain a tower of finite metric graphs $\mathcal{T}_{H}$ with maps $\mathcal{T}_{H_{1}}\to \mathcal{T}_{H_{2}}$ if $H_{1}\subset H_{1}$. These form a directed system of finite metric graphs $\mathcal{T}_{H}$ for $H$ ranging over the set of open subgroups of $G$.  
\end{definition}

\begin{remark}\label{rem:EdgeLengths}\label{rem:DoubleCosetSpacesEquivalent}
Rather than using the quotient of a quotient construction, we can also use double coset spaces to reconstruct $\mathcal{T}_{H}$. %
Namely, let $D_{x}\bs G/H$ be the double coset space associated to %
$D_{x}$ and $H$ in $G$. 
The bijections $D_{x}\bs G/H\simeq H\bs G/D_{x}\simeq H\bs (G/D_{x})$ show that this space gives the fiber of $\mathcal{T}_{H}\to \mathcal{T}$ over $x\in\mathcal{T}$. Moreover, let $gH$ be a representative of an edge $e_{H}$ lying over $e\subset \mathcal{T}$. One then easily sees that the edge length is $\ell(e)/|\mathrm{Orb}_{D_{x}}(gH)|$. Finding the edge lengths thus reduces to finding the %
order of the orbit of the left coset $gH$ under the left action of $D_{x}$. We will use this language of double cosets for reconstructing the skeleton of a modular curve associated to a open subgroup of $\PSL_{2}(\hat{\Z})$.    %
\end{remark}

\begin{remark}\label{rem:FiberwiseBehaviorCovering}
 To know the individual fibers, it suffices to know the decomposition groups up to conjugacy. For the  reconstruction algorithm above however, it is crucial that we know the exact decomposition groups, as these can give rise to  %
 non-trivial twists. For example, the three different coverings in \cite[Figure 1]{Helminck2023} arise from $S_{3}$-labelings whose groups on the vertices are different $2$-cycles. These coverings are locally isomorphic, but not globally. Hence it does not suffice to simply know the conjugacy class of the decomposition group. Even if the group is abelian, it is not enough to know the order of the decomposition group, as we can have non-trivial twists arising from different choices of subgroups (one can for instance construct examples using $\Z/2\Z\times \Z/2\Z$). For 
cyclic abelian groups the orders of the decomposition groups are enough to reconstruct the skeleton, and the corresponding reconstruction algorithm can be found in
 \cite[Lemma 3.9]{H2022}. 
 For modular curves, the main reason that there are no non-trivial twists comes from the fact that the Borel subgroup $\Gamma_{0}(\Z/p^{n}\Z)$ is self-normalizing, see \cref{lem:NormalizerBorel}. 
\end{remark}

\subsection{Simplicial structure of coverings}\label{sec:SimplicialStructureCoverings}

We now show how monodromy labelings arise naturally from coverings of schemes or analytic spaces. 
Let $\phi: X'\to{X}$ be a finite separable dominant map of normal connected Noetherian\footnote{One can relax the Noetherianity assumption in many cases using standard techniques.} schemes with Galois closure $\overline{\phi}:\overline{X}\to{X'}$ and Galois group $G$. We write $H$ for the subgroup corresponding to the covering $X'\to{X}$.  Let $\overline{x}\in\overline{X}$ be a point lying over $x\in{X}$. The decomposition group $D_{\overline{x}/x}$ associated to $\overline{x}$ and $x$ is  the stabilizer of $\overline{x}$ in $G$. We will also write $D_{x}$ for this group if $\overline{x}$ is understood. Similarly, if $X$ is a scheme of finite type over a non-archimedean field $K$ and $\overline{x}\in{\overline{X}^{\an}}$ with image $x\in{X^{\an}}$, then $D_{\overline{x}/x}$ is the stabilizer of $\overline{x}$ under the action of $G$.  
\begin{lemma}\label{lem:DoubleCosetSpaces}
Let $\phi:X'\to{X}$ be a finite separable dominant morphism of Noetherian normal schemes and let $x\in{X}$. 
There is a bijection %
\begin{equation*}
D_{x}\backslash G/H=D_{x}\bs(G/H)\to \phi^{-1}(x). 
\end{equation*}
Explicitly, we send a left-coset $\tau H$ to $\overline{\phi}(\tau^{-1}(\overline{x}))$. If $X$ is furthermore a scheme of finite type over a non-archimedean field $K$, then the same holds for the Berkovich analytification of $\phi$.
\end{lemma}
\begin{proof}
We first note that the fiber $\overline{\phi}^{-1}(x)$ can be identified with the left coset space $G/D_{x}$, since the action of $G$ is transitive on the fiber. The quotient of this fiber by $H$ is $\phi^{-1}(x)$. Group-theoretically, this quotient is exactly the orbit space $H\bs (G/D_{x})=H\bs G/D_{x}$. Here $H$ acts by left-multiplication on the set of left cosets %
$G/D_{x}$. By %
reversing the order of the double cosets (note that this inverts a representative), we then obtain the desired bijection.

To obtain the same result for Berkovich analytifications, let $L=\mathcal{H}(x)$ be the completed residue field of a point $x$ in $X^{\an}$. The formation of quotients commutes with flat base change, so we have $\overline{X}_{L}/G=X_{L}$. Since the valuation on $L$ extends uniquely to a valuation on $\overline{L}$, we can now %
use the scheme-theoretic result on closed points. 
\end{proof}

\begin{remark}
Note that the induced bijection $H\bs G/D_{x}\to \phi^{-1}(x)$ is simply given by $\tau D_{x}\mapsto \overline{\phi}(\tau(\overline{x}))$, which recovers the bijection in \cite[Section 1.9, page 55]{Neukirch1999}. %
\end{remark}

Let $\phi:X'\to{X}$ be a separable finite morphism of proper smooth curves over a non-archimedean field $K$ (which we will refer to as a covering of curves from now on) with Berkovich analytification $\phi^{\an}:X'^{\an}\to{X^{\an}}$.
We can describe the structure of this covering using the simultaneous semistable reduction theorem, see \cite[Theorem 5.22]{ABBR2015}. This says that we can find a semistable vertex $V(\Sigma)$ of %
$X^{\an}$ with skeleton $\Sigma$ such that the inverse image of $V(\Sigma)$ in $X'^{\an}$ is again a semistable vertex set $V(\Sigma')$ with skeleton $\Sigma'$. If $\Sigma$ is loopless (which we can assume by subdividing), then 
this implies that the inverse images of the induced open annuli corresponding to open edges in $\Sigma$ are disjoint unions of open annuli.  
Note that the theorem also allows us to expand $\Sigma$ to a larger skeleton of $X^{\an}$. For any finite metric tree $\mathcal{T}$, we can thus assume that $\mathcal{T}\subset \Sigma$. %
\begin{lemma}\label{lem:SectionMetricTree}
Let $\phi:X'\to X$ be a covering of curves as above, and let $\mathcal{T}\subset X^{\an}$ be a finite metric tree. Then there is a finite metric tree $\mathcal{T}'\subset X'^{\an}$ that maps homeomorphically onto $\mathcal{T}$ under $\phi^{\an}$. %
\end{lemma} 
\begin{proof}
This easily follows from the above structure theorems (we note that easier proofs are possible here as well).
\end{proof}
\begin{definition}
We call a graph as in \cref{lem:SectionMetricTree} a topological section of $\mathcal{T}$ with respect to $\phi$. 
\end{definition}
\begin{lemma}\label{lem:InducedMonodromyLabeling}
Let $\overline{\phi}:\overline{X}\to {X}$ be the Galois closure of $\phi$ with Galois group $G$ and let $\overline{\mathcal{T}}$ be a topological section of $\mathcal{T}$ with respect to $\overline{\phi}$. Set $D_{x}:=D_{\overline{x}/x}$, where $\overline{x}$ is the unique point in $\overline{\mathcal{T}}$ lying over $x$. Then $D$ defines a $G$-monodromy labeling.  
\end{lemma}
\begin{proof}
We use the simultaneous semistable reduction theorem as above to find a suitable semistable vertex set $V(\Sigma)$ with $\Sigma\supset \mathcal{T}$ for $\overline{\phi}$. The decomposition groups over the induced open edges are automatically constant. %
To see the inclusions, we only have to note that if an edge is stabilized by $\sigma\in{G}$, then the adjacent vertices are also automatically stabilized by $\sigma$ since the edge lies in the inverse image of an open edge $e\subset \Sigma$ that is stable under the action of $G$. 
\end{proof}
The above now gives the following local Galois-theoretic reconstruction algorithm for coverings of curves. 
 \begin{theorem}\label{pro:DCStoGraphs}
 Let $\phi: X'\to X$ be a finite separable covering of curves and let $\mathcal{T}\subset X^{\an}$ be a finite metric tree. Let $\overline{\mathcal{T}}\subset \overline{X}^{\an}$ be a topological section of $\mathcal{T}$ in the Berkovich analytification of the Galois closure $\overline{\phi}:\overline{X}\to X$ of $\phi$. Let $G$ be the Galois group of $\overline{\phi}$, let $H$ be the subgroup corresponding to $X'$ and let %
 $D$ be the $G$-monodromy labeling associated to $\overline{\mathcal{T}}$ by \cref{lem:InducedMonodromyLabeling}. Then we have an isomorphism of metric graphs $\phi^{-1}(\mathcal{T})\simeq \mathcal{T}_{H}$, where $\mathcal{T}_{H}$ is the graph induced from \cref{def:MonodromyTower}. 
 \end{theorem}
 \begin{proof}
 The topological structure follows from the construction in \cref{def:MonodromyTower}, \cref{lem:DoubleCosetSpaces} and the considerations in \cref{lem:InducedMonodromyLabeling}. To see that the edge lengths are correct, 
note that $\phi^{\an}$ is locally piecewise-linear with expansion factor the local degree. In the Galois case, this local degree is exactly $|D_{x}|$, and the general case easily follows from this. 
 \end{proof}

 \begin{remark}\label{rem:GlueLocalPicture}
We recall a similar theorem from \cite{Recovering23} here. Let  $\phi: X'\to{X}$ be a %
 finite separable dominant morphism of normal connected Noetherian schemes  and let $T\subset {X}$ be a finite subset, which we view as a poset through specializations. %
 One can then reconstruct the poset structure of $\phi^{-1}(T)$ by locally gluing double cosets, see  %
 \cite[Theorem 2.27]{Recovering23}. The idea is locally similar to the one in \cref{pro:DCStoGraphs}: we consider an edge in the Hasse diagram of $T$ and we attach local monodromy groups over the vertices. Since we do not always have inclusions of monodromy groups as in the monodromy labelings defined here (see \cite[Example 2.14]{Recovering23}), we introduce a modified monodromy group over the edge. This group allows us to glue the two fibers over the vertices. To obtain the global structure of the poset $\phi^{-1}(T)$, we then glue the local pictures over the edges using suitable local isomorphisms. In \cite{Recovering23}, this is achieved using $2$-limits of the corresponding topoi $D_{x}\mhyphen\mathrm{Sets}$. In the context of Berkovich spaces, the corresponding notion would have to be a certain continuous $2$-limit. We will not go into this in further detail here, since this additional topological gluing data is unnecessary for metric trees.
 \end{remark}
 
\section{The geometric tower of modular curves}\label{sec:CoveringsModularCurves}

In this section we study the geometric tower of modular curves over the field of complex $p$-adic numbers $\C_{p}$. %
We start by introducing functorial notation for the various group-theoretic objects in \cref{sec:GroupNotation}. We review the tower of modular curves in \cref{sec:ModularTower}, and we determine the tame structure of these coverings in \cref{sec:TameStructure}. We then review the results in \cite{WS16} in \cref{sec:ReviewWS}, and we use this to determine the wild structure of the tower in \cref{sec:WildTower}. In \cref{sec:PrunedSkeleton} we prove \ref{thm:MainThm3}, which is our main theorem.
We conclude this section with a quick application to the potential good reduction of quotients of Jacobians of modular curves, see \cref{sec:PotentialGoodReduction}.

\subsection{Subgroups of $\SL_{2}$}\label{sec:GroupNotation}

Let $\SL_{2}$ be the functor $\text{(Rings)} \to \text{(Groups)}$ sending a commutative unitary ring $R$ to the corresponding matrix group $\SL_{2}(R)$. This functor is representable by $\Z[a,b,c,d]/(ad-bc-1)$. We fix our notation for various subfunctors of $\SL_{2}$ here. %

\begin{definition}\label{def:StandardFunctors}
Let $\Gamma_{0},\Gamma_{1}, \Gamma^{\pm}_{1},\Gamma_{sp},\Gamma^{+}_{sp}\subset \SL_{2}$ be the representable functors $(\mathrm{Rings})\to (\mathrm{Groups})$ defined by 
\begin{align*}
\Gamma_{0}(R)&=\{\begin{pmatrix}
a & b \\
0 & d
\end{pmatrix}
\in\SL_{2}(R)\},\\
\Gamma_{1}(R)&=\{\begin{pmatrix}
1 & b \\
0 & 1
\end{pmatrix}
\in\SL_{2}(R)\},\\
\Gamma^{\pm}_{1}(R)&=\{\begin{pmatrix}
\pm{}1 & b \\
0 & \pm{}1
\end{pmatrix}
\in\SL_{2}(R)\},\\
\Gamma_{sp}(R)&=\{\begin{pmatrix}
t & 0\\
0 & t^{-1}
\end{pmatrix}
\in\SL_{2}(R)\},\\
\Gamma_{sp}^{+}(R)&=\{\begin{pmatrix}
t & 0\\
0 & t^{-1}
\end{pmatrix}
\in\SL_{2}(R)\}\cup \{\begin{pmatrix}
 0 & -t^{-1}\\
t & 0
\end{pmatrix}
\in\SL_{2}(R)\}.
\end{align*}
We refer to these as the standard Borel subgroup, the unipotent subgroup, the projectivized unipotent subgroup, the split torus, and the normalizer of the split torus.
\end{definition}

\begin{remark}
These groups are related to the corresponding congruence subgroups in $\SL_{2}(\Z)$ by taking pre-images. For instance, we have that $\Gamma_{1}(N)$ is the pre-image of $\Gamma_{1}(\Z/N\Z)$ under the map $\SL_{2}(\Z)\to \SL_{2}(\Z/N\Z)$.
\end{remark}

\begin{remark}
Note that the functors $H$ above define closed subgroup schemes of $\SL_{2}$ that are smooth over $\Z$ (or $\Z[1/2]$ for $\Gamma^{\pm}_{1}$), with corresponding %
ideals %
\begin{align*}
J_{0}&=(c),\\
J_{1}&=(a-1,d-1,c),\\
J^{\pm}_{1}&=(a^2-1,ad-1,c),\\
J_{sp}&=(b,c),\\
J^{+}_{sp}&=(b,c)\cap (a,d).
\end{align*}

Since the Hom-functor %
preserves limits and the direct product of rings is a limit, we find that %
\begin{equation*}
H(\Z/N\Z)\simeq H(\Z/N_{1}\Z)\times H(\Z/N_{2}\Z)
\end{equation*}
for $N_{i}\in\mathbb{Z}$ with $(N_{1},N_{2})=1$. %
\end{remark}

We will also need a set of functors that $p$-adically connect the standard Borel functor $\Gamma_{0}$ to $\SL_{2}$. We will call these Hecke-Iwahori functors. As opposed to the previous subfunctors, these will not be representable. 
\begin{definition}
Let $p$ be a prime number. The Hecke-Iwahori group functor $I_{k,p}: \text{(Rings)}\to \mathrm{(Groups)}$ of level $k\geq{1}$ in $\SL_{2}$ with respect to $p$  is defined by 
\begin{equation*}
I_{k,p}(R)=\{\begin{pmatrix}
a & b\\
c& d
\end{pmatrix}\in \SL_{2}(R):\exists\,c_{0}\in{R} \text{ s.t. }c=c_{0}p^{k}\}.
\end{equation*}
Here we again write $p$ for the image of $p\in\Z$ in $R$. %
If $p$ is clear from context, then we also write $I_{k}$ for this functor. %
For $k=0$, we have $I_{0}=\SL_{2}$. %
If the ring $R$ in question is clear from context, then we will sometimes also omit $R$ for brevity. 
\end{definition}

 We will also be interested in the projectivizations of these functors. 
\begin{definition}\label{def:Projectivization}
Let $H$ be a subgroup functor of $\SL_{2}$ and let $P:\SL_{2}\to \PSL_{2}$ be the natural transformation sending a matrix $M\in\SL_{2}(R)$ to its image in $\PSL_{2}(R)=\SL_{2}(R)/\langle -1\rangle$. The projectivization of $H$ is the functor $P(H):(\mathrm{Rings})\to (\mathrm{Groups})$ defined by $P(H)(R):=P(H(R))$. Here we view $H(R)$ as a subset of $\SL_{2}(R)$.  %
\end{definition}

Note that projectivizations are usually not representable. Indeed, let $N=N_{1}N_{2}$ with $(N_{1},N_{2})=1$. For $\SL_{2}$, we have that   %
the surjective map 
$$\PSL_{2}(\Z/N\Z)\to  \PSL_{2}(\Z/N_{1}\Z)\times \PSL_{2}(\Z/N_{2}\Z)$$
induced by the isomorphism 
$$\SL_{2}(\Z/N\Z)\to  \SL_{2}(\Z/N_{1}\Z)\times \SL_{2}(\Z/N_{2}\Z)$$
is not injective, so that $\PSL_{2}$ does not commute with taking limits. We do however have an exact sequence %
\begin{align}\label{eq:ExactSequences}
(1)\to \mathrm{SL}_{2}(\Z/N_{1}\Z)\to \PSL_{2}(\Z/N\Z)\to \PSL_{2}(\Z/N_{2}\Z)\to (1),
\end{align}
and a similar one with the roles of $N_{1}$ and $N_{2}$ reversed. 
These will be useful in the upcoming sections.
\begin{definition}
Let $N=p^{n}M$ with $(M,p)=1$, and let $H\subset \PSL_{2}(\Z/N\Z)$ be a subgroup. We say that $H$ is decomposable with respect to $p$ if there exist subgroups $H_{p}$ and $H_{M}$ in $\SL_{2}(\Z/p^{n}\Z)$ and $\SL_{2}(\Z/M\Z)$ respectively such that %
$P(H_{p}\times H_{M})=H$. %
\end{definition}

\begin{lemma}
Let $H$ be a subgroup scheme of $\SL_{2}$. Then $PH(\Z/N\Z)$ is decomposable with respect to every prime divisor $p$ of $N$. %
\end{lemma}
\begin{proof}
Since representable functors commute with taking limits, we have $H(\Z/N\Z)\simeq H(\Z/p^{n}\Z)\times H(\Z/M\Z)$, which provides the two subgroups. 
\end{proof}

\subsection{The tower of modular curves}\label{sec:ModularTower}

We establish our notation for the geometric tower of modular curves here. We will 
first follow the exposition in \cite{Rohrlich1997} and \cite{DiamondShurman2005};  
the moduli-theoretic interpretation of modular curves in \cite{KM85} and \cite{WS16} will be used later on. We will try to point out the non-canonicity in the exposition wherever we can.  %

Let $X(1)=\mathbb{P}^{1}_{\mathbb{C}_{p}}$ be the standard modular curve over $\mathbb{C}_{p}$ with local coordinate given by the $j$-invariant.  
Consider the elliptic curve $E_{j}$ over $\C_{p}(j)$ given by the affine equation 
\begin{equation*}
y^2=4x^3-ax-a,
\end{equation*}
where $a=\dfrac{27j}{j-1728}$. We {choose} a compatible set of bases $\{P_{N},Q_{N}\}$ of $E[N]$ for $N$ ranging over the positive integers. This also specifies a compatible set of primitive $N$-th roots of unity through 
the Weil pairing $\langle P_{N},Q_{N}\rangle=\mu_{N}\in \C_{p}$.  By letting the absolute Galois group $G_{\C_{p}(j)}$ act on the $E[N]$, we then obtain 
a Galois representation 
\begin{equation*}
\rho_{E}: G_{\mathbb{C}_{p}(j)}\to \SL_{2}(\hat{\Z})/\{1,-1\}=\PSL_{2}(\hat{\Z})
\end{equation*}  
which represents the monodromy of the tower of modular curves. Here we view $\PSL_{2}(\hat{\Z})$ as the inverse limit of the groups $\PSL_{2}(\Z/N\Z)$. 
For any open subgroup $H$ of $\PSL_{2}(\hat{\Z})$ the subgroup $\rho^{-1}_{E}(H)$ %
gives a covering of curves 
\begin{equation*}
X_{H}\to X(1)
\end{equation*}
by taking the normalization of $X(1)$ in the function field extension corresponding to the fixed field of the group $\rho^{-1}_{E}(H)$. If $H$ is moreover normal, then $X_{H}\to{X(1)}$ is Galois with Galois group $\PSL_{2}(\hat{\Z})/H$. %

\begin{definition}\label{def:ModularCurve}
Let $H\subset \PSL_{2}(\hat{\Z})$ be an open subgroup. We write $X_{H}$ for the corresponding %
connected modular curve over $\C_{p}$. For an inclusion of open subgroups $H_{1}\subset H_{2}\subset \PSL_{2}(\hat{\Z})$, we have an induced finite morphism $\phi_{H_{1}/H_{2}}: X_{H_{1}}\to X_{H_{2}}$ of algebraic curves of degree $[H_{2}:H_{1}]$.  %
This morphism is Galois if $H_{1}$ is a normal subgroup in $H_{2}$. If $H_{2}=\PSL_{2}(\hat{\Z})$, then we usually omit $H_{2}$. We call the compatible collection of coverings $X_{H}\to X(1)$ arising from the open subgroups of $\PSL_{2}(\hat{\Z})$ the geometric tower of modular curves.   
\end{definition}

\begin{definition}\label{def:StandardModularCurves}
Let $\overline{H}\subset \PSL_{2}(\Z/N\Z)$ be a subgroup. The subgroup $H\subset \PSL_{2}(\hat{\Z})$ %
associated to $\overline{H}$ is the inverse image of $\overline{H}$ under the surjective map
\begin{equation*}
\PSL_{2}(\hat{\Z})\to \PSL_{2}(\Z/N\Z).
\end{equation*} 
We write $X_{H}$ for the corresponding modular curve. If $\overline{H}=(1)\subset \PSL_{2}(\Z/N\Z)$, then we denote the corresponding modular curve by $X(N)$. Similarly, if $\overline{H}$ is $P(\Gamma(\Z/N\Z))$ for one of the standard functors in \cref{def:StandardFunctors}, then we denote the corresponding modular curves by $X_{0}(N)$, $X_{1}(N)$, $X_{sp}(N)$ and $X^{+}_{sp}(N)$. 
\end{definition}

\begin{remark}\label{rem:DecomposableGCoverings}
Let $N=p^{n}M$ for $M\geq{3}$ with $(M,p)=1$ and $n\geq{1}$. 
Suppose that $H\subset \PSL_{2}(\Z/N\Z)$ is decomposable, so that there exist two subgroups $H_{p}$ and $H_{M}$ in $\SL_{2}(\Z/p^{n}\Z)$ and $\SL_{2}(\Z/M\Z)$ respectively such that $P(H_{p}\times H_{M})=H$. Write $E({H_{M}}):=P(\SL_{2}(\Z/p^{n}\Z)\times H_{M})$ %
for the inverse image of $P(H_{M})\subset \PSL_{2}(\Z/M\Z)$ under the map 
\begin{equation*}
\PSL_{2}(\Z/N\Z)\to \PSL_{2}(\Z/M\Z).
\end{equation*} 
Let $$\epsilon(H_{p})=\begin{cases}
H_{p} \quad \mathrm{ if }\,-1\notin H_{M},\\
\pm{H_{p}} \quad\mathrm{ if }\,-1\in H_{M}.
\end{cases}
$$ An easy check shows that the injection $\SL_{2}(\Z/p^{n}\Z)\to E({H_{M}})$ given by $\sigma\mapsto (\sigma,1)/\pm{1}$ induces a %
bijection of left-coset spaces  
\begin{equation*}
\SL_{2}(\Z/p^{n}\Z)/\epsilon(H_{p})\to E({H_{M}})/H.
\end{equation*}
Consider the chain of coverings $X(N)\to X_{H}\to X_{E({H_{M}})}\to X(1)$ induced from 
the chain of subgroups $(1)\subset H \subset E({H_{M}}) \subset \PSL_{2}(\Z/N\Z)$. Note that $X(N)\to X_{E({H_{M}})}$ is Galois with Galois group $E({H_{M}})$. In light of the material in Sections \ref{sec:MonodromyLabelingSection} and \ref{sec:SimplicialStructureCoverings}, we view the subcovering $X_{H}\to X_{E(H_{M})}$ inside this Galois covering as the one corresponding to the inclusion of subgroups $H\subset E(H_{M})$.  To calculate with the corresponding %
left-cosets $E({H_{M}})/H$ as in \cref{pro:DCStoGraphs}, %
we can then work with $\SL_{2}(\Z/p^{n}\Z)/\epsilon(H_{p})$.  

For instance, if both $H_{M}$ and $H_{p}$ are $(1)$, then we obtain the chain of coverings 
\begin{equation*}
X(N)\to X(M)\to X(1).
\end{equation*}
Here the first is Galois with Galois group $\SL_{2}(\Z/p^{n}\Z)$, and the second is Galois with Galois group $\PSL_{2}(\Z/M\Z)$, see the exact sequences after \cref{def:Projectivization}. 

If $M=1,2$, then we use the following trick to reduce to the above. If we have a decomposable subgroup in $\PSL_{2}(\Z/N'\Z)$ for $N'|N$, then we take the inverse images of $H'_{p}$ and $H'_{N}$ under the projection maps for $\SL_{2}$ to obtain a decomposable subgroup in $\PSL_{2}(\Z/N\Z)$. 
\end{remark}

\subsection{The induced tower of Berkovich spaces}\label{sec:TameStructure}

Let $X_{H}$ be a modular curve as in \cref{def:ModularCurve}. We write $X^{\an}_{H}$ for its Berkovich analytification, consisting of pairs $(P,v_{P})$, where $P\in{X_{H}}$ and $v_{P}:\C_{p}(P)\to \mathbb{R}\cup\{\infty\}$ is a valuation on the residue field $\C_{p}(P)$ that extends the valuation on $\C_{p}$. Let $H$ be an open normal subgroup of $\PSL_{2}(\hat{\Z})$ with Galois covering $X_{H}\to X(1)$ and let ${x}_{H}\in X^{\an}_{H}$ be a point lying over $x\in X(1)^{\an}$. The action of the Galois group on the points lying over $x$ is transitive, and we write $D_{{x}_{H}/x}$ for the corresponding stabilizers or decomposition groups. 
\begin{definition}\label{def:DGDefinition}
Let $\overline{x}:=(x_{H})\in \prod X^{\an}_{H}$ be a set of points lying over $x\in X(1)^{\an}$ in the modular curves $X^{\an}_{H}$. Here the product runs over all open subgroups of $\PSL_{2}(\hat{\Z})$.  We say that this is a compatible set if $\phi_{H_{1}/H_{2}}(x_{H_{1}})=x_{H_{2}}$ for $H_{1}\subset H_{2}$. The decomposition group of $\overline{x}$ is the inverse limit of $D_{x_{H}/x}$ for all open normal subgroups $H$ of $\PSL_{2}(\hat{\Z})$. We denote it by $\rho_{E}(D_{\overline{x}})$.   
\end{definition}  
\begin{remark}
We can also define this decomposition group in the following way. %
Consider the normalization $\overline{X}$ of $X(1)$ in an algebraic closure of $\C_{p}(j)$. We then define the underlying set of the analytification of this $K$-scheme in the usual way (without giving this space any additional structure) and take a point $\overline{x}\in \overline{X}^{\an}$ lying over $\overline{x}$. The decomposition group is then the stabilizer of this point, and we obtain the decomposition group in \cref{def:DGDefinition} as the image of $D_{\overline{x}}$ under the Galois representation $\rho_{E}$. This also explains our notation for $\rho_{E}(D_{\overline{x}})$. 
\end{remark}

We now describe the images of the various decomposition groups in \cref{def:DGDefinition} in $\PSL_{2}(\Z/M\Z)$ for $(M,p)=1$. To that end, we recall the notions of residual tameness. %
Let $\phi: X'\to X$ be a finite morphism of smooth proper curves over $\mathbb{C}_{p}$ with induced morphism of Berkovich analytifications $\phi^{\an}$. We will also write $\phi$ for this morphism to ease notation.  %
Let $x'\in X'^{\an}$ be a point mapping to $x=\phi(x')$. We say that $\phi$ is residually tame at $x'$ if the induced morphism of completed residue fields 
\begin{equation*}
\mathcal{H}(x)\to \mathcal{H}(x')
\end{equation*}
is tame or moderate in the sense of \cite[Section 2.4, Page 48]{Berkovich1993}, see  \cite[Proposition 2.4.7]{Berkovich1993}.  %
We say that the morphism $\phi$ is residually tame if it is residually tame at every point of $X'^{\an}$. We say that $\phi$ is topologically tame at $x'$ if the degree of the induced morphism of residue fields
\begin{equation*}
\tilde{\mathcal{H}}(x)\to \tilde{\mathcal{H}}(x')
\end{equation*}
is coprime to $p$. The morphism $\phi$ is topologically tame if it is topologically tame at every point $x'\in X'^{\an}$. Topologically tame morphisms are automatically residually tame, but not conversely.

\begin{lemma}\label{lem:ResidualTameness}
Let $M\geq{3}$ with $p\nmid{M}$ and $p\geq{5}$. The covering $X(M)\to X(1)$ is residually tame. %
\end{lemma}

\begin{proof}
We will show that the covering of analytifications is locally Kummer of degree coprime to $p$ outside the Gauss vertex. Over the Gauss vertex, the induced morphism of residue curves will be separable, so that the morphism $X(M)\to X(1)$ is also residually tame there. %

Let $\mathcal{M}(1)$ be the stack of generalized elliptic curves whose fibers are geometrically integral (this is denoted by $\mathcal{M}_{1}$ in \cite{DR73}, see Remarque 2.6). This is a Deligne-Mumford stack that is proper and smooth over $\Spec(\Z)$. 
Consider the morphism of coarse moduli spaces $\mathcal{X}(M)\to \mathcal{X}(1)$ associated to the morphism of stacks $\mathcal{M}(M)\to\mathcal{M}(1)$. Here $\mathcal{M}(M)$ is the stack corresponding to elliptic curves $E\to S$ with an isomorphism $(\Z/M\Z)^{2}\to E[M]$. In particular, this will have various connected components after taking a base change to $\C_{p}$, corresponding to the different $M$-th roots of unity, and each of these will be isomorphic to the curve $X(M)$ introduced in \cref{def:StandardModularCurves}. %

We base change the above stacks to $\Z[1/M]$ and retain the same notation. %
For $M\geq{3}$, $\mathcal{M}(M)$ is representable, with representing scheme $\mathcal{X}(M)$ (the non-cuspidal case is \cite[Corollary 4.7.2]{KM85}). %
For $\mathcal{M}(1)$, we have $\mathcal{X}(1)\simeq \mathbb{P}^{1}_{\Z[1/M]}$ by \cite[Th\'{e}or\`{e}me 1, Page 267]{DR73}. %

The morphism of stacks $\mathcal{M}(M)\to\mathcal{M}(1)$ is finite \'{e}tale over the non-cuspidal substack $\mathcal{Y}(1)$, %
so that the induced morphisms of \'{e}tale local rings are isomorphisms.  Using \cite[Section 8.2.1, Page 172]{DR73} over $\mathcal{Y}(1)$, we see that we only have ramification on the level of coarse moduli spaces at points with additional automorphisms. Moreover, the automorphism groups are finite groups of order at most $24$. From this, we deduce that the induced morphisms of \'{e}tale local rings for the {coarse moduli spaces} are Kummer. %
Indeed, if $x$ is associated to a codimension one point over the generic fiber not equal to $0,1728,\infty$, then there are no additional automorphisms (the automorphism $-1$ acts trivially on the \'{e}tale local ring, see \cite[Section 1.3.2]{Edixhoven90} for instance). Similarly, if $x$ is the generic point of a component of the special fiber, then there are no additional automorphisms. Using Abhyankar's lemma and purity of the branch locus, we then find that the extensions of \'{e}tale local rings are Kummer of degree coprime to $p$. This in turn implies that the corrresponding morphisms of Berkovich analytifications are topologically tame at these points, and thus residually tame. The residual tameness over the Gauss vertex also immediately follows from this description.     

Over the cusp at infinity, we have that the morphism of stacks $\mathcal{M}(1)\to \mathcal{X}(1)$ is \'{e}tale by \cite[Lemme 1.5, page 269]{DR73}, so that we can %
use the Tate curve. Base changing $\mathcal{X}(M)\to \mathcal{X}(1)$ over the valuation ring $R$ of $\mathbb{C}_{p}$, we find that the extension of completed local rings is given by $R[[q]]\to R[[q]][q^{1/N}]$, as desired. %
\end{proof}

By \cref{lem:ResidualTameness} and \cite[Theorem 4.13]{Helminck2023}, we find that  
the induced morphisms of Berkovich analytifications are completely governed by the behavior at their type-$1$ points. %
We define a monodromy labeling on a metric tree in $X(1)^{\an}$ to describe these. Let $\zeta_{G}\in X(1)^{\an}$ be the Gauss point with respect to the $j$-coordinate. For a point $P$ of type $1$, we can evaluate $j$ at $P$ to obtain an element $j(P)\in\C_{p}\cup\{\infty\}$. We will sometimes omit $P$ %
and write $j$ for the corresponding $j$-invariant if $P$ is clear from context. %
Let $e_{j}=e_{j(P)}$ be the unique line segment from $\zeta_{G}$ to $P$ in $X(1)^{\an}$. %
We define the tame metric tree $\mathcal{T}_{tame}$ to be the union of $e_{0}$, $e_{1728}$ and $e_{\infty}$. The corresponding monodromy groups are as follows. %
Set  \begin{align*}
\sigma_{0}&=\begin{pmatrix}
0 & -1 \\
1 & 1
\end{pmatrix},\\
\sigma_{1728}&=\begin{pmatrix}
0 & 1\\
-1 & 0
\end{pmatrix},\\
\sigma_{\infty}&=\begin{pmatrix}
1 & 1\\
0 & 1
\end{pmatrix}
\end{align*}
and let $H_{j}=
\langle \sigma_{j}\rangle\subset \SL_{2}(\hat{\Z})$. As before, we write $P(H_{j})$ for the induced subgroup of $G'$. %
We define a $G'$-monodromy labeling on $\mathcal{T}_{tame}$ as follows: $D_{x}=P(H_{j})$ for $x\in e_{j}\bs \{\zeta_{G}\}$ and $D_{\zeta_{G}}=G'$. The following is now a consequence of \cite[Theorem 4.13]{Helminck2023} and standard theorems on the ramification of the complex tower of modular curves.   

\begin{lemma}\label{lem:TameDecompositionGroups}
There is a compatible set of sections of $\mathcal{T}_{tame}$ in the tower induced by $G'$ such that the associated $G'$-monodromy labeling is $D$. For every open subgroup $H\subset \PSL_{2}(\hat{\Z}')$, the induced $\mathcal{T}_{H,tame}$ is a skeleton for $X_{H}$. 
\end{lemma}

We now construct the canonical supersingular metric tree inside $X(1)^{\an}$. We start with the Gauss point $\zeta_{G}$ of $X(1)^{\an}$ with respect to the coordinate $j$, which will be the central vertex of $\mathcal{T}_{can}$. Let $\tilde{E}/\mathbb{F}_{p^{2}}$ be a supersingular elliptic curve with $j$-invariant $\overline{j}\in\mathbb{F}_{p^{2}}$. We choose a set of lifts $\mathcal{S}\subset \mathbb{Q}^{\unr}_{p}\subset \mathbb{C}_{p}$ of these supersingular $j$-invariants. As before, we will identify these with points $P$ of type $1$ in $X(1)^{\an}$, and we will write $j\in\mathcal{S}$ rather than $j(P)$. %
Note that $|\mathcal{S}|=g(X_{0}(p))+1$, where $g(X_{0}(p))$ is the genus of the modular curve $X_{0}(p)$. If $\overline{0}$ or $\overline{1728}$ is a supersingular $j$-invariant, then we assume that the lifts are $0$ and $1728$ respectively. For every $x\in\mathcal{S}$, there is a unique geodesic $e_{j}$ from $x$ to $\zeta_{G}$, which we view as a continuous injective map %
$[0,\infty]\to X(1)^{\an}$, where $0$ is mapped to $\zeta_{G}$ and $\infty$ is mapped to $x$. We moreover use the normalization induced from \cite[Section 2.3]{BPRa1} with $v(p)=1$, which we call the standard parametrization. %

To define the canonical supersingular metric tree $\mathcal{T}_{can}$, we will need the following function on $\mathcal{S}$: %
\begin{equation*}
b(j)=\begin{cases}
1 \text{ if }j\neq{0,1728},\\
2 \text{ if }j=1728,\\
3 \text{ if }j=0.
\end{cases}
\end{equation*}       
\begin{definition}
Let $e_{j}$ be the infinite line segment from $\zeta_{G}$ to $x\in\mathcal{S}$ in $X(1)^{\an}$ and let $e_{j,c}$ be the closed line segment corresponding to $[0,b(j)p/(p+1)]$ in the standard parametrization $[0,\infty]$ of $e_{j}$. The canonical supersingular metric tree $\mathcal{T}_{can}$ is the union of these line segments $e_{j,c}$. This is independent of our choice of lifts $\mathcal{S}\subset \mathbb{Q}^{\unr}_{p}$. There is a canonical retraction map %
$X(1)^{\an}\to \mathcal{T}_{can}$, which we denote by $\tau_{\mathcal{T}_{can}}$.  Consider the half-closed half-open interval %
\begin{equation*}
[b(j)p^{1-n}/(p+1),b(j)p^{2-n}/(p+1))\subset [0,\infty]
\end{equation*}
in the standard parametrization of $e_{j}$. 
We write $I_{n,j}$ for the corresponding line segment in $\mathcal{T}_{can}$. We refer to this as the (strict) canonical locus of order $p^{n}$. The image of the point $b(j)p^{1-n}/(p+1)$ under the map $[0,\infty]\to X(1)^{\an}$ corresponding to $e_{j}$ for $n\geq{1}$ is $\zeta_{j,n}$. The image of the point $b(j)p^{1- n}/2$ for $n\geq{1}$ is $\zeta'_{j,n}$.        
\end{definition}

\begin{example}
The canonical supersingular tree for $p=37$ is shown in \cref{fig:CanonicalSupersingularTree2}. Here we can choose $\mathcal{S}=\{8, 3+\sqrt{15},3-\sqrt{15}\}$. The black intermediate vertices are $\zeta_{j,n}$. The $\zeta'_{j,n}$ lie in between the various $\zeta_{j,n}$.   
\end{example}

We will give a moduli-theoretic interpretation of this canonical supersingular tree in the next section.

\subsection{An interlude on elliptic curves over valued fields and canonical subgroups}\label{sec:EllipticCurvesValuedFields}

We give a short moduli-theoretic description of elliptic curves over valued fields. %
\begin{definition}
Let $L_{i}\supset \C_{p}$ be two valued field extensions, and let $E_{i}/L_{i}$ be two elliptic curves. We say that the $E_{i}$ are $v$-isomorphic (or simply: isomorphic) if %
there exists 
an isomorphism of $L$-schemes \begin{equation*}
E_{1}\times_{\Spec(L_{1})} \Spec(L) \to E_{2}\times_{\Spec(L_{2})}\Spec(L)
\end{equation*}  
arising from valued field extensions $L_{i}\to L$. 

Let $E/L$ be an elliptic curve over a valued field $L\supset \C_{p}$ and suppose that the $j$-invariant of $E$ is not in $\C_{p}$, so that we obtain %
a homomorphism $\C_{p}(j)\to L$. We define a point $s(E/L)\in X(1)^{\an}$ by giving $\C_{p}(j)$ the valuation induced from $L$. This is independent of the chosen $v$-isomorphism class of $E$.
\end{definition}

Let $P\in X(1)^{\an}$ be a type-$2$ point with completed residue field $L\supset \C_{p}(j)$. We can consider the elliptic curve $E_{j,L}$ from \cref{sec:ModularTower} as an elliptic curve over this valued field. Applying the map $s(\cdot)$, we then directly see that $s(E_{j,L})=P$. We thus see that $E_{j}$ gives rise to a representative of every isomorphism class of elliptic curves over a valued field by endowing $\C_{p}(j)$ with different valuations.     
Moreover, the unique stable model $\mathcal{E}/R_{L}$ for each of these points in $X(1)^{\an}$ can already be defined over the valuation ring arising from the valuation on $\C_{p}(j)$. Indeed, one easily sees that no finite extension is needed here since the value group of $\C_{p}$ is divisible, see \cite{Helminck2019Faithful} for an elementary approach using minimal models. %

We now recall the notion of a canonical subgroup over an arbitrary complete valued field here, see \cite{Rabinoff2012a}. Let $E$ be an elliptic curve over an extension $L\supset \mathbb{C}_{p}$ as above, and suppose that $v(j)\geq{0}$, so that we can find a smooth model $\mathcal{E}/R_{L}$ (we will not need the bad reduction case). %
Note that the smooth model is uniquely determined up to automorphism (see for instance \cite[Theorem 5.38]{BPR2016}).
Let $\overline{\infty}$ be the reduction of $\infty$, which we view as a closed point of $\mathcal{E}$. The completion of $\mathcal{O}_{\mathcal{E},\overline{\infty}}$ is then isomorphic to $R_{L}[[T]]$ for a parameter $T$. If $\mathcal{E}$ is given by a minimal Weierstrass equation with coordinates $x$ and $y$, and $\mathrm{char}(k_{L})\neq{2,3}$, then we can for instance choose $T=x/y$.
\begin{definition}
Let $G=E[p^{n}](\overline{L})$ be the geometric $p^{n}$-torsion points of $E$ and let $r>0$ be a fixed real number. We define $G_{r}=\{P\in{G}:v(T(P))>r\}$. If $G_{r}\simeq \Z/p^{n}\Z$ for some $r>0$, then we say that $E$ has a canonical subgroup of order $p^{n}$. This is independent of the chosen parameter $T$ of $\mathcal{E}$ at $\overline{\infty}$. The existence of a canonical subgroup of order $p^{n}$ is moreover insensitive to %
valued field extensions, so that we can talk about the canonical subgroup of order $p^{n}$ of a point $P\in X(1)^{\an}$ with $v_{P}(j)\geq{0}$. %
\end{definition}  

   We can now geometrically characterize the existence of canonical subgroups %
   as follows.   

\begin{lemma}\label{lem:CanonicalLoci1}
Let $\tau_{\mathcal{T}_{can}}$ be the retraction map $X(1)^{\an}\to \mathcal{T}_{can}$. %
Let $E$ be an elliptic curve with good reduction over a valued field $L\supset{\C_{p}}$ corresponding to a point $x\in X(1)^{\an}$ and let $\mathcal{E}/R_{L}$ be a smooth model over the valuation ring $R_{L}$. Then $\mathcal{E}/R_{L}$ admits a canonical subgroup of order $p$ if and only if $\tau_{\mathcal{T}_{can}}(x)$ lies in the interior of $\mathcal{T}_{can}$. Similarly, $\mathcal{E}/R_{L}$ admits a canonical subgroup of order $p^{n}$ if and only if $\tau_{\mathcal{T}_{can}}(x)$ is either $\zeta_{G}$, or it lies in $\bigcup_{m\geq{n}} I_{n,j}$ for some $j$.      
\end{lemma}

\begin{proof}
Let $M\geq{5}$. We retain the notation from the proof of \cref{lem:ResidualTameness}. Recall our assumption $p\neq 2,3$, which ensures that the morphism of coarse moduli spaces $\mathcal{X}(M)\to \mathcal{X}(1)$ is sufficiently tame. Note that the subquotient $\mathcal{X}_{1}(M)$ corresponding to the congruence subgroup $\Gamma_{1}(M)$ is also representable, so that $\mathcal{X}_{1}(M)\to \mathcal{X}(1)$ satisfies the same tameness properties.  According to \cite[Section 3]{Buzzard2003}, we need to find parameters at the supersingular points of $\mathcal{X}_{1}(M)$. If $\overline{j}\neq{0,1728}$, then $j$ is again a parameter since the map of coarse moduli spaces is \'{e}tale. If $\overline{j}=0$, then there are points at which $j^{1/3}$ is a parameter, and if $\overline{j}=1728$, then there are points at which $(j-1728)^{1/2}$ is a parameter. The canonical loci are then obtained by taking $v(H)<p^{2-n}/(p+1)$ for $H$ one of the parameters above. %
This immediately gives the desired loci.   
\end{proof}

\begin{remark}
If we identify $e_{j}$ with $[0,b(j)p/(p+1)]$, then we call $\tau(x)\in[0,b(j)p/(p+1)]$ the modified Hasse invariant of $x$. By the above, the modified Hasse invariant of a point $x$ completely determines whether the associated elliptic curve $E_{x}$ has a canonical subgroup of order $p^{n}$.    
\end{remark}

\subsection{The tower over the central vertex}\label{sec:pAdicTower}

In this section, we give %
the splitting behavior of the tower of modular curves over the central vertex of the canonical %
supersingular metric tree $\mathcal{T}_{can}$. Using \cref{lem:ResidualTameness}, we will see that we can promote the results in \cite[Section 13]{KM85} to arbitrary maps of modular curves.  %

\begin{lemma}\label{lem:DecompositionGroupCentralVertex2}
Let $\zeta_{G}\in X(1)^{\an}$ be the Gauss point. We define a $\PSL_{2}(\hat{\Z})$-monodromy labeling on $\{\zeta_{G}\}$ through $D_{\zeta_{G}}=P(\Gamma_{0}(\Z_{p})\times \SL_{2}(\Z'))$. There is a set of sections of $\zeta_{G}$ in the tower of modular curves whose associated monodromy labeling is $D$.  
\end{lemma}
\begin{proof}
Let $M\geq{3}$ be an integer coprime to $p$ and consider the $\SL_{2}(\Z/p^{n}\Z)$-covering $X(p^{n}M)^{\an}\to X(M)^{\an}$. Let $\mathcal{X}(M)^{can,\zeta^{i}_{M}}$ be the $\zeta^{i}_{M}$-canonical moduli problem over $\Z_{p}[\zeta_{M}]$ for some primitive $\zeta^{i}_{M}$, %
see \cite[Section 2.1]{EP21} (this provides the auxiliary level structure denoted by $\mathcal{P}$ there). Note that this defines a connected scheme. Similarly, we write $\mathcal{X}(p^{n}M)^{can,\zeta^{j}_{p^{n}M}}$ for the corresponding scheme over $\Z_{p}[\zeta_{p^{n}M}]$, and we assume that the Weil pairings are chosen in a compatible way, so that $(\zeta^{j}_{p^{n}M})^{p^{n}}=\zeta^{i}_{M}$. By applying a base change, we can consider both as schemes over $\Z_{p}[\zeta_{p^{n}M}]$.   %

Using \cite[Theorem 13.10.3]{KM85}, we find that there is a point $\zeta_{p^{n}M,G}\in X(p^{n}M)^{\an}$ whose decomposition group is $\Gamma_{0}(\Z/p^{n}\Z)$. Indeed, we can interpret the Berkovich-theoretic point $\zeta_{M,G}$ lying over $\zeta_{G}$ as being induced from the generic point of the special fiber of $\mathcal{X}(M)^{can,\zeta^{i}}$, and the covering over this point is \'{e}tale: we have that the special fiber of $\mathcal{X}(p^{n}M)^{can,\zeta^{j}_{p^{n}M}}$  is the disjoint union of a set of components that are smooth over $\overline{\mathbb{F}}_{p}$, and the induced maps are finite flat, so that the map is generically smooth of relative dimension zero. The Igusa components are moreover geometrically irreducible, so that we in fact obtain an induced point  %
$\zeta_{p^{n}M,G}\in X(p^{n}M)^{\an}$ with the desired decomposition group.  %
Note that there is a commutative diagram 
 \begin{equation*}
 \begin{tikzcd}
 (0)  \arrow[r] & D_{\zeta_{p^{n}M,G}/\zeta_{M,G}}  \arrow[r]\arrow[d] & D_{\zeta_{p^{n}M,G}/\zeta_{1,G}}  \arrow[r] \arrow[d] & D_{\zeta_{M,G}/\zeta_{1,G}}\arrow[d]\arrow[r] & (0) \\
  (0)  \arrow[r] & \SL_{2}(\Z/p^{n}\Z)  \arrow[r] & \PSL_{2}(\Z/p^{n}M\Z) \arrow[r] & \PSL_{2}(\Z/M\Z) \arrow[r] & (0)
 \end{tikzcd}
 \end{equation*}
 with exact horizontal rows. Since $D_{\zeta_{M,G}/\zeta_{1,G}}$ is $\PSL_{2}(\Z/M\Z)$ by \cref{lem:TameDecompositionGroups}, it follows that $ D_{\zeta_{p^{n}M,G}/\zeta_{1,G}}$ is contained in $P(\Gamma_{0}(\Z/p^{n}\Z)\times \SL_{2}(\Z/M\Z))$, and by comparing orders we find the desired equality.    
\end{proof}

\subsection{A review of \cite{WS16}}\label{sec:ReviewWS}

In this section we review some of the results and concepts in \cite{WS16}. We first describe the infinite Lubin-Tate tower and the various groups that act on this space. We then recall the notion of a CM point and its associated %
linking orders. %
Finally, we review the dual intersection graph $\mathcal{T}^{o}$ constructed in  %
\cite[Section 6.3]{WS16} and we determine the stabilizers of the Berkovich-theoretic points associated to the vertices of $\mathcal{T}^{o}$.  
 
 Let $E_{0}/\overline{\F}_{p}$ be a supersingular elliptic curve and let $G_{0}/\overline{\F}_{p}$ be the formal completion of $E_{0}$ at infinity. Since $E_{0}$ is supersingular, $G_{0}$ is a formal group of height $2$ and dimension $1$ over $\overline{\F}_{p}$. %
Let $K_{0}$ be the completion of the maximal unramified extension $\Q^{unr}_{p}$ of $\Q_{p}$ and write $\mathcal{C}$ for the category of %
complete local Noetherian $\mathcal{O}_{K_{0}}$-algebras with residue field $\overline{\F}_{p}$. Let $A\in \mathrm{Ob}(\mathcal{C})$. %
A deformation $(G,i)$ of $G_{0}$ consists of a one-dimensional formal group $G/A$ %
together with an isomorphism $i:G\otimes_{A}\overline{\F}_{p}\to G_{0}$. %
 This defines a functor $\mathcal{C}\to \mathrm{(Sets)}$ %
and this functor is representable by $A_{0}\simeq \mathcal{O}_{K_{0}}[[u]]$. %
We denote the associated %
affine formal scheme by $\mathcal{M}^{(0)}_{G_{0},0}$. We view this as the affine formal scheme corresponding to a supersingular point on the Katz-Mazur model of a suitable modular curve, see \cite[Proposition 4.7.4]{WS13}\footnote{Note that the moduli problem $\Gamma_{1}(N)$ used there can easily be replaced by $\Gamma(N)$, as this is only necessary to make the moduli problem representable.}.       

We can similarly study deformations of $G_{0}$ together with a level-$p^{n}$ Drinfeld structure in the sense of \cite[Section 2.2]{WS16}. The corresponding functor $\mathcal{C}\to (\mathrm{Sets})$ is representable by a regular complete local ring $A_{m}$, and we denote the corresponding formal scheme by $\mathcal{M}_{G_{0},m}=\mathrm{Spf}(A_{m})$. Consider the direct limit $A'_{\infty}=\lim_{\rightarrow}A_{m}$. Taking the completion of this ring with respect to the topology induced by $\mathfrak{m}_{0}\subset A_{0}$, we obtain %
the ring $A_{\infty}$. 
\begin{definition}\label{def:InfiniteLubinTateSpace}
The infinite Lubin-Tate space $\mathcal{M}^{(0)}_{G_{0},\infty}=\mathcal{M}^{(0)}_{\infty}$ associated to $G_{0}$ is $\Spf(A_{\infty})$. %
\end{definition}

We quickly recall the natural action of $\GL_{2}(\Z_{p})$ on this space. 
The functorial definition of $A_{m}$ allows us to define an %
action of %
$\GL_{2}(\Z/p^{m}\Z)$ on $A_{m}$. These actions are compatible for varying $m$, so that we also obtain an action of $\GL_{2}(\Z_{p})$ on the ring $A'_{\infty}$. This action is moreover continuous, as the action of $\GL_{2}(\Z_{p})$ on $\mathfrak{m}_{0}$ is trivial. We conclude that the action extends to $A_{\infty}$, and thus to %
$\mathcal{M}^{(0)}_{\infty}$. %

Using the Weil pairing on $G_{0}$, one obtains a natural map 
\begin{equation*}
\mathcal{M}^{(0)}_{G_{0},\infty}\to \mathcal{M}^{(0)}_{\bigwedge^{2}G_{0},\infty},
\end{equation*}
see \cite[Section 6]{WS13} and \cite[Section 2.5]{WS16}. Since $\bigwedge^{2}G_{0}$ is a formal group of dimension one and height one, we find that the latter is isomorphic to the completion of the integral $p^{\infty}$-cyclotomic extension $W_{cycl}\supset \mathcal{O}_{K_{0}}$ by classical Lubin-Tate theory. We now take the adic space associated to the base change of $\mathcal{M}^{(0)}_{G_{0},\infty}$ to $%
\C_{p}$. %
We denote this by $\mathcal{M}^{(0),ad}_{\infty,\overline{\eta}}$, as in \cite[Section 2.10]{WS16}. The connected components of this space correspond to the different embeddings of $W_{cycl}$ into $\mathcal{O}_{\C_{p}}$, see \cite{Strauch2008}.  
\begin{definition}
A connected component of $\mathcal{M}^{(0),ad}_{\infty,\overline{\eta}}$ corresponding to an embedding  $W_{cycl}\to \mathcal{O}_{\C_{p}}$ is denoted by $\mathcal{M}^{{o},ad}_{\infty,\overline{\eta}}$. This space is invariant under the action of $\SL_{2}(\Z_{p})$. We similarly denote the induced connected components of the adic spaces $\mathcal{M}^{(0),ad}_{G_{0},m,\overline{\eta}}$ by $\mathcal{M}^{o,ad}_{m,\overline{\eta}}$. We write $\pi_{m}$ for the natural map $\mathcal{M}^{{o},ad}_{\infty,\overline{\eta}}\to \mathcal{M}^{o,ad}_{m,\overline{\eta}}$.   
\end{definition}
Let $D$ be the division quaternion algebra over $\Q_{p}$ (unique up to isomorphism) with reduced norm $N:D\to \Q_{p}$. Concretely, this algebra can be represented as the non-commutative algebra over $\Q_{p}$ generated by $1,i,j,k$, subject to the relations $i^2=d$ for $d\in\Z_{p}^{*}$ a non-square, $j^2=p$ and $ij=k=-ij$. Note that this is the endomorphism algebra of the  
the formal group over $\Z_{p}$ associated to the fixed supersingular elliptic curve $\overline{E}_{0}$, up to isogeny. 

Consider the group 
$\mathcal{G}:=(\GL_{2}(\Q_{p})\times D^{*})^{\mathrm{det}=N}$ of pairs $(g_{1},g_{2})\in\GL_{2}(\Q_{p})\times D^{*}$ such that $\mathrm{det}(g_{1})=N(g_{2})$. This group acts on the connected component %
$\mathcal{M}^{\mathrm{o},ad}_{\infty,\overline{\eta}}$, and the action of  
$\SL_{2}(\Z_{p})$ described earlier is compatible with this action, in the sense that if we view it as a subgroup through the embedding $\sigma \mapsto (\sigma,1)$, then it gives the same action\footnote{In line with the definition in \cite{WS16}, this is a right-group action. The notation used here will reflect this. }. To describe the action of the part coming from the quaternion algebra, one has to use the crystalline nature of the universal cover $\tilde{G}$, which allows one to lift endomorphisms of $G_{0}$ to $\tilde{G}$. We refer the reader to \cite[Section 2.9]{WS16}. For a given point $x$ in $\mathcal{M}^{\mathrm{o},ad}_{\infty,\overline{\eta}}$ and $\sigma\in\mathcal{G}$, we write $x^{\sigma}$ for the image of $x$ under the action, as in \cite{WS16}.

 Among the various points in $\mathcal{M}^{o,ad}_{\infty,\overline{\eta}}$, the points with \emph{complex multiplication} play an especially important role in the construction of semistable models in \cite{WS16}, which we recall here. %
 We can represent $x\in \mathcal{M}^{o,ad}_{\infty,\overline{\eta}}(\mathcal{O}_{\C_{p}})$ by a triple $(G,i,\phi)$, where ${G}$ is a one-dimensional formal group over $\mathcal{O}_{\C_{p}}$, $i:{G}_{0}\to {G}\otimes \overline{\F}_{p}$ is an isomorphism, and 
\begin{equation*}
\phi: \Q_{p}^2\to V(G)
\end{equation*} 
is an isomorphism of $\Q_{p}$-vector spaces. Here $V(G)$ is the $p$-adic rational Tate module of $G$. 

\begin{definition}
Let $x\in \mathcal{M}^{o,ad}_{\infty,\overline{\eta}}(\mathcal{O}_{\C_{p}})$ with triple $(G,i,\phi)$.
For a quadratic extension $L\supset \Q_{p}$, we say that $x$ has complex multiplication by $L$ (or: CM by $L$) if there exists an injection 
\begin{equation*}
\mathcal{O}_{L}\to \mathrm{End}(G),
\end{equation*} 
where $\mathcal{O}_{L}$ is an order in $L$. The point $x$ has CM if it has CM by some $L\supset \Q_{p}$.  We say that $x$ is unramified if $L\supset \Q_{p}$ is unramified, and $x$ is ramified if $L\supset \Q_{p}$ is ramified.   
\end{definition}

For more information on these types of points, we refer the reader to \cite{Gross86}, \cite{Wewers2007} and \cite{CM2007}. %

\begin{remark}
Let $E/{\C_{p}}$ be the elliptic curve associated to the formal group $G$. In many cases, even though  
 the formal group $G$ has complex multiplication, the elliptic curve $E$ will not.  %
 These types of elliptic curves were called elliptic curves with fake CM in \cite{CM2007}. We note here that the elliptic curves with complex multiplication lie \emph{dense} in the space of fake CM points by \cite[Corollary 1.4]{HMRL2021}. 
We can thus view elliptic curves with fake CM as suitable limits of elliptic curves with real CM.   
\end{remark}

Consider a pair $(x,n)$ consisting of a CM point $x$ and an integer $n\geq{0}$. For each such pair, one has an affinoid subdomain $\mathcal{Z}_{x,m}\subset  \mathcal{M}^{o,ad}_{\infty,\overline{\eta}}$, and these cover $\mathcal{M}^{o,ad}_{\infty,\overline{\eta}}$. The exact definition of these affinoids can be found in \cite[Section 5.4]{WS16}. We note here that we follow \cite{WS16} in not adding the parentheses, which are saved for a special modified affinoid. Namely, if $v=(x,n)$, then $\mathcal{Z}_{v}=\mathcal{Z}_{(x,n)}$ is the affinoid constructed in \cite[Section 6.3]{WS16} by removing finitely many residue regions. %
\begin{lemma}\label{lem:AffinoidRedIntegral}
The reduction of $\mathcal{Z}_{x,n}$ is integral. 
\end{lemma}  
\begin{proof}
This follows from the results in \cite[Sections 5.5-5.8]{WS16} for the primitive vertices and the proof of \cite[Theorem 6.4.4]{WS16} for the imprimitive vertices. \end{proof}
\begin{remark}
Consider the non-singular projective curve $C/\overline{\F}_{p}$ corresponding to the affine model $y^{p}+y=x^{p+1}$, which occurs naturally as the curve associated to vertices $(x,2k)$ with $x$ unramified. %
Note that the covering induced by $(x,y)\mapsto x$ is completely ramified over $\infty$ (as one sees using Newton polygons), so it is connected. 
\end{remark}

By \cref{lem:AffinoidRedIntegral}, $\overline{\mathcal{Z}}_{x,m}$ has a unique generic point. Using the considerations in \cite[Remark 7.3.11]{BB2013}, we see that this corresponds to a valuation of rank one on the %
adic space $\mathcal{M}^{o,ad}_{\infty,\overline{\eta}}$. %
\begin{definition}
Let $x$ be a CM point, and let $n\geq{0}$ be an integer. Then the Berkovich or Hausdorff point $[x,n]\in \mathcal{M}^{o,ad}_{\infty,\overline{\eta}}$ is defined to be the rank-one valuation associated to the reduction of the affinoid $\mathcal{Z}_{x,m}$.  
\end{definition}  

The affinoids $\mathcal{Z}_{x,m}$ are not fixed by the group action $\mathcal{G}$, but their stabilizers can be given in terms of \emph{linking orders} $\mathcal{L}_{x,n}$. These are orders inside $M_{2}(\Q_{p})\times D$ that are related to the classification of $2$-dimensional representations by the local Langlands correspondence, see \cite{BH06} and \cite{WS10}. %
Rather than giving the full definition, we will go through an explicit example in some detail and refer the reader to \cite[Definition 4.2.2]{WS16}. %
\begin{example}\label{exa:LinkingOrderExample}
Let $L=\Q_{p}(\pi)$ with $\pi^2=p$, and let $D=(d,p)_{\Q_{p}}$ be the unique (up to isomorphism) division quaternion algebra over $\Q_{p}$ with $i^2=d$ for $d\in\Z^{*}_{p}$ a non-square and $j^2=p$. We write $\mathcal{O}_{L}=\Z_{p}[\pi]$ for the maximal order in $L$ with prime ideal $\mathfrak{p}_{L}=(\pi)$, and $\mathcal{O}_{D}$ for the maximal order in $D$.  
We fix an embedding ${L}=\Q_{p}(\pi)\to D=(d,p)_{\Q_{p}}$ using $\pi\mapsto j$. %
Let $\{e_{1},e_{2}\}$ be a basis of $V(G)\simeq \Q_{p}^{2}$, and consider the action of $L$ on $V(G)$ given by 
\begin{align*}
\pi(e_{1})&=e_{2},\\
\pi(e_{2})&=pe_{1}.
\end{align*}
Let $x$ be a point with complex multiplication by $L$ and suppose that the corresponding embeddings $i_{1}:L\to M_{2}(\Q_{p})$ and $i_{2}:L\to D$ are as above. We write $\Delta_{x}(L):L\to M_{2}(\Q_{p})\times D$ for the diagonal embedding and define 
\begin{align*}
\varpi_{1}&=\begin{pmatrix}
1 & 0\\
0 & -1
\end{pmatrix}\\
\varpi_{2}&=j.
\end{align*}
Suppose that $n=2k+1$ is odd. We then have
\begin{equation*}
\mathcal{L}_{x,n}=\Delta_{x}(\mathcal{O}_{L})+(\mathfrak{p}^{n}_{L}\times \mathfrak{p}^{n}_{L})+(\mathfrak{p}^{k+1}_{L}\varpi_{1}\times \mathfrak{p}^{k+1}_{L}\varpi_{2}).
\end{equation*}
Similarly, if $n=2k$, then 
\begin{equation*}
\mathcal{L}_{x,n}=\Delta_{x}(\mathcal{O}_{L})+(\mathfrak{p}^{n}_{L}\times \mathfrak{p}^{n}_{L})+(\mathfrak{p}^{k}_{L}\varpi_{1}\times \mathfrak{p}^{k}_{L}\varpi_{2}). 
\end{equation*}
These are both subalgebras of $M_{2}(\Z_{p})\times \mathcal{O}_{D}$. %

Suppose $n=0$. A quick calculation then shows that we can explicitly %
describe the order as $\mathcal{L}_{x,0}=\mathcal{U}\times \mathcal{O}_{D}$, where 
\begin{equation*}
\mathcal{U}=\{
\begin{pmatrix}
a_{0}+a_{1}+a_{2} & (b_{0}+b_{1}-b_{2})p\\
b_{0}+b_{1}+b_{2} & a_{0}+a_{1}-a_{2}
\end{pmatrix}\,:\,a_{i},b_{i}\in\Z_{p}\}=\{\begin{pmatrix}
c_{1} & c_{2}p\\
c_{3} & c_{4}
\end{pmatrix}\,:\,c_{i}\in\Z_{p}\}.
\end{equation*}
In other words, the linking order is an Iwahori algebra times the full maximal order $\mathcal{O}_{D}$. %
\end{example}

\begin{remark}
Suppose that $x$ is an unramified point with linking order $\mathcal{L}_{x,0}$. Then $\mathcal{L}_{x,0}$ is conjugate to $M_{2}(\Z_{p})\times \mathcal{O}_{D}$. In particular, %
we find that there is an unramified CM point such that $\mathcal{L}_{x,0}=M_{2}(\Z_{p})\times \mathcal{O}_{D}$. The corresponding adjacent points $(y,0)$ will have linking orders conjugate to an Iwahori algebra times $\mathcal{O}_{D}$, see \cref{exa:LinkingOrderExample}.  Changing the conjugacy class of $\mathcal{L}_{x,0}$ will be important for the determination of decomposition groups later on, see \cref{lem:CentralVertex}. %

\end{remark}

\begin{definition}
Let $x$ be a point with complex multiplication by $L\supset\Q_{p}$ and linking order $\mathcal{L}_{x,n}$. %
We set 
$\mathcal{K}_{x,n}=\Delta_{x}(L^{*}) \mathcal{L}^{*}_{x,n}$
and $\mathcal{K}^{1}_{x,n}=\mathcal{K}_{x,n}\cap \mathcal{G}$. The decomposition group $D_{x,n}$ associated to this group is the intersection of $\mathcal{K}^{1}_{x,n}$ with $\SL_{2}(\Z_{p})\subset \mathcal{G}$. 
\end{definition}
By \cite[Theorem 5.1.2]{WS16}, the group $\mathcal{K}^{1}_{x,n}$ is contained in the stabilizer of $\mathcal{Z}_{x,n}$. We will shortly see that this is in fact an equality.   
\begin{definition}
Consider the set of all pairs $(x,n)$, where $x\in \mathcal{M}^{o,ad}_{\infty,\overline{\eta}}(\mathcal{O}_{\C_{p}})$ is a CM point and $n\geq{0}$ is an integer. We call the integer $n$ the level of the pair. We define an equivalence relation on this set of pairs as follows: we have $(x,n)\sim (y,m)$ if $n=m$ and there exists a $\sigma\in \mathcal{K}^{1}_{x,n}$ such that $x^{\sigma}=y$. The corresponding equivalence classes form the vertices of a graph $\mathcal{T}^{o}$. We will often denote an equivalence class by $(x,n)$ again. The incidence relations among these vertices are as follows. The vertices $(x,n)$ and $(y,m)$ are incident (up to exchanging pairs) if one of the following holds: 
\begin{enumerate}
\item $m=0=n$, $x$ is unramified, $y$ is ramified, and $\mathcal{A}_{y}\subset \mathcal{A}_{x}$, where $\mathcal{A}_{x}$ and $\mathcal{A}_{y}$ are the chain orders associated to $x$ and $y$, see \cite[Section 12.1]{BH06} and \cite[Definition 4.2.1]{WS16}. 
\item $(x,n)$ is equivalent to $(y,m+1)$.  
\end{enumerate}  
 The vertices of level zero give rise to a connected subtree $\mathcal{T}^{o}_{0}$. For a fixed CM point $x$, we write $\mathcal{T}^{o}_{(x,0)}$ for the connected graph generated by vertices $(x',n)$ such that $(x',0)\sim (x,0)$. We write $\mathcal{T}^{o}_{x}$ for the connected graph generated by $(x,n)$ for $n\geq{0}$. 
\end{definition}
\begin{remark}
Note that, unlike in the previous sections, the graph $\mathcal{T}^{o}$ is not a metric graph. We will however see in \cref{lem:CompatibilityLemmaQuotient2} that we can identify the images of finite subgraphs of $\mathcal{T}^{o}$ under the projection maps $\pi_{m}$ with finite metric graphs.   %
\end{remark}
The group $\mathcal{G}$ acts naturally on the graph $\mathcal{T}^{o}$ through $(x,n)^{\sigma}=(x^{\sigma},n)$. Moreover, the stabilizer of a vertex $v=(x,n)$ is exactly the group $\mathcal{K}^{1}_{x,n}$. 
\begin{lemma}
Let $v=(x,n)\in\mathcal{T}^{o}$. We have $\mathrm{Stab}(v)=\mathcal{K}^{1}_{x,n}$. 
\end{lemma}
\begin{proof}
We automatically have $\mathrm{Stab}(v)\supset \mathcal{K}^{1}_{x,n}$. Suppose that $(x^{\sigma},n)=(x,n)^{\sigma}=(x,n)$. There then is a $\tau\in\mathcal{K}^{1}_{x,n}$ such that $(x^{\sigma})^{\tau}=x$. But the stabilizer in $\mathcal{G}$ of the CM point $x$ is %
$\Delta_{x}(L^{*})\subset \mathcal{K}^{1}_{x,n}$, so that the desired equality follows.     
\end{proof}

\begin{lemma}
The stabilizer of $\mathcal{Z}_{x,n}$ inside $\mathcal{G}$ is $\mathcal{K}^{1}_{x,n}$. 
\end{lemma}
\begin{proof}
The inclusion $\mathcal{K}^{1}_{x,n}\subset \mathrm{Stab}(\mathcal{Z}_{x,n})$ is \cite[Theorem 5.1.2]{WS16}. For the other inclusion, suppose that $\sigma$ stabilizes $\mathcal{Z}_{x,n}$, but that it is not in $\mathcal{K}^{1}_{x,n}$. The isomorphism $\sigma: \mathcal{Z}_{x,n}\to \mathcal{Z}_{x,n}$ then induces an isomorphism of reductions $\overline{\mathcal{Z}}_{x,n}\to \overline{\mathcal{Z}}_{x,n}$, so that the generic point is fixed. %
We thus see that $\sigma$ necessarily fixes the Berkovich point $[x,n]$.   %
Moreover, $[x,n]$ is in the affinoid $\mathcal{Z}_{v}\subset \mathcal{Z}_{x,n}$ used for the semistable cover in \cite[Section 6.4]{WS16}, since only residue disks corresponding to closed points are removed.  %
By definition, we have that the vertex in $\mathcal{T}^{o}$ corresponding to $(x^{\sigma},n)$ is different. As in \cite{WS16}, we now consider the induced finite-level affinoids $\mathcal{Z}^{m}_{v}$ for $m$ large enough. %
The affinoids $\mathcal{Z}^{m}_{v}$ and $\mathcal{Z}^{m}_{v^{\sigma}}$ intersect, as they contain the image of $[x,n]$. But %
 $v$ and $v^{\sigma}$ are not neighbors in the dual intersection graph $\mathcal{T}^{m}$ corresponding to this semistable cover (as these are given by vertices with lower or higher level $n$), so we see that they cannot intersect. We conclude that the inclusion is an equality.     
\end{proof}

\begin{corollary}\label{cor:DecompositionCentralVertex}
The stabilizer of $[x,n]$ inside $\mathcal{G}$ is $\mathcal{K}^{1}_{x,n}$. The stabilizer of $[x,n]$ in $\SL_{2}(\Z_{p})$ is $D_{x,n}$. 
\end{corollary}

\begin{remark}
To see how the material in this section links to coverings of modular curves, let $Y(M)^{can,\zeta^{i}_{M}}$ be the canonical model of level $M$ with Weil pairing $\zeta^{i}_{M}$ over $\Z_{p}[\zeta_{M}]$ for $M\geq{3}$, see \cite[Section 2.1]{EP21} and \cite[Chapter 13]{KM85}. %
If we take a supersingular point in $Y(M)^{can,\zeta^{i}_{M}}(\overline{\F}_{p})$ corresponding to an edge in $\mathcal{T}_{M,can}$ (this finite metric tree is the inverse image of $\mathcal{T}_{can}$ in $X(M)^{\an}$), then the completed local ring is isomorphic to the ring $A_{0}\simeq \mathcal{O}_{K_{0}}[[u]]$ (after extending the base to $\mathcal{O}_{K_{0}}$). Moreover, the completed local rings corresponding to points in the Drinfeld level-$p^{n}$ model over $Y(M)^{can}$ are isomorphic to the $A_{n}$, see \cite[Proposition 4.7.4]{WS13}.     %
\end{remark}

\subsection{The tower over the outer line segments}\label{sec:WildTower}

We now use the results in \cite{WS16} to recover the decomposition groups over our canonical supersingular tree. The basic idea is as follows: we first show using group-theoretical considerations and a result by Bouw and Wewers that we can pinpoint the image of a single unramified vertex in Weinstein's tree under the projection map to the Berkovich analytification of $X(1)$. When then calculate the overall structure of the quotient tree, and we use this together with a convergence result to determine the location of the remaining vertices.   

\begin{remark}
Throughout this section, we fix an auxiliary level $M\geq{3}$ and a closed point $\overline{x}\in \mathcal{X}(M)^{can,\zeta^{i}_{M}}(\overline{\F}_{p})$. By identifying $\mathcal{X}(M)_{\C_{p}}^{can,\zeta^{i}_{M}}$ with $X(M)$, we see that $\overline{x}$ corresponds to an edge $e_{j}\subset\mathcal{T}_{M,can}$. As in the previous section, we fix an embedding of the $p^{\infty}$-cyclotomic extension into $\C_{p}$. We retain the notation for the natural projection maps %
\begin{equation*}
\pi_{m}: \mathcal{M}^{\mathrm{o},ad}_{\infty,\overline{\eta}}\to \mathcal{M}^{o,ad}_{m,\overline{\eta}},
\end{equation*} 
and for $m=0$ we identify $\mathcal{M}^{o,ad}_{0,\overline{\eta}}$ with an residue disk in $X(M)^{ad}$. Here ${X}(M)^{ad}$ is the adic space corresponding $X(M)$. %
The points of rank one in $X(M)^{ad}$ lie in the Berkovich analytic space $X(M)^{an}$ by definition.  %
\end{remark}

\begin{lemma}\label{lem:CentralVertex}
There is an unramified pair $(x,0)$ such that $D_{x,0}=\SL_{2}(\Z_{p})$.  
\end{lemma}
\begin{proof}
The linking orders $\mathcal{L}_{x,0}$ are all conjugate to $M_{2}(\Z_{p})\times \mathcal{O}_{D}$. We can then find an element $x$ such that $\mathcal{L}_{x,0}=M_{2}(\Z_{p})\times \mathcal{O}_{D}$, which gives %
the equality of decomposition groups. 
\end{proof}
\begin{definition}\label{def:CentralVertex}
Let $x_{c}$ be the point obtained from \cref{lem:CentralVertex}. We call $[x_{c},0]$ a central vertex in $\mathcal{M}^{o,ad}_{\infty,\overline{\eta}}$. 
\end{definition} 
\begin{lemma}
Let $\phi_{M}:X(M)^{\an}\to X(1)^{\an}$ be the natural map. Then  %
$\phi_{M}(\pi_{0}[x_{c},0])=\zeta_{j}$. %
\end{lemma} 
\begin{proof}
Consider the compatible set of points $\{\pi_{n}[x_{c},0]\}$. By \cref{cor:DecompositionCentralVertex}, we directly find that the decomposition group of this compatible set as in \cref{def:DGDefinition} is $\SL_{2}(\Z_{p})$. This implies that 
the decomposition group of $\pi_{1}([x_{c},0])$ over $\pi_{0}([x_{c},0])$ in the covering $X(pM)^{\an}\to X(M)^{\an}$ is $\SL_{2}(\Z/p\Z)$. We now consider the exact sequence of decomposition groups 
 \begin{equation*}
 (0)\to \SL_{2}(\Z/p\Z) \to D_{\overline{x}/x} \to H_{j}
 \end{equation*}
 induced by $X(pM)^{\an}\to X(M)^{\an}\to X(1)^{\an}$, where $\overline{x}=\pi_{1}([x_{c},0])$ and $x=\phi_{M}(\pi_{0}([x_{c},0]))$. From this, we find that the image of ${D}_{\overline{x}/x}$ in $\PSL_{2}(\F_{p})$ is the full group. 
 
 Consider the covering $X(p)^{\an}\to X(1)^{\an}$. For every supersingular elliptic curve over $\overline{\F}_{p}$ with corresponding open disk $U\subset X(1)^{\an}$, there is only one point in the pre-image of $U$ with decomposition group $\PSL_{2}(\F_{p})$ by the results in \cite{BW2004}. We conclude that this point is the image of $\overline{x}$ in $X(p)^{\an}$. We note that the results in \cite{BW2004} also show that the length from the Gauss vertex to $x=\phi_{M}(\pi_{0}[x_{c},0])$ is $b(j)p/(p+1)$.  Suppose for a contradiction that the modified Hasse invariant of $x$ is less than $b(j)p/(p+1)$. Consider the elliptic curve $E_{j,L}$ that generates the tower of modular curves as in \cref{sec:ModularTower}, where $L$ is now the completion of $\C_{p}(j)$ with respect to the valuation induced by $\phi_{M}(\pi_{0}[x_{c},0])$. Since its modified Hasse invariant is less than $b(j)p/(p+1)$, we find by \cref{lem:CanonicalLoci1} that $E_{j,L}$ has a canonical subgroup of order $p$. Since this subgroup is invariant under the decomposition group, we obtain a contradiction, so that the modified Hasse invariant of $x$ is $b(j)p/(p+1)$. Since the length of $x$ to the Gauss vertex is $b(j)p/(p+1)$, we conclude that $x=\zeta_{j}$.   %
\end{proof}

\begin{corollary}
The decomposition group over $\zeta_{j}$ is the image of $\SL_{2}(\Z_{p})\times H_{j}$ in $\PSL_{2}(\hat{\Z})$, up to conjugation. 
\end{corollary}

We now move to the remaining points in $\mathcal{T}_{can}$. We first review the structure of the tree of level zero $\mathcal{T}^{o}_{0}$, which is a barycentric subdivision of the Bruhat-Tits tree over $\Z_{p}$. Here the unramified vertices $(x,0)$ in the tree correspond to the ordinary vertices of the Bruhat-Tits tree, and the ramified vertices $(y,0)$ correspond to the vertices in the subdivided edges of the tree. In terms of the terminology of \cite{WS16}, we have that the chain order $\mathcal{A}_{x}$ of an unramified point $x$ is conjugate to $M_{2}(\Z_{p})$, and the chain order $\mathcal{A}_{y}$ of a ramified point $y$ is conjugate to the Iwahori subalgebra. We then have that there is an edge between $(x,0)$ and $(y,0)$ if and only if there is an inclusion of chain orders $\mathcal{A}_{y}\subset\mathcal{A}_{x}$. For instance, for the central vertex $x_{c}$, there are $p+1$ Iwahori subalgebras, corresponding to elements of $\mathbb{P}^{1}_{\Z}(\F_{p})$. For each of these, there is a unique distinct $M_{2}(\Z_{p})$ conjugate $\mathcal{A}_{x'}$ that contains it. By continuing in this way for the other $p$ Iwahori subalgebras that are contained in $\mathcal{A}_{x'}$, we obtain the Bruhat-Tits tree. From this identification, we immediately see that the action of $\GL_{2}(\Q_{p})$ is preserved. This implies that the quotient by $\SL_{2}(\Z_{p})$ is an infinite line segment, starting with the vertex $(x_{c},0)$. We record this in a lemma.
\begin{lemma}\label{lem:BruhatTitsTree2}
The procedure described above identifies $\mathcal{T}^{o}_{0}$ with a barycentric subdivision of the Bruhat-Tits tree. This identification preserves the action of $\SL_{2}(\Z_{p})$, and the quotient $\mathcal{T}^{o}_{0}/\SL_{2}(\Z_{p})$ is an infinite line segment, starting with the vertex $(x_{c},0)$.  
\end{lemma}
We now consider the quotients of the other graphs $\mathcal{T}^{o}_{x}$ and $\mathcal{T}^{o}_{(x,0)}$. 

\begin{lemma}\label{lem:StructureOuterCMEdges2}
Let $\Gamma\subset \SL_{2}(\Z_{p})$ be a subgroup. 
The graph $\mathcal{T}_{(x,0)}/\Gamma$ is a tree, %
and the graph $\mathcal{T}_{x}/\Gamma$ is a subtree. 
\end{lemma}
\begin{proof}
By definition, the action preserves the level $n$ of a vertex $(x,n)$. The map $\mathcal{T}_{x}\to \mathcal{T}_{x}/\Gamma$ is thus an isomorphism. Suppose that $\sigma(x,n)=(y,n)$ for $\sigma\in\Gamma$. This means that there is a $\tau\in\mathcal{K}^{1}_{\sigma(x),n}$ such that $\tau(\sigma(x))=y$. Since $\mathcal{K}^{1}_{\sigma(x),n}\subset \mathcal{K}^{1}_{\sigma(x),m}$ for $m<n$, we have that $\sigma(x,m)=(y,m)$ for all $m<n$. In other words, the graphs $\mathcal{T}_{x}/\Gamma$ for various $x$ are glued from the bottom, so that the overall graph is again a tree.  
\end{proof}

We now recall that the vertices $(x,n)$ can be identified with valuations $[x,n]$ of rank one. The recipe in \cite[Section 6.3]{WS16} shows that we can also identify the intermediate edges with line segments in the Berkovich analytification, at least on a finite level:   

\begin{lemma}\label{lem:CompatibilityLemmaQuotient2}
We can identify the edges in the graph $\mathcal{T}/\Gamma(p^{n})$ with line segments in the Berkovich analytification of $X(p^{n}M)$.  %
\end{lemma}
\begin{proof}
Let $(x,n)$ and $(x,n+1)$ be vertices with Berkovich points $[x,n]$ and $[x,n+1]$. 
It suffices to show that for some large enough $m$, the smallest connected subspace containing $\pi_{m}([x,n])$ and $\pi_{m}([x,n+1])$ (in terms of Berkovich analytifications) is a line segment, %
since we can then simply take the image of this line segment in lower $X(p^{m}M)^{\an}$, which again forms a line segment since the vertices are not identified. The lemma now follows from the construction in \cite[Section 6.3]{WS16}.     
\end{proof}

\begin{lemma}\label{lem:CompatibilityLemmaQuotient3}
Under the identification in \cref{lem:CompatibilityLemmaQuotient2}, the image of the Bruhat Tits line segment from \cref{lem:BruhatTitsTree2} is mapped to the line segment from $\zeta_{j}$ to the Gauss vertex in $X(1)^{\an}$. Moreover, the unramified vertices $(x,0)$ are mapped to $\zeta_{j,n}$, and the ramified vertices $(y,0)$ are mapped to $\zeta'_{j,n}$.      
\end{lemma}

\begin{proof}
Let $\phi_{M}:X(M)^{\an}\to X(1)^{\an}$ be the projection map. 
Since $\mathcal{Z}_{x,m}\supset \mathcal{Z}_{x,m+1}$ and $\bigcap \mathcal{Z}_{x,m}=\{x\}$ (see \cite[Theorem 5.1.2]{WS16}), we have that the $\phi_{M}(\pi_{0}[x,m])$ converge to $\phi_{M}(\pi_{0}(x))$ for $m\to \infty$ in $\mathbb{P}^{1,\an}_{\C_{p}}$.  
Using \cref{lem:CompatibilityLemmaQuotient2} %
and the fact that the Hasse invariants of formal groups of height $2$ are the numbers $p^{1-n}/(p+1)$ and $p^{1-n}/2$ for ramified and unramified points respectively %
(see \cite[Proposition 4.6]{Wewers2007} and \cite[Lemma 4.8]{CM2007}), we quickly find the desired statement. %
\end{proof}

 \begin{definition}\label{def:BruhatTitsSegment}
 Let $(x_{c},0)$ be the central vertex in \cref{def:CentralVertex} with linking order $M_{2}(\Z_{p})\times \mathcal{O}_{D}$ and let $(y_{c},0)$ be the adjacent ramified point corresponding to the standard Iwahori algebra.
Let $\pi_{D}$ be an element %
of $D$ with $N(\pi_{D})=p$ and set 
 $g=(\begin{pmatrix}
1 & 0 \\
0 & p
\end{pmatrix},\pi_{D})\in(\mathrm{GL}_{2}(\Q_{p})\times D^*)^{\mathrm{det}=N}$. Set %
$x_{n}=g^{n}x_{c}$ and $y_{n}=g^{n}y_{c}$. We call the graph these points generate a Bruhat-Tits line in $\mathcal{T}^{o}$.  
 \end{definition}
 \begin{lemma}\label{lem:AdjacencyLemma}
 The vertex $(x_{1},0)$ is the unramified vertex not equal to $(x_{c},0)$ adjacent to the ramified vertex $(y_{c},0)$. More generally, the vertex $(x_{n},0)$ is the unramified vertex not equal to $(x_{n-1},0)$ adjacent to $(y_{n-1},0)$.  
 \end{lemma}
 \begin{proof}
Let $\mathcal{A}_{x}$ be the chain order of a point $x$ with complex multiplication. We have $\sigma(\mathcal{A}_{x})\sigma^{-1}=\mathcal{A}_{\sigma(x)}$. Note that $\mathcal{A}_{x_{c}}=M_{2}(\Z_{p})$ by construction. Let $g_{1}=\begin{pmatrix} 1 & 0 \\ 0 & p\end{pmatrix}$ and $\sigma=\begin{pmatrix} a & b \\ c & d \end{pmatrix}\in M_{2}(\Z_{p})$. We then find 
\begin{equation*}
g_{1}^{n}\sigma g_{1}^{-n}=\begin{pmatrix} a & b/p^{n} \\ cp^{n} & d \end{pmatrix}.
\end{equation*}
For $n=1$, we have that the standard Iwahori algebra with elements divisible by $p$ in the lower-left corner is contained inside this algebra. We conclude that the vertices $(x_{c},0)$, $(y_{c},0)$ and $(x_{1},0)$ form two edges in the subdivided Bruhat-Tits tree. By conjugating this standard Iwahori algebra, we then also easily find the other inclusions. 
 \end{proof}
 \begin{lemma}
$D_{(x_{n},0)}=I_{n,p}(\Z_{p})$ and $D_{(y_{n},0)}=I_{n,p}(\Z_{p})$. 
 \end{lemma}
\begin{proof}
By \cref{lem:AdjacencyLemma}, we have $(x_{n},0)=(x_{c},0)^{g^{n}}$.
Note that conjugation by $\pi_{D}$ sends $\mathcal{O}_{D}$ to $\mathcal{O}_{D}$. %
We then have
\begin{equation*}
\mathcal{L}_{(x_{n},0)}=g^{n}\mathcal{L}_{(x_{c},0)}g^{-n}=\mathcal{I}_{n}\times \mathcal{O}_{D}. 
\end{equation*} 
Here $\mathcal{I}_{n}=\{\begin{pmatrix}
c_{1} & c_{2}\\
c_{3}p^{n} & c_{4}
\end{pmatrix}\,:\,c_{i}\in\Z_{p}\}$ is the standard upper-triangular Iwahori algebra of level $n$. 
By intersecting $\mathcal{I}_{n}$ with $\SL_{2}(\Z_{p})$, we then find the standard Iwahori group of level $n$. %
Note that this also completely determines $D_{(y_{n},0)}$, since the vertices adjacent to $(y_{n},0)$ are not conjugate by $\SL_{2}(\Z_{p})$.    
\end{proof}

\subsection{The pruned skeleton of a modular curve}\label{sec:PrunedSkeleton}

We now combine the results in the last two sections to find the pruned skeleton of a modular curve. We first define a $\PSL_{2}(\hat{\Z})$-labeling of $\mathcal{T}_{can}$ using the results of the previous sections. 
\begin{definition}\label{def:PSL2Labeling}
Let $\mathcal{T}_{can}$ be the canonical supersingular tree and let $G=\PSL_{2}(\hat{\Z})$. We define a $G$-monodromy labeling using the data in %
\cref{tab:TableDecompositionGroups2}. Here we take the image of the group in the table under the map $\SL_{2}(\hat{\Z})\to \PSL_{2}(\hat{\Z})$.  For every open subgroup $H\subset \PSL_{2}(\hat{\Z})$, this determines a finite connected metric graph $\mathcal{T}_{H,can}$ using the construction in \cref{def:MonodromyTower}.    
\begin{table}[h]
\begin{center}
{\renewcommand{\arraystretch}{1.5}
\begin{tabular}{ |c|c| } 
 \hline
Point or interval in $\mathcal{T}_{can}$ & Monodromy group \\
\hline
$\zeta_\mathrm{G}$ & $\Gamma_{0}(\mathbb{Z}_{p})\times \SL_{2}(\hat{\Z}')$\\ %
\hline
 $I_{n,j}$ & $I_{n,p}(\Z_{p})\times H_{j}$  \\
 \hline
$\zeta_{j}$ & $\SL_{2}(\mathbb{Z}_{p})\times H_{j}$ \\
 \hline
\end{tabular} 
}
 \caption{The groups over the supersingular tree $\mathcal{T}_{can}$ whose images in $\PSL_{2}(\hat{\Z})$ give a monodromy labeling over $\mathcal{T}_{can}$ for the tower of modular curves. %
 We view these groups as %
 subgroups of the product $\SL_{2}(\mathbb{Z}_{p})\times \SL_{2}(\hat{\Z}')$, where $\hat{\Z}'$ is the inverse limit of $\Z/N\Z$ over all $N$ with $(N,p)=1$. %
 \label{tab:TableDecompositionGroups2}
 }
\end{center}
\end{table}
\end{definition}

Before we prove that this data indeed describes the desired pruned skeleta, we give an easy lemma. 

\begin{lemma}\label{lem:NormalizerBorel}
The Borel subgroup $\Gamma_{0}(\Z/p^{n}\Z)\subset \SL_{2}(\Z/p^{n}\Z)$ is self-normalizing. 
\end{lemma}
\begin{proof}
Suppose that $\sigma\notin \Gamma_{0}(\Z/p^{n}\Z)$ and consider its induced action on $\mathbb{P}^{1}_{\Z}(\Z/p^{n}\Z)$ (see \cref{sec:QuotientFunctors}). If $\sigma$ takes $[1:0]$ to $[a:1]$ for some $a\in \Z/p^{n}\Z$, then we can use the classical result for fields. If $\sigma$ takes $[1:0]$ to $[1:dp^{i}]$ for some $i>0$ and $d\in(\Z/p^{n}\Z)^{*}$, then conjugating the Borel gives a subgroup with non-trivial entries in the lower-left corner, as one easily checks. We conclude that $\sigma$ takes $[1:0]$ to $[1:0]$, so that $\sigma\in\Gamma_{0}(\Z/p^{n}\Z)$.    %
\end{proof}

\begin{theorem}\label{thm:MainThm3}
Let $\mathcal{T}_{can}\subset X(1)^{\an}$ be the canonical supersingular tree and let $\phi_{H}: X^{\an}_{H}\to X(1)^{\an}$ be the morphism of modular curves corresponding to %
an open subgroup $H$ of $\PSL_{2}(\hat{\Z})$. Let $\mathcal{T}_{can,H}$ be the metric graph induced by the monodromy labeling in \cref{tab:TableDecompositionGroups2}.   %
Then $\phi^{-1}_{H}(\mathcal{T}_{can})\simeq \mathcal{T}_{can,H}$. Moreover, $\mathcal{T}_{can,H}$ deformation retracts onto the pruned skeleton of $X^{\an}_{H}$.
\end{theorem}

\begin{proof}
We first show that the given monodromy labeling is the monodromy labeling associated to $X(N)^{\an}\to X(1)^{\an}$, where $N=p^{n}M$, $(M,p)=1$ and $M$ is sufficiently large. In particular, we can assume $M\geq{3}$. %
We first choose a section %
$\mathcal{T}_{M,can}\subset X(M)^{\an}$ of $\mathcal{T}_{can}$. Consider the covering $X(p^{n}M)\to X(M)$ with Galois group $\SL_{2}(\Z/p^{n}\Z)$. 
Using \cref{lem:NormalizerBorel} and \cref{lem:DecompositionGroupCentralVertex2}, we can now uniquely characterize the different points over the central vertex of $X(M)^{\an}$ by the corresponding decomposition group in $\SL_{2}(\Z/p^{n}\Z)$. 
For every supersingular $j$-invariant in $X(M)^{\an}$, we choose the image of the Bruhat-Tits line segment as in \cref{def:BruhatTitsSegment}. %
Note that these are all connected to the same central vertex for varying $j$, since their decomposition groups are equal. In other words, we obtain a section $\mathcal{T}_{p^{n}M,can}$ of $\mathcal{T}_{M,can}$, and we know the corresponding decomposition groups by the results in \cref{sec:pAdicTower} and  \cref{sec:WildTower}. We now give the decomposition groups of $\mathcal{T}_{p^{n}M,can}$ for the covering $\pi: X(p^{n}M)^{\an}\to X(1)^{\an}$.  
 For $x$ in $\mathcal{T}_{p^{n}M,can}$ not over the central vertex with supersingular $j$-invariant $j$, we have a commutative diagram  
 \begin{equation*}
 \begin{tikzcd}
 (0)  \arrow[r] & D_{x/\pi_{n}(x)}  \arrow[r]\arrow[d] & D_{x} \arrow[r] \arrow[d] & P(H_{j}) \arrow[d]\arrow[r] & (0) \\
  (0)  \arrow[r] & \SL_{2}(\Z/p^{n}\Z)  \arrow[r] & \PSL_{2}(\Z/p^{n}M\Z) \arrow[r] & \PSL_{2}(\Z/M\Z) \arrow[r] & (0)
 \end{tikzcd}.
 \end{equation*}
Here the horizontal sequences are exact (see \cref{eq:ExactSequences}) and the vertical arrows are injective. 
Note that both $D_{x/\pi_{n}(x)}$ and $H_{j}$ contain $-1$. We now directly find that  
that $D_{x}$ is the image of 
$D_{x/\pi_{n}(x)}\times H_{j}$ in $\PSL_{2}(\Z/p^{n}M\Z)$. %
As we saw in \cref{lem:DecompositionGroupCentralVertex2}, the same proof works over the central vertex after we replace $H_{j}$ with $\SL_{2}(\Z/M\Z)$.   %

By Proposition \ref{pro:DCStoGraphs}, it follows that for any open subgroup $H\subset \PSL_{2}(\hat{\Z})$, we can recover the inverse image of $\mathcal{T}_{can}$ under the map $X^{\an}_{H}\to X(1)^{\an}$ from this monodromy labeling. %
We have to prove that these graphs retract onto the pruned minimal skeleton of the corresponding modular curve. For $X(p^{n}M)^{\an}$,  %
this directly follows from \cref{lem:StructureOuterCMEdges2}, as the trees $\mathcal{T}_{(x,0)}$ are $1$-connected to $(x,0)$. 

To prove the general case, we complete our monodromy labeling for the covering $X(p^{n}M)^{\an}\to X(1)^{\an}$ over $\mathcal{T}_{can}$ to a monodromy labeling over a larger tree. More explicitly, this tree $\mathcal{T}_{can,mod}$ can be interpreted as the image of the tree used in Weinstein's recipe (see \cite[Section 6.4]{WS16}) for a semistable model of $X(p^{n}M)$. Here we again interpret the edges as line segments in the Berkovich analytification as in \cref{lem:CompatibilityLemmaQuotient2}. The new segments in $\mathcal{T}_{can,mod}$ attach to the points $\zeta_{j,n}$ and $\zeta'_{j,n}$, see \cref{lem:CompatibilityLemmaQuotient3}. Since quotients of semistable models are again semistable, we see that the inverse image of this tree gives the full skeleton of any quotient of $X(p^{n}M)$.    %
For the new attached segments, we can easily determine the decomposition groups. Indeed, over these segments the covering $X(M)^{\an}\to X(1)^{\an}$ is completely split by the residual tameness. We conclude that the decomposition groups are simply given by their factor in $\SL_{2}(\Z/p^{n}\Z)\subset P(\SL_{2}(\Z/p^{n}\Z)\times \SL_{2}(\Z/N\Z))$. %
Note that these factors increase monotonically over these attached trees (see \cref{lem:StructureOuterCMEdges2}), including the boundary point where $H_{j}$ is added. 
It directly follows from this monotonicity of the groups $D_{x}$ that the inverse image of the attached parts is a disjoint union of trees, which are automatically $1$-connected to the skeleton. This concludes the proof. 
\end{proof}

\subsection{First Betti numbers and a criterion for potential good reduction}%
\label{sec:PotentialGoodReduction}

As a first application of \cref{thm:MainThm3}, we give a formula for the first Betti numbers of the Berkovich analytifications of modular curves associated to decomposable subgroups. Let $p\neq{2,3}$ be a prime and let $M\geq{3}$ be an integer with $(M,p)=1$ and $n\geq{1}$. These will be fixed throughout this section. %
Recall that if $H$ is decomposable subgroup of $\PSL_{2}(\Z/N\Z)$ for $N=p^{n}M$ with respect to $p$, then there exist two subgroups $H_{p}$ and $H_{M}$ such that the image of $H_{p}\times H_{M}$ in $\PSL_{2}(\Z/N\Z)$ gives $H$. Note that if $H$ is a decomposable subgroup of $\PSL_{2}(\Z/N'\Z)$ for $N'|N$, then we can take inverse images to obtain a decomposable subgroup of $\PSL_{2}(\Z/N\Z)$. In particular, the assumption $M\geq{3}$ is not a restriction.  %

To reconstruct the pruned skeleton of $X_{H}$, we first recall the following.  %
\begin{lemma}\label{lem:ReconstructGeneralizedTree}
Let $j\in{S}$. There is a bijection between the set of edges in $\mathcal{T}_{P(H_{M}),can}$ over $e_{j}\subset \mathcal{T}_{can}$ and the double coset space $H_{j}\bs \PSL_{2}(\Z/M\Z)/P(H_{M})$. In particular, if $j\neq {0,1728}$, then the number of points lying over $j$ is $[\PSL_{2}(\Z/M\Z):P(H_{M})]$.  
\end{lemma}

\begin{definition}
Let $s=|\mathcal{S}|=g(X_{0}(p))+1$ be the number of supersingular elliptic curves over $\mathbb{F}_{p^{2}}$ and let $H_{M}$ be as above. We denote the number of supersingular points associated to $P(H_{M})$ by $$s(H)=\sum_{j\in\mathcal{S}}|H_{j}\bs \PSL_{2}(\Z/M\Z)/P(H_{M})|.$$ We refer to the tree $\mathcal{T}_{P(H_{M}),can}$ %
as the canonical supersingular tree for $H_{M}$.    
\end{definition}

Let $E(H_{M})=P(\SL_{2}(\Z/p^{n}\Z)\times {H_{M}})$ be the inverse image of $P(H_{M})$ in $\PSL_{2}(\Z/M\Z)$. 
Recall from \cref{rem:DecomposableGCoverings} that the chain of subgroups 
\begin{equation*}
(1)\subset H \subset E(H_{M}) \subset \PSL_{2}(\Z/M\Z)
\end{equation*}
gives rise to a chain of coverings
\begin{equation*}
X(M)\to X_{H}\to X_{E(H_{M})}\to X(1).
\end{equation*} 
Here $X_{E(H_{M})}$ is simply $X_{P(H_{M})}$. The covering $X(M)\to X_{E(H_{M})}$ is Galois with Galois group $E(H_{M})$. We will use this relative covering for our calculations.   
We have an %
identification of left coset spaces  
\begin{equation*}
\SL_{2}(\Z/p^{n}\Z)/\epsilon(H_{p})\simeq E(H_{M})/H_{M},
\end{equation*}  
see \cref{rem:DecomposableGCoverings}. 
We will study these cosets in more detail for the functors $\Gamma_{0}$, $\Gamma_{1}$, $\Gamma^{\pm}_{1}$, $\Gamma_{sp}$ and $\Gamma^{+}_{sp}$ in \cref{sec:CosetSchemes}. For now, we use these identifications to obtain %
the following formula that characterizes the topological structure of the Berkovich analytification of a modular curve.   %

\begin{theorem}\label{thm:MainThm2v2}
Let $H\subset \PSL_{2}(\Z/N\Z)$ be a decomposable subgroup with subgroups $H_{p}$ and $H_{M}$ in $\SL_{2}(\Z/p^{n}\Z)$ and $\SL_{2}(\Z/M\Z)$. Let $b_{p}(H)=|\epsilon(H_{p})\backslash \mathbb{P}^{1}_{\Z}(\Z/p^{n}\Z)|$ and let $s(H)$ be the number of supersingular $j$-invariants associated to the modular curve $X_{P(H_{M})}$.  Then 
\begin{equation*}
\beta_{1}(\Sigma(X_{H}))=(s(H)-1)(b_{p}(H)-1).
\end{equation*}
\end{theorem}
\begin{proof}
We calculate the Euler characteristic of $\Sigma^{pr}(X_{H})$. %
The decomposition groups of the coverings $X(M)\to X_{P(H_{M})}$ over $\mathcal{T}_{P(H_{M}),can}$ are symmetric in the direction of each leaf of $\mathcal{T}_{P(H_{M}),can}$. %
For $i\geq{1}$, write $c_{i}=|\epsilon(H_{p})\bs \SL_{2}(\Z/p^{n}\Z)/I_{i}|$ for the contributions of the Iwahori double coset classes. For $i=0$, we set $I_{0}=\Gamma_{0}(\Z/p^{n}\Z)$, so that $b_{p}(H)=|\epsilon(H_{p})\bs \SL_{2}(\Z/p^{n}\Z)/I_{0}|=c_{0}$. We then find
\begin{align*}
\#{E(\Sigma(X_{H}))}&=s(H)(c_{0}+c_{1}+...c_{n}),\\
\#{V(\Sigma(X_{H}))}&=s(H)(c_{1}+...+c_{n})+s(H)+c_{0},
\end{align*}
which directly gives the desired formula. 
\end{proof}

\begin{corollary}\label{cor:SpecialPrimesPGR}
Let $p\in \{5,7,13\}$ and let $\overline{H}\subset \PSL_{2}(\Z/p^{n}\Z)$ be a subgroup with induced modular curve $X_{H}$ and Jacobian $J_{H}$. Then $J_{H}$ has potential good reduction %
 at any prime. \end{corollary}
\begin{proof}
As is well known, these are exactly the primes ($\neq{2,3}$) for which $s=1$. %
We conclude using %
\cref{thm:MainThm2v2}.
\end{proof}
\begin{remark}
We note that the same argument can be used for $p=2,3$ as soon as one has suitable generalizations of the results in \cite{WS16} to fields with %
residue characteristic $2$, and \cite{BW2004} or \cite{EP21} to fields with residue characteristic $3$. 
\end{remark}
\begin{example}
Let $X_{sp}(13)$ be the modular curve associated to the split Cartan subgroup as in \cite[Section 6.2]{BDMTV2019} (we called this the standard split torus here). From our theorem above, it directly follows that the Jacobian has potential good reduction. To show that the curve has potential good reduction, one can then calculate the quotients of the curves in \cite{WS16} to find a single vertex with positive genus. This is the curve corresponding to $[y,0]$, where $y$ is a ramified CM point with Hasse invariant $1/(p+1)$.     
\end{example}

\begin{proposition}\label{pro:PotentialGoodReduction}
Suppose that $H_{2}\subset H_{1}$ are two decomposable subgroups of $\PSL_{2}(\Z/M\Z)$ with invariants $b_{p}(H_{i})$ and $s(H_{i})$ not equal to $1$. Let $J_{i}$ be the Jacobians of the corresponding modular curves and let $p^{*}:J_{1}\to J_{2}$ be the natural pullback map. Then $J_{2}/p^{*}(J_{1})$ has potential good reduction over $p$ if and only if $s(H_{1})=s(H_{2})$ and $b_{p}(H_{1})=b_{p}(H_{2})$.  
\end{proposition}
\begin{proof}
Let $J_{2}=p^{*}(J_{1})\oplus J'_{2}$ be the decomposition up to isogeny provided by Poincar\'{e}'s  theorem. We can assume that the indicated abelian varieties all have semistable reduction. We now note that $s(\cdot)$ and $b_{p}(\cdot)$ are monotone functions, in the sense that if $H_{2}\subset H_{1}$, then $s(H_{1})\leq s(H_{2})$ and $b_{p}(H_{1})\leq b_{p}(H_{2})$. Indeed, this easily follows from the inclusions $\epsilon(H_{1,p})\supset \epsilon(H_{2,p})$ and $H_{1,N}\supset H_{2,N}$. %
Since the toric rank is additive, we conclude %
using  \cref{thm:MainThm2v2}. %
\end{proof}

\begin{corollary}(Deligne-Rapoport)
The abelian variety $J_{1}(p)/J_{0}(p)$ has potential good reduction everywhere. 
\end{corollary}
\begin{proof}
We have $\SL_{2}(\F_{p})/\Gamma_{0}(\F_{p})\simeq \mathbb{P}^{1}_{\Z}(\F_{p})$ (see \cref{sec:QuotientFunctors} for generalizations), and the 
orbit space $\Gamma_{0}(\mathbb{F}_{p})\bs \mathbb{P}^{1}_{\Z}(\mathbb{F}_{p})$ can be represented by $[0:1]$ and $[1:0]$. The corresponding orbits under the %
action of $\Gamma_{1}(\mathbb{F}_{p})$ are still of lengths $p$ and $1$ respectively, so we conclude using \cref{pro:PotentialGoodReduction}. 
\end{proof}
\begin{remark}
The other cases associated to subgroups $H\subset \mathbb{F}_{p}^{*}$ in \cite[Section 3, Page 253]{DR73}  also directly follow from this corollary since the invariant $b_{p}$ is a monotone function. %
\end{remark}

To find direct factors of Jacobians of modular curves with potential good reduction, one can also construct isogenies 
using various other maps such as Hecke operators. We show how this leads to a short proof of the potential good reduction of the variety %
$J_{0}(p^{n})^{\mathrm{new}}$ of level $p^{n}$-newforms. %

\begin{corollary}
Let $n>1$ and $p>2,3$. The abelian variety $J_{0}(p^{n})^{\mathrm{new}}$ has potential good reduction everywhere. 
\end{corollary} 
\begin{proof}
This can be deduced from \cite[Theorem 14.7.2]{KM85}, as mentioned in the proof of \cite[Theorem 9.4]{CM2010}. In the latter, it was used to obtain a formula for the toric rank of $J_{0}(p^{n}N)$. %
We reverse the argument here and obtain the potential good reduction of $J_{0}(p^{n})^{\mathrm{new}}$ from our formulas for the toric rank of $J_{0}(p^{n})$ (with apologies to the reader for some forward-referencing). %
We have the exact sequence 
\begin{equation*}
(0)\to J_{0}(p^{n-1})\to J_{0}(p^{n})^2\to J_{0}(p^{n+1})\to J_{0}(p^{n})^{\mathrm{new}} \to (0)
\end{equation*}
up to inverting isogenies. We can assume that $s>1$, since otherwise the statement is contained in \cref{cor:SpecialPrimesPGR}. In \cref{pro:BorelDCS1}, we show that $b_{p}=2n$. %
Using the additivity of the toric rank, we now compute %
\begin{align*}
t(J_{0}(p^{n})^2/J_{0}(p^{n-1}))/(s-1)&=4n-2-(2(n-1)-1)=2n+1,\\
t(J_{0}(p^{n+1})/J_{0}(p^{n})^2)/(s-1)&=2(n+1)-1-(2n+1)=0,
\end{align*}
so that $J_{0}(p^{n})^{\mathrm{new}}$ indeed has potential good reduction.
\end{proof}

\section{Explicit skeleta for subgroup schemes of $\SL_{2}$}\label{sec:CosetSchemes}

In this section we show how to obtain the pruned skeleta of modular curves associated to various subgroup schemes of $\SL_{2}$, see \cref{def:StandardFunctors} for a list. The most important part here comes from $\Gamma(\Z/p^{n}\Z)$. We show that the corresponding left-coset spaces can be captured in terms of the $\Z/p^{n}\Z$-valued points of certain schemes. In \cref{sec:CompX0N} we use these group-theoretic results to determine the component groups of the curves $X_{0}(N)$. %

\subsection{Coset schemes}\label{sec:QuotientFunctors}        

We first introduce the following coset schemes, which describe the coset functors associated to the functors  %
$\Gamma_{0}$, $\Gamma_{1}$, $\Gamma^{\pm}_{1}$, $\Gamma_{sp}$ and $\Gamma^{+}_{sp}$ from \cref{sec:GroupNotation}. %
\begin{definition}\label{def:CosetSchemes}
The coset schemes associated to $\Gamma_{0}$, $\Gamma_{1}$, $\Gamma^{\pm}_{1}$, $\Gamma_{sp}$ and $\Gamma_{sp}^{+}$ respectively are %
\begin{align*}
\mathcal{F}_{0}&=\mathbb{P}^{1}_{\Z},\\
\mathcal{F}_{1}&=\mathbb{A}^{2}_{\mathbb{Z}}\backslash\{0\},\\
\mathcal{F}^{\pm}_{1}&=\mathcal{F}_{1}/\langle -1\rangle,\\%
\mathcal{F}_{sp}&=(\mathbb{P}^{1}_{\mathbb{Z}})^{2}\backslash \Delta,\\
\mathcal{F}^{+}_{sp}&=\mathcal{F}_{sp}/\langle \pi \rangle.
\end{align*}
Here $\mathbb{A}^{2}_{\mathbb{Z}}\backslash\{0\}=\Spec(\Z[X,Y])\bs \{(X,Y)\}$, $-1$ acts on $\mathcal{F}_{1}$ diagonally, and 
$\pi$ is the restriction of the natural involution on $(\mathbb{P}^{1}_{\Z})^2$ defined on $R$-valued points by sending a pair $(M_{1},M_{2})$ of locally free of rank 1 $R$-submodules in $R^2$
 to $(M_{2},M_{1})$. 
\end{definition}

We first review the $R$-valued points of the schemes in \cref{def:CosetSchemes}. The $R$-valued points of $\mathcal{F}_{1}$ consist of pairs $(r_{1},r_{2})\in{R^{2}}$ that generate the unit ideal in $R$. Equivalently, we locally have that either $r_{1}$ or $r_{2}$ is invertible. In particular, if $R$ is local then either $r_{1}$ or $r_{2}$ is a unit. The $R$-valued points of $\mathcal{F}^{\pm}_{1}$ are pairs $(r_{1},r_{2})$ as above, up to an identification $(r_{1},r_{2})\sim (-r_{1},-r_{2})$. For $\mathbb{P}^{1}_{\Z}$, we have a similar interpretation for local rings $R$: we define a primitive pair $(x,y)\in{R^{2}}$ to be a pair such that either $x$ or $y$ is a unit. %
We say two primitive pairs $(x,y)$ and $(v,w)$ are equivalent if there exists a unit $u\in{R}$ such that $uv=x$ and $uw=y$. An $R$-valued point of $\mathbb{P}^{1}_{\Z}$ is an equivalence class of primitive pairs. An equivalence class is denoted by $[x:y]$, as in the case of fields. %
Let $R$ be a local ring as before. Through the open immersion 
\begin{equation*}
\mathcal{F}_{sp}\to (\mathbb{P}^{1}_{\Z})^2,
\end{equation*}
we can view a point $P\in \mathcal{F}_{sp}(R)$ as a pair $(P_{1},P_{2})$, where $P_{1}=[x_{1}:y_{1}]$ and $P_{2}=[x_{2}:y_{2}]$. We write $\overline{P}_{i}$ for the reductions of these points, which are obtained by %
composing with the map $\Spec(k)\to \Spec(R)$. Note that a pair $(P_{1},P_{2})$ is in $\mathcal{F}_{sp}(R)$ exactly when $\overline{P}_{1}\neq \overline{P}_{2}$. Indeed, assume for simplicity that $y_{1}$ and $y_{2}$ are invertible, so that we can represent $P_{i}$ by $[x_{i}:1]$. The scheme $\mathcal{F}_{sp}$ is locally affine, and near $P_{i}$ we have a local chart given by $\Z[x_{1},x_{2}][(x_{1}-x_{2})^{-1}]$. The desired description directly follows. This similarly shows that points in $\mathcal{F}^{+}_{sp}(R)$ are pairs $(P_{1},P_{2})$ with $\overline{P}_{1}\neq \overline{P}_{2}$ up to permutation. We note that an analogous description %
for the $R$-valued points of the schemes above can also be given for direct products of local rings $R=R_{1}\times ...\times R_{n}$, since projective modules over $R$ are free.  

\begin{remark}
We will assume for the remainder of this section that $R$ is a local ring or a direct product of local rings. 
\end{remark}

The group $\SL_{2}(R)$ acts on the $R$-valued points of each of the coset schemes above in the usual way. For instance, 
if 
\begin{equation*}
\sigma=
\begin{pmatrix}
a & b \\
c & d
\end{pmatrix},
\end{equation*}
then
\begin{align*}
\sigma([x:y])&=[ax+by:cx+dy].
\end{align*}
It now directly follows that the schemes above describe the correct cosets: \begin{lemma}\label{lem:CosetsInjection}
Let $R=R_{1}\times ...\times R_{n}$ be a direct product of local rings and let $\Gamma$ be one of the functors in \cref{def:StandardFunctors} with coset scheme $\mathcal{F}_{\Gamma}$ as in \cref{def:CosetSchemes}. There is a bijection
\begin{equation*}
\SL_{2}(R)/\Gamma(R)\simeq \mathcal{F}_{\Gamma}(R). 
\end{equation*}
\end{lemma}
\begin{proof}
By assumption, projective modules over these rings are free. We will indicate how to obtain the identifications in the lemma when $R$ is a local ring, the general case is similar. One first shows that the 
action of $\SL_{2}(R)$ on the coset schemes %
is transitive using the explicit representation of $R$-valued points given before the lemma. 
The proof is almost exactly the same as in the case of fields.  %
The stabilizers of $[1:0]$, $(1,0)$, $(1,0)/\sim$, $([1:0],[0:1])$ and $([1:0],[0:1])/\sim$ are then $\Gamma_{0}(R)$, $\Gamma_{1}(R)$, $\Gamma^{\pm}_{1}(R)$, $\Gamma_{sp}(R)$ and $\Gamma^{+}_{sp}(R)$ respectively, so we obtain the desired statement from the orbit-stabilizer theorem. We leave the details to the reader. 
\end{proof}

\begin{remark}

The bijection in Lemma \ref{lem:CosetsInjection} does not extend to a bijection for all commutative rings. Consider for instance the functor $\Gamma_{0}$. We obtain an injection $\SL_{2}(R)/\Gamma_{0}(R)\to \mathcal{F}_{0}(R)$ by considering the action of $\SL_{2}(R)$ on $[0:1]$ as before, but we note that the $R$-modules constructed in this way are all \emph{free}. Let $R$ be a ring with a non-free invertible module $M$ that admits an embedding $M\subset R^{2}$. %
Then this by definition gives a non-free point of $\mathbb{P}^{1}_{\Z}(R)$, and this is not in the image of $\SL_{2}(R)/\Gamma_{0}\to \mathbb{P}^{1}_{\Z}(R)$. Using this, one can show that $\SL_{2}/\Gamma_{0}$ is not a sheaf in the Zariski topology, so that 
we do not obtain an identification of functors $\SL_{2}/\Gamma_{0}= \mathbb{P}^{1}_{\mathbb{Z}}$. The fppf-sheafification associated to %
$\SL_{2}/\Gamma_{0}$ however is $\mathbb{P}^{1}_{\Z}$, but we will not need this, since our main focus lies on $R=\Z/N\Z$. %
\end{remark}

\subsection{Borel double coset spaces}\label{sec:DCSCentralVertex}

We now find representatives for the double coset spaces arising from \cref{thm:MainThm3}. We will start by considering the double coset spaces in $\SL_{2}$; the corresponding double coset spaces in $\PSL_{2}$ can be obtained without too much trouble from this data.

We start with the following elementary observation. Consider the left-action of $\SL_{2}(R)$ on the left coset spaces $\SL_{2}(R)/\Gamma_{0}(R)$, $\SL_{2}(R)/\Gamma_{1}(R)$, $\SL_{2}(R)/\Gamma^{\pm}_{1}(R)$, $\SL_{2}(R)/\Gamma_{sp}(R)$ and $\SL_{2}(R)/\Gamma^{+}_{sp}(R)$, and the coset schemes $\mathcal{F}_{0}(R)$, $\mathcal{F}_{1}(R)$, $\mathcal{F}^{\pm}_{1}(R)$, $\mathcal{F}_{sp}(R)$ and $\mathcal{F}^{+}_{sp}(R)$. It is then easy to see that these actions are compatible with the identifications from \cref{lem:CosetsInjection}. We use this to calculate the double coset spaces $\Gamma_{0}(R)\bs \SL_{2}(R)/\Gamma(R)$, where $\Gamma$ is one of our five functors. %

\begin{proposition}\label{pro:BorelDCS1}
Write $R=\Z/p^{n}\Z$ for $n\geq{1}$ and $U=R^{*}$. Then
\begin{align*}
|\Gamma_{0}(R)\backslash \mathcal{F}_{0}(R)|&=2n,\\
|\Gamma_{0}(R)\backslash \mathcal{F}_{1}(R)|&=2p^{(n-1)/2} \text{ if n is odd,}\\
|\Gamma_{0}(R) \bs \mathcal{F}_{1}(R)|&=p^{n/2}+p^{n/2-1} \text{ if n is even,}\\
|\Gamma_{0}(R) \bs \mathcal{F}_{1}(R)|&=|\Gamma_{0}(R) \bs \mathcal{F}^{\pm}_{1}(R)|,\\
|\Gamma_{0}(R)\backslash {\mathcal{F}}_{sp}(R)|&=4n,\\
|\Gamma_{0}(R)\bs \mathcal{F}^{+}_{sp}(R)|&=2n \text{ if }p\equiv{1}\bmod{4},\\
 |\Gamma_{0}(R)\bs \mathcal{F}^{+}_{sp}(R)|&=2n+1  \text{ if }p\equiv{3}\bmod{4}.
\end{align*} 
Let $r$ be a fixed nonsquare in $U$. The representatives of the double coset space $\Gamma_{0}(R)\bs \SL_{2}(R)/\Gamma_{0}(R)$ in $\mathcal{F}_{0}(R)=\mathbb{P}^{1}_{\Z}(R)$, together with the order of the corresponding $\Gamma_{0}(R)$-orbit, are given in \cref{tab:DCSpacesG0}. Here $d=1$ or $d=r$.  
\begin{table}[h]
\begin{center}
\begin{tabular}{ |c|c| } 
 \hline
Representative of the double coset & Order of the $\Gamma_{0}(R)$-orbit \\
\hline
$[0:1]$ & $p^{n}$ \\
\hline
$[1:0]$ & $1$\\
\hline
$[1:dp^{i}]$ & $\phi(p^{n-i})/2$\\
\hline
\end{tabular}
\end{center}
\caption{\label{tab:DCSpacesG0}The representatives of the double coset spaces corresponding to $\Gamma_{0}(R)$ and $\Gamma_{0}(R)$. The second column gives the order of the corresponding $\Gamma_{0}(R)$-orbit.}
\end{table}

The representatives of the double coset space $\Gamma_{0}(R)\bs \SL_{2}(R)/\Gamma_{1}(R)$ in $\mathcal{F}_{1}(R)$, together with the order of the corresponding $\Gamma_{0}(R)$-orbit, are given in \cref{tab:DCSpacesG1}. Here $u$ can be viewed as an element of $U/U_{k}\cdot U_{n-k}\simeq (\Z/p^{k_{0}}\Z)^{*}$, where $k_{0}=\mathrm{min}\{k,n-k\}$.

\begin{table}[h]
\begin{center}
\begin{tabular}{ |c|c| } 
 \hline
Representative of the double coset & Order of the $\Gamma_{0}(R)$-orbit \\
\hline
$(0,1)$ & $\phi(p^{n})p^{n}$ \\
\hline
$(1,0)$ & $\phi(p^{n})$\\
\hline
$(1,up^{k})$  & $\phi(p^{n})$ for $n-k\leq{k}$ \\
\hline 
$(1,up^{k})$  & $\phi(p^{n})p^{n-2k}$ for $k<n-k$\\
\hline
\end{tabular}
\end{center}
\caption{\label{tab:DCSpacesG1}The representatives of the double coset spaces corresponding to $\Gamma_{0}(R)$ and $\Gamma_{1}(R)$. The second column gives the order of the corresponding $\Gamma_{0}(R)$-orbit. }
\end{table}
For $\Gamma^{\pm}_{1}$, the action of $-1$ preserves the $\Gamma_{0}(R)$-orbits in $\mathcal{F}_{1}(R)$. The orders of the first two orbits in \cref{tab:DCSpacesG1} are multiplied by $1/2$, the last two are the same.  %

The representatives of the double coset space $\Gamma_{0}(R)\bs \SL_{2}(R)/\Gamma_{sp}(R)$ in $\mathcal{F}_{sp}(R)$, %
together with the order of the corresponding $\Gamma_{0}(R)$-orbit, are given in \cref{tab:DCSpacesSP}. Here $d\in\{1,r\}$ %
and $0<k<n$.  %
\begin{table}[h]
\begin{center}
\begin{tabular}{ |c|c| } 
 \hline
Representative of the double coset & Order of the $\Gamma_{0}(R)$-orbit \\
\hline
$([0:1],[1:r]])$ & $\phi(p^{n})p^{n-k}/2$ \\
\hline 
$([0:1],[1:dp^{k}])$ & $\phi(p^{n})p^{n-k}/2$ \\
\hline
$([0:1],[1:0])$ & $p^{n}$\\
\hline
$([1:-1],[0:1]])$ & $\phi(p^{n})p^{n-k}/2$ \\
\hline
$([1:dp^{k}],[0:1])$ & $\phi(p^{n})p^{n-k}/2$\\
\hline
$([1:0],[0:1])$ & $p^{n}$\\
\hline
\end{tabular}
\end{center}
\caption{\label{tab:DCSpacesSP}The representatives of the double coset spaces corresponding to $\Gamma_{0}(R)$ and $\Gamma_{sp}(R)$. The second column gives the order of the corresponding $\Gamma_{0}(R)$-orbit.}
\end{table}

For the representatives of $\Gamma_{0}(R)\bs \SL_{2}(R)/\Gamma^{+}_{sp}(R)$ in $\mathcal{F}^{+}_{sp}(R)$, suppose first that $p\equiv{1}\bmod{4}$. Then the representatives are given by the first four representatives in the table above. The sizes of the second and third $\Gamma_{0}(R)$-orbits are the same, and the sizes of the first and fourth orbits are $\phi(p^{n})p^{n-k}/4$. If $p\equiv{3}\bmod{4}$, then the representatives are given by the first three in the table above, and the sizes of the $\Gamma_{0}(R)$-orbits are the same.

\end{proposition}
\begin{proof}
We start with $\mathcal{F}_{0}(R)$. The orbit of $[0:1]$ is $[a:1]$ for $a\in{R}$, and $[1:0]$ is fixed under the action of $\Gamma_{0}(R)$. We next consider elements of the form $[1:up^{i}]$, where $u\in{U}$ and $i>0$. %
We call these elements of valuation $i$. Write $U_{n-i}\subset U$ for the subgroup of elements that are $1$ modulo $p^{n-i}$.
Note that we can identify the set of $[1:up^{i}]$ of valuation $i$ with $U/U_{n-i}\simeq (\Z/p^{n-i}\Z)^{*}$.
For $[1:p^{i}]$ and $\sigma=\begin{pmatrix}
a & b \\
0 & d
\end{pmatrix}$, we find 
\begin{equation*}
\sigma
\cdot [1:p^{i}]=[a+bp^{i}:dp^{i}]=[1+bdp^{i}:d^2p^{i}]=[1:d^2p^{i}/(1+bdp^{i})].
\end{equation*}    
Note that the $p$-adic valuation of the second coordinate is preserved. We thus have an action of %
$\Gamma_{0}(R)$ on $U/U_{n-i}\simeq (\Z/p^{n-i}\Z)^{*}$: %
\begin{equation*}
\sigma (r)=d^2/(1+bdp^{i})r.
\end{equation*}  
The element $1+bdp^{i}$ is a square by Hensel's lemma. We thus see that the action of $\Gamma_{0}(R)$ is multiplication by a square, and all squares are in fact attained by taking $b=0$. From this, we directly find that $[1:p^{i}]$ and $[1:rp^{i}]$ give two separate orbits, each of order $\phi(p^{n-i})/2$. This gives all the necessary data for $\Gamma_{0}(R)$. 

To find the double cosets associated to $\Gamma_{1}$ and $\mathcal{F}_{1}(R)$, we will reverse the order of the groups and instead calculate the action of $\Gamma_{1}(R)$ on $\mathbb{P}^{1}_{\Z}(\Z/p^{n}\Z)$. We then reverse the order again to obtain explicit representatives in $\mathcal{F}_{1}(R)$. 

Let $\sigma=\begin{pmatrix}
1 & b\\
0 & 1
\end{pmatrix}$ and let $u\in U$. %
We then have 
\begin{equation*}
\sigma([1:up^{k}])=[1+ubp^{k}:up^{k}].
\end{equation*}
Since $1+ubp^{k}$ can represent arbitrary elements of $U_{k}$, so can its inverse. We thus find that the different orbits $[1:up^{k}]$ are classified by the group $U/(U_{n-k}\cdot U_{k})$. Let $k_{0}=\mathrm{min}\{k,n-k\}$. Then $U/U_{n-k}\cdot U_{k}\simeq (\Z/p^{k_{0}}\Z)^{*}$, so these give $\phi(p^{k_{0}})$ %
 different orbits. To translate these back, note that the element $[1:up^{k}]$ can be represented by the matrix $\begin{pmatrix} 1 & 0 \\ up^{k} & 1\end{pmatrix}$, whose inverse is $\begin{pmatrix} 1 & 0 \\ -up^{k} & 1\end{pmatrix}$. This element in turn corresponds to $(1,-up^{k})\in \mathcal{F}_{1}(R)$, giving the desired double coset representatives.   

To calculate the orders of their orbits under $\Gamma_{0}(R)$, we will calculate the orders of their stabilizers. The stabilizers of $(0,1)$ and $(1,0)$ are easily seen to be of orders $1$ and $p^{n}$. For $(1,up^{k})$, let $\sigma=\begin{pmatrix}
a & b \\
0 & d
\end{pmatrix}$
be an element in the stabilizer. We then have the equations 
\begin{align*}
a+bup^{k}&=1,\\
udp^{k}&=up^{k},\\
ad&=1.
\end{align*}
 Suppose first that $n-k\leq k$. Choosing $b$ freely, this determines $a$. Moreover, the third equation together with $n-k\leq{k}$ implies %
 $d\equiv{1}\bmod p^{n-k}$, so that all three equations are satisfied. We thus have $|\mathrm{Orb}_{\Gamma_{0}(R)}(1,up^{k})|=\phi(p^{n})$.     Now suppose that $k<n-k$. It is easy to see from the equations above that necessary and sufficient conditions on $b\in\Z/p^{n}\Z$ are given by $b\equiv{0}\bmod p^{n-2k}$. In other words, the stabilizer has order $p^{2k}$. This gives $|\mathrm{Orb}_{\Gamma_{0}(R)}(1,up^{k})|=\phi(p^{n})p^{n-2k}$.
 
 For $\Gamma^{\pm}_{1}(R)$, one checks that the $\Gamma_{0}(R)$-stabilizers of the $(1,up^{k})$ are the same as for $\Gamma_{1}(R)$. For $(0,1)$ and $(1,0)$ however, they increase by a factor $2$.

We now apply the same technique to $\Gamma_{sp}(R)$. We have 
\begin{equation*}
\begin{pmatrix}
t & 0 \\
0 & t^{-1}
\end{pmatrix}
\cdot [x:y]=[tx:t^{-1}y]=[t^2x:y]=[x:t^2y],
\end{equation*} 
where $t\in{R^{*}}$. A set of representatives for the orbits under this action are given by
\begin{equation*}
[1:0], [0:1], [dp^{i}:1],[1:dp^{i}], [1:1], [r:1].
\end{equation*}
Here $i$ is assumed to be larger than zero, and $d\in\{1,r\}$. By representing these elements using matrices (in the sense that $v=\sigma([1:0])$ for some $\sigma$), and translating these to elements of $\SL_{2}(R)/T(R)$, we then find the representatives in the table. As an example, we now calculate the stabilizer of $([0:1],[1:p^{i}])$ for $i>0$. From the first equation, we obtain $b=0$. The second then gives $d^2\equiv{1}\bmod p^{n-i}$, and conversely every element $d$ that satisfies $d^2\equiv{1}\bmod{p^{n-i}}$ gives an element of the stabilizer. We conclude that the stabilizer has order $2p^{i}$. The other stabilizers follow a similar pattern.       

For $\mathcal{F}^{+}_{sp}(R)$, we calculate the action of $\tau=\begin{pmatrix} 0 & -1 \\ 1 & 0\end{pmatrix}$ on the representatives in $\mathbb{P}^{1}_{\Z}(R)$. %
If $p\equiv{1}\bmod{4}$, then $-1$ is a square modulo $p$, and thus in $R^{*}$. Similarly, if $p\equiv{3}\bmod{4}$, then $-r$ is a square in $R^{*}$. Using this, one immediately finds the desired representatives, and the stabilizers also easily follow.  
\end{proof}

\begin{remark}
Using \cref{thm:MainThm2v2}, we see that this calculation completely determines the topological structure of the Berkovich analytification of these modular curves.  
\end{remark}

\subsection{Hecke-Iwahori double coset spaces}\label{sec:DCSOuterParts}

We now show how to find the Hecke-Iwahori double coset spaces arising from \cref{thm:MainThm3}. %
This allows us to find the edge lengths of the edges in the pruned Berkovich skeleton. 

\begin{lemma}\label{lem:IwahoriToBorel}
Let $H_{n}$ be a subgroup of $\SL_{2}(\Z/p^{n}\Z)$ and let $k\leq n$. Write $H_{k}$ for the image of $H_{n}$ in $\SL_{2}(\Z/p^{k}\Z)$.  
The reduction maps $\Z/p^{n}\Z\to \Z/p^{k}\Z$ and the inclusions $I_{n}(\Z/p^{n}\Z)\subset I_{n-1}(\Z/p^{n}\Z)\subset ...\subset \SL_{2}(\Z/p^{n}\Z)$ induce a commutative diagram 
\begin{equation*}
\begin{tikzcd}
	I_{n}(\Z/p^{n}\Z) \bs \SL_{2}(\Z/p^{n}\Z)/H_{n}& \Gamma_{0}(\Z/p^{n}\Z) \bs \SL_{2}(\Z/p^{n}\Z)/H_{n} \\
	I_{n-1}(\Z/p^{n}\Z) \bs \SL_{2}(\Z/p^{n}\Z)/H_{n}& \Gamma_{0}(\Z/p^{n-1}\Z) \bs \SL_{2}(\Z/p^{n-1}\Z)/H_{n-1} \\
	I_{n-2}(\Z/p^{n}\Z) \bs \SL_{2}(\Z/p^{n}\Z)/H_{n} & \Gamma_{0}(\Z/p^{n-2}\Z) \bs \SL_{2}(\Z/p^{n-2}\Z)/H_{n-2}\vspace{0.5cm}\\
	{\vdots} & {\vdots} \vspace{0.5cm}\\
	I_{1}(\Z/p^{n}\Z) \bs \SL_{2}(\Z/p^{n}\Z)/H_{n} & \Gamma_{0}(\Z/p\Z) \bs \SL_{2}(\Z/p\Z)/H_{1}
		\arrow[from=1-1, to=2-1]
	\arrow[from=1-1, to=1-2]
	\arrow[from=1-2, to=2-2]
	\arrow[from=2-1, to=3-1]
	\arrow[from=2-2, to=3-2]
	\arrow[from=2-1, to=2-2]
	\arrow[from=3-1, to=3-2]
	\arrow[from=5-1, to=5-2]
	\arrow[dotted, from=3-2, to=4-2]
	\arrow[dotted, from=3-1, to=4-1]
		\arrow[dotted, from=4-2, to=5-2]
	\arrow[dotted, from=4-1, to=5-1]
\end{tikzcd}
\end{equation*}
The horizontal maps in this diagram are bijections.    
\end{lemma}
\begin{proof}
Note that for any $0\leq k \leq n$,  the natural reduction map
\begin{equation*}
\mathrm{red}:\SL_{2}(\Z/p^{n}\Z)\to \SL_{2}(\Z/p^{k}\Z)
\end{equation*}
is a group homomorphism. The horizontal maps and the vertical maps on the right are then given by reducing a representative $\sigma$ of the double coset space. %
The fact that $\mathrm{red}(\cdot)$ is a group homomorphism easily shows that these maps are well defined, and the commutativity also easily follows.  
Let $k\leq n$ and consider the horizontal map 
\begin{equation}\label{eq:LinkIwahoriBorel}
I_{k}(\Z/p^{n}\Z) \bs \SL_{2}(\Z/p^{n}\Z)/H_{n}\to \Gamma_{0}(\Z/p^{k}\Z) \bs \SL_{2}(\Z/p^{k}\Z)/H_{k}.
\end{equation}
We first note that there is a bijection $\SL_{2}(\Z/p^{n}\Z)/I_{k}(\Z/p^{n}\Z)\to \mathbb{P}^{1}_{\Z}(\Z/p^{k}\Z)$. Indeed, the action of $\SL_{2}(\Z/p^{n}\Z)$ on $\mathbb{P}^{1}_{\Z}(\Z/p^{k}\Z)$ is still transitive, and the stabilizer of $[1:0]$ is exactly $I_{k}(\Z/p^{n}\Z)$.  The first set in \cref{eq:LinkIwahoriBorel} is thus the set of $H_{n}$-orbits of $\mathbb{P}^{1}_{\Z}(\Z/p^{k}\Z)$. But this is the set of $H_{k}$-orbits of $\mathbb{P}^{1}_{\Z}(\Z/p^{k}\Z)$, which in turn can be identified with $\Gamma_{0}(\Z/p^{k}\Z) \bs \SL_{2}(\Z/p^{k}\Z)/H_{k}$. One easily checks that this map is the same %
as the %
reduction map.   %
\end{proof}
\begin{remark}
We note here that the analogue of \cref{lem:IwahoriToBorel} for $\PSL_{2}(\Z/p^{n}\Z)$ also holds, and the same proof can be used. 
\end{remark}

We now return to an arbitrary $N=p^{n}M$ with $(M,p)=1$ and write $\Gamma$ for one of our functors. We will first assume $M\geq{3}$. %
To find the pruned skeleton of the curve corresponding to $\Gamma(\Z/N\Z)$, we first determine the pruned skeleton of the curve corresponding to $P(H_{M})$ for $H_{M}=\Gamma(\Z/M\Z)$ as in \cref{lem:ReconstructGeneralizedTree}. Write $\mathcal{T}_{P(H_{M}),can}$ for this tree. Over each edge, we consider the group $E(H_{M})$, corresponding to the full level $p^{n}M$-structure curve over $X_{P(H_{M})}$. We can also consider a smaller curve here with just full level $p^{n}$-structure, but this makes no difference in the end for our formulas. The covering $X_{H}\to X_{P(H_{M})}$ then corresponds to the inclusion of subgroups $H\subset E(H_{M})$, where $H=\Gamma(\Z/N\Z)$. We then have the equality of left coset spaces 
\begin{equation*}
\SL_{2}(\Z/p^{n}\Z)/\epsilon(H_{p})\to E(H_{M})/H.
\end{equation*}
We now note that $\epsilon(H_{p})=H_{p}$ for all of the functors under consideration. Indeed, if $\Gamma\neq\Gamma_{1}$, then $\pm H_{p}=H_{p}$. If $\Gamma=\Gamma_{1}$, then this is by definition. In other words, we can use the material from \cref{pro:BorelDCS1}. If $M=1,2$, then some caution has to be exercised. If $M=1$, then we introduce an auxiliary $M'\geq{3}$ and set $H_{M'}=\SL_{2}(\Z/M'\Z)$ in the formulas above. For $\Gamma\neq \Gamma_{1}$ this doesn't change anything, but for $\Gamma=\Gamma_{1}$ one finds $\epsilon(H_{p})=\Gamma^{\pm}_{1}(\Z/p^{n}\Z)$, which gives different orbit lengths and thus different edge lengths. For $M=2$ one similarly has to be careful, as $\PSL_{2}(\Z/2p^{n}\Z)\simeq \PSL_{2}(\Z/p^{n}\Z)\times \SL_{2}(\Z/2\Z)$. 

The data above completely determines the structure of the covering over the generalized Gauss vertex in $\mathcal{T}_{\Gamma(\Z/M\Z),can}$ (which is the pre-image of the Gauss vertex under the natural map to $X(1)^{\an}$). %
To determine the behavior over a supersingular edge in $\mathcal{T}_{\Gamma(\Z/M\Z),can}$, we then reduce these representatives %
modulo lower powers $p^{k}<p^{n}$. By \cref{lem:IwahoriToBorel}, this gives the double coset spaces over every supersingular edge. 

It is often useful to reverse this construction and start on the outside of a supersingular edge of $\mathcal{T}_{\Gamma(\Z/M\Z),can}$. We illustrate this for various subgroups here. %
For $\Gamma_{0}$, we start with a single vertex and two outgoing edges. In terms of \cref{lem:IwahoriToBorel}, we view these as corresponding to $[1:0]$ and $[0:1]$ in $\mathbb{P}^{1}_{\Z}(\F_{p})$. The edge corresponding to $[0:1]$ is stable, in the sense that there is no further bifurcation when moving to the central vertex. For the $[1:0]$-edge, we iteratively split the $[1:0]$-branch into three branches. At the first step for instance, these three branches correspond to the orbits of %
$[1:0]$, $[1:p]$ and $[1:rp]$ in $\mathbb{P}^{1}_{\Z}(\Z/p^{2}\Z)$. Here $[1:p]$ and $[1:rp]$ remain stable when moving to the central vertex, and $[1:0]$ keeps splitting into three branches until $\mathbb{P}^{1}_{\Z}(\Z/p^{n}\Z)$ is reached. These branches then connect to the pre-images of the corresponding generalized Gauss vertex.  The final result can be found in \cref{fig:X0(pn)}. 

For $\Gamma_{1}$, the structure is similar: %
we start with a single vertex and two outgoing edges, corresponding to $(1,0)$ and $(0,1)$. We then split $(1,0)$ at every step into $\phi(p^{k_{0}})+1$ branches: one for every $u\in U/U_{k}\cdot U_{n-k}$ and one for $(1,0)$. %

For $\Gamma_{sp}$, one easily sees that the local pruned skeleton consists of two copies of $\Gamma_{0}$. To obtain the skeleton for $\Gamma^{+}_{sp}$, we quotient the skeleton for $\Gamma_{sp}$ by $\pi$. The action of $\pi$ is free for $p\equiv{1}\bmod{4}$. For $p\equiv{3}\bmod{4}$, the edges corresponding to $([0:1],[1:r])$ and $([0:1],[1:0])$ are stabilized by $\pi$, so that we obtain one additional edge in the quotient compared to the case where $p\equiv{1}\bmod{4}$.    %

We record the result of this construction in a theorem.             

\begin{theorem}\label{thm:MainThm5}
The structure of the local pruned skeleton of the modular curves corresponding to the functors $\Gamma_{0}$, $\Gamma_{1}$, $\Gamma_{sp}$ and $\Gamma^{+}_{sp}$ is given by the procedure described above. %
 The global pruned skeleton is obtained by gluing $s(H_{M})$ copies of this pruned skeleton at the endpoints, where $s(H_{M})$ is the number of generalized supersingular $j$-invariants associated to $P(\Gamma(\Z/M\Z))$. The edge lengths are determined by \cref{rem:EdgeLengths}, \cref{pro:BorelDCS1} and \cref{lem:IwahoriToBorel}.   %
\end{theorem}

\begin{remark}
We can compare these skeleta with the graphs for $X_{0}(p^{n})$ obtained in \cite{DR73}, \cite{Edixhoven90}, \cite{CM2010} and \cite{Tsushima2015} for $n\leq{4}$.  It is a pleasant exercise to see that our approach gives exactly the same graphs, with the same edge lengths (which are sometimes called thicknesses or intersection multiplicities in the literature). %
Note that our approach ignores most of the intermediate components, which correspond to subdivisions of our graphs.   
\end{remark}

\subsection{Component groups for $X_{0}(N)$}\label{sec:CompX0N}

We now use the material from the previous section to explicitly give the monodromy matrix with respect to $p$ of the modular curve $X_{0}(N)$. %
As before, $p$ will be a prime not equal to $2$ or $3$, and we will consider a fixed integer $N=p^{n}M$, where $(M,p)=1$. We write $\mathcal{T}_{0,M}=\mathcal{T}_{\Gamma_{0}(M),can}$ for the induced canonical metric tree associated to $X_{0}(M)^{\an}$. %
This tree has three different types of edges, corresponding to the local ramification behavior of the point in the covering $X_{0}(M)^{\an}\to X(1)^{\an}$. %
\begin{definition}\label{def:TypeOfSupersingularPoint}
Let $\mathcal{S}\subset \Q^{unr}_{p}\subset \C_{p}$ be a set of lifts of the supersingular $j$-invariants over $\overline{\F}_{p}$, which we view as $\C_{p}$-valued points of $X(1)$. %
Let $\mathcal{S}_{M}$ be the inverse image of $\mathcal{S}$ under the map $X_{0}(M)\to X(1)$, %
whose elements correspond to the edges of the canonical supersingular tree $\mathcal{T}_{0,M}$.  For $k\in\{1,2,3\}$, we say that $j\in\mathcal{S}_{M}$ is of type $k$ if the corresponding line segment in $\mathcal{T}_{0,M}$ has length $kp/(p+1)$ with respect to the normalization $v(p)=1$. We denote the number of edges of type $1$ in $\mathcal{T}_{0,M}$ by $u$. The number of edges of type $2$ are denoted by $r_{1728}$, and the number of edges of type $3$ are denoted by $r_{0}$. 
\end{definition}
\begin{remark}
Note that if a point in $\mathcal{S}_{M}$ lies over a point in $\mathcal{S}\bs \{0,1728\}$, %
then the corresponding edge is automatically of type $1$. For a point of $\mathcal{S}_{M}$ lying over $1728$, its corresponding edge is of type $1$ if and only if the map $X_{0}(M)\to X(1)$ is ramified at this point; it is of type $2$ if and only if this map is unramified at this point. %
The same interpretation also holds for $j=0$ by replacing type $2$ with type $3$.    
\end{remark}
\begin{remark}
For brevity, we will assume for the remainder of this section that $p>11$, so that there exists at least one supersingular $j$-invariant of type $1$ in $\mathcal{S}_{M}$. We leave the other cases, which follow a similar pattern, to the reader. We will also assume $n\geq{2}$ and refer to \cite[Appendix 1]{Mazur1977} for $n=1$ (although our techniques work in exactly the same way in this case). %
We will moreover normalize our valuation according to Krir's theorem. In other words, if $\pi$ is a uniformizer for $L_{n}$, then $v(\pi)=1$ and %
$v(p)=c_{n}$, where $c_{n}:=[L^{\un}_{n}:\Q^{\un}_{p}]=p^{2(n-2)}(p^2-1)$. %
\end{remark}

The main difficulty in computing the component group is finding a suitable basis of the homology of $\Sigma^{\pr}(X_{0}(N))$. We give a basis here that makes the intersection matrices almost tridiagonal, which greatly simplifies the calculations. Note that one can in principle always make these intersection matrices tridiagonal using Lanczos' algorithm.   %

\begin{definition}
Let $\mathcal{S}_{M}$ be the set of supersingular elliptic curves 
of $X_{0}(M)$ lying over $\mathcal{S}$, %
and let $j_{0}\in\mathcal{S}_{M}$ be a fixed supersingular $j$-invariant of type $1$ as in \cref{def:TypeOfSupersingularPoint}. For every $j_{i}\in\mathcal{S}_{M}\bs\{j_{0}\}$, we will construct a set of $2n-1$ paths in $\Sigma^{pr}(X_{0}(N))$. In the construction, we use points in $\mathbb{P}^{1}_{\Z}(\Z/p^{n}\Z)$ to mark the different points over the central vertex of $\Sigma^{pr}(X_{0}(N))$. 

Starting at $[0:1]$, there is a unique line segment in the graph over $j_{0}$ that connects $[0:1]$ to $[1:rp]$. Similarly, there is a unique line segment that connects $[1:rp]$ back to $[0:1]$ in the graph over $j_{i}$. This is the first path. For the second path, we start at $[1:rp]$ and connect this to the vertex $[1:p]$ using the two unique line segments over $j_{0}$ and $j_{i}$. For the third path, we start at $[1:p]$ and connect this to $[1:rp^2]$ using the two line segments over $j_{0}$ and $j_{i}$, and so on. We denote these paths by $\gamma_{k,i}$, where $k$ stands for the $k$-th path, and $i$ for the $j$-invariant $j_{i}$. Note that these paths $\gamma_{k,i}$ form a basis for $H_{1}(\Sigma^{pr}(X_{0}(N)))$. We call this the \emph{ladder basis} of $H_{1}(\Sigma^{pr}(X_{0}(N)))$, see \cref{fig:LadderPicture}.  %
\end{definition}

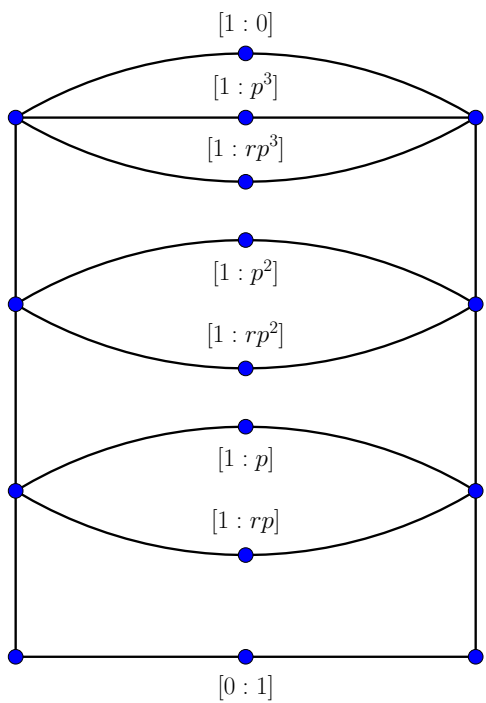
\begin{figure}[h]
\scalebox{0.55}{
\begin{tikzpicture}[line cap=round,line join=round,>=triangle 45,x=1cm,y=1cm]
\clip(-2,-3.5) rectangle (12.5,14.515199965509472);
\draw [line width=1.6pt] (0,-2)-- (11,-2);
\draw [line width=1.6pt] (0,-2)-- (0,2);
\draw [line width=1.6pt] (11,-2)-- (11,2);
\draw [shift={(5.5,-7)},line width=1.6pt]  plot[domain=1.0222469243443686:2.119345729245425,variable=\t]({1*10.547511554864494*cos(\t r)+0*10.547511554864494*sin(\t r)},{0*10.547511554864494*cos(\t r)+1*10.547511554864494*sin(\t r)});
\draw [shift={(5.5,11)},line width=1.6pt]  plot[domain=4.163839577934161:5.260938382835217,variable=\t]({1*10.547511554864494*cos(\t r)+0*10.547511554864494*sin(\t r)},{0*10.547511554864494*cos(\t r)+1*10.547511554864494*sin(\t r)});
\draw [line width=1.6pt] (0,2)-- (0,6.5);
\draw [line width=1.6pt] (11,2)-- (11,6.5);
\draw [shift={(5.5,-2.5)},line width=1.6pt]  plot[domain=1.0222469243443686:2.119345729245425,variable=\t]({1*10.547511554864494*cos(\t r)+0*10.547511554864494*sin(\t r)},{0*10.547511554864494*cos(\t r)+1*10.547511554864494*sin(\t r)});
\draw [line width=1.6pt] (0,11)-- (11,11);
\draw [shift={(5.5,2)},line width=1.6pt]  plot[domain=1.0222469243443686:2.119345729245425,variable=\t]({1*10.547511554864494*cos(\t r)+0*10.547511554864494*sin(\t r)},{0*10.547511554864494*cos(\t r)+1*10.547511554864494*sin(\t r)});
\draw [shift={(5.5,15.5)},line width=1.6pt]  plot[domain=4.163839577934161:5.260938382835217,variable=\t]({1*10.547511554864494*cos(\t r)+0*10.547511554864494*sin(\t r)},{0*10.547511554864494*cos(\t r)+1*10.547511554864494*sin(\t r)});
\draw [shift={(5.5,20)},line width=1.6pt]  plot[domain=4.163839577934161:5.260938382835217,variable=\t]({1*10.547511554864494*cos(\t r)+0*10.547511554864494*sin(\t r)},{0*10.547511554864494*cos(\t r)+1*10.547511554864494*sin(\t r)});
\draw [line width=1.6pt] (0,11)-- (0,6.5);
\draw [line width=1.6pt] (11,11)-- (11,6.5);
\begin{scriptsize}
\draw [fill=qqqqff] (0,-2) circle (5pt);
\draw [fill=qqqqff] (11,-2) circle (5pt);
\draw [fill=qqqqff] (0,2) circle (5pt);
\draw [fill=qqqqff] (11,2) circle (5pt);
\draw [fill=qqqqff] (0,6.5) circle (5pt);
\draw [fill=qqqqff] (11,6.5) circle (5pt);
\draw [fill=qqqqff] (0,11) circle (5pt);
\draw [fill=qqqqff] (11,11) circle (5pt);
\draw [fill=qqqqff] (5.5,3.5475115548644944) circle (5pt);
\draw[color=black] (5.5,2.75) node {\Large{$[1:p]$}};
\draw [fill=qqqqff] (5.5,0.4524884451355061) circle (5pt);
\draw[color=black] (5.5,1.25) node {\Large{$[1:rp]$}};
\draw [fill=qqqqff] (5.5,8.047511554864494) circle (5pt);
\draw[color=black] (5.5,7.25) node {\Large{$[1:p^2]$}};
\draw [fill=qqqqff] (5.5,4.952488445135507) circle (5pt);
\draw[color=black] (5.5,5.75) node {\Large{$[1:rp^2]$}};
\draw [fill=qqqqff] (5.5,9.452488445135508) circle (5pt);
\draw[color=black] (5.5,10.25) node {\Large{$[1:rp^3]$}};
\draw [fill=qqqqff] (5.5,11) circle (5pt);
\draw[color=black] (5.5,11.75) node {\Large{$[1:p^3]$}};
\draw [fill=qqqqff] (5.5,12.547511554864494) circle (5pt);
\draw[color=black] (5.5,13.25) node {\Large{$[1:0]$}};
\draw [fill=qqqqff] (5.5,-2) circle (5pt);
\draw[color=black] (5.5,-2.75) node {\Large{$[0:1]$}};
\end{scriptsize}
\end{tikzpicture}
}
\caption{\label{fig:LadderPicture}
The ladder-like picture for the pruned skeleton of $X_{0}(p^{n}M)$. Here we have depicted the part of the pruned skeleton lying over two supersingular $j$-invariants $j_{0}$ and $j_{i}$ for $n=4$. The other parts are connected to the central vertices in a similar fashion. 
}
\end{figure}

\begin{definition}

Let $j_{i}$ be a supersingular $j$-invariant of type $1$ not equal to $j_{0}$. The entries $a_{k,k'}$ of the $(2n-1)\times{(2n-1)}$ internal unramified monodromy matrix $A_{int}$ are $\langle \gamma_{k,i}, \gamma_{k',i}\rangle$. This is independent of our choice of $j_{i}$. 
\end{definition}
We can use this to build up the rest of the monodromy matrix as follows. 

\begin{lemma}\label{lem:LocalMonodromyFactors}
Set %
\begin{align*}
\mathcal{O}&=2p^{n-2},\\
\mathcal{B}&=(p-1)p^{n-2},\\
\mathcal{E}&=(p-1)^2p^{n-3},
\end{align*}
and
$$v=[2\mathcal{O}+2\mathcal{B},4\mathcal{O},4\mathcal{O}+2\mathcal{E},4\mathcal{O},4\mathcal{O}+2\mathcal{E},...,4\mathcal{O}+2\mathcal{E},4\mathcal{O},2\mathcal{O}+2\mathcal{B}].$$ The internal monodromy matrix $A_{int}$ is a tridiagonal matrix whose main diagonal is $v$, and whose off-diagonals are $[2\mathcal{O},....,2\mathcal{O}]$. %
The monodromy matrix $A$ is given by $(s-1)^2$ square block matrices, which we identify with pairs $(j_{1},j_{2})$ of supersingular $j$-invariants not equal to $j_{0}$. %
The block corresponding to $(j_{1},j_{2})$ for $j_{1}\neq{j_{2}}$ is $1/2A_{int}$. %
The diagonal blocks corresponding to $(j,j)$ are given by
\begin{itemize}
\item $A_{int}$ if $j$ is of type $1$, 
\item  
$3/2A_{int}$ if $j$ is of type $2$, 
\item $2A_{int}$ if $j$ is of type $3$. 
\end{itemize}
\end{lemma}
\begin{proof}
This is a simple calculation using \cref{thm:MainThm5}, \cref{rem:EdgeLengths} and the material in \cref{pro:BorelDCS1}. For instance, the off-diagonals in $A_{int}$ can be calculated as follows. Each of these corresponds to twice the length of the line segment $e$ starting at 
$[1:p^{i}]$ or $[1:dp^{i}]$ as in \cref{fig:LadderPicture}, and going to either of the adjacent $4$-valent vertices. If we interpret the pruned skeleton as a ladder, then $e$ can be seen as half a rung. Multiplying the normalization factor from Krir's theorem $p^{2(n-2)}(p^2-1)$, the inverse of the local degree $\phi(p^{n-i})/2$, and the length $p^{1-i}/(p+1)$ of the image of $e$ in $X(1)^{\an}$, we then find  
\begin{equation*}
\ell(e)=\dfrac{2p^{1-i}p^{2(n-2)}(p^2-1)}{\phi(p^{n-i})(p+1)}=2p^{n-2}=
\mathcal{O}.
\end{equation*}
Similarly, we can calculate the length of a single side-rail $e$ (going up vertically) as follows. Assume first that it is not the bottom or top side-rail. The dilation factor is then $p^{n-i-1}$ and the length of the image of $e$ in $X(1)^{\an}$ is $(p-1)/(p^{i}(p+1))$. Multiplying these as before, we find     
\begin{equation*}
\ell(e)=\dfrac{p^{2(n-2)}(p^2-1)(p-1)}{p^{i}(p+1)p^{n-i-1}}=p^{n-3}(p-1)^{2}=\mathcal{E}.
\end{equation*}
Similarly, by calculating the lengths of half of the top and bottom side-rails, one finds $\mathcal{B}$. By adding these factors, one easily obtains the structure of the monodromy matrix as stated in the lemma.  
\end{proof}
To state the final result more succinctly, we introduce the following notation.
\begin{definition}
Let $b\geq{1}$ and $n\geq{2}$ be integers, and let $p$ be a prime number. %
We define the basic building block associated to $n$, $p$ and $b$ to be  %
\begin{equation*}
G_{b}=G_{n,p,b}=(\Z/bp^{n-2}\Z)^{2n-2}\times \Z/b(p^n-p^{n-2})\Z.
\end{equation*}
If $c$ is an integer strictly greater than zero then $G^{c}_{b}$ is a direct product of $c$ copies of $G_{b}$. If $c=0$, then $G^{0}_{b}=(1)$. 
\end{definition}
Let $N$ be a $\Z$-module and let $k$ be an integer. We write $N[1/k]$ for the induced $\Z[1/k]$-module obtained by localizing. 
The following theorem can be seen as a generalization of the results for $X_{0}(pM)$ in 
\cite[Appendix A, Section 2]{Mazur1977}. 

\begin{theorem}\label{thm:MainThmComponentGroup}
Let $p>11$ be a prime number and let $\Psi_{N}:=\Psi_{N}(\overline{\F}_{p})$ be the geometric group of connected components of the special fiber of the N\'{e}ron model of the Jacobian of the modular curve $X_{0}(p^{n}M)$ over Krir's extension $L_{n}\supset \Q^{un}_{p}$. If $u>1$\footnote{Note that $u>1$ as soon as $p>23$.}, then
\begin{equation*}
\Psi_{N}[1/2]\simeq G_{1}[1/2]\times G^{u-2}_{3}[1/2]\times G^{r_{1728}}_{5}[1/2]\times G^{r_{0}}_{7}[1/2].
\end{equation*}
If $u=1$, then 
\begin{equation*}
\Psi_{N}[1/2]\simeq G^{r_{1728}}_{5}[1/2]\times G^{r_{0}}_{7}[1/2].
\end{equation*}
\end{theorem}
\begin{proof}
We start by row reducing the monodromy matrix $A$ with respect to the ladder basis over the principal ideal domain $\Z[1/2]$ to obtain a matrix with multiples of $A_{int}$ on the diagonal. %
Consider the case where $u>1$. %
We then obtain one copy of $A_{int}$ and $u-2$ copies of $3/4A_{int}$ on the diagonal from the supersingular $j$-invariants of type $1$, see the formulas in \cref{lem:LocalMonodromyFactors}. If $p\equiv{5}\bmod{12}$, then
we obtain an additional $r_{1728}$ copies of $5/4A_{int}$. 
If $p\equiv{7}\bmod{12}$, then we obtain 
an additional $r_{0}$ copies of $7/4A_{int}$. 
If $p\equiv{11}\bmod{12}$, then we obtain all of the above.  
If $u=1$, then there are no copies of $A_{int}$ or $3/4A_{int}$, so that everything reduces to $5/4A_{int}$ or $7/4A_{int}$.  

It now suffices to row-reduce the tridiagonal matrix $A_{int}$ over $\Z[1/2]$ to an upper-triangular matrix. We first determine several determinants. Let $h(p)=(p+1)^2/p$ and $B_{int}=A_{int}/\mathcal{O}$. Let $K_{m}$ be the determinant of the $m\times{m}$-tridiagonal matrix obtained by removing the last $2n-1-m$ rows and columns of $B_{int}$.     
The recursion formula for generalized continuants %
 gives %
\begin{equation*}
K_{m}=a_{m}K_{m-1}-4K_{m-2},
\end{equation*}  
where $a_{1}=p+1=a_{2n-1}$, $a_{2k}=4$ for $1\leq{k}\leq n-1$,  and $a_{2k+1}=h(p)$ for $1\leq k < n-1$. Note that this recursion formula follows from two cofactor expansions. %
A simple exercise in induction shows that  %
\begin{align*}
K_{2k}&=2^{2k}p^{k} \,\text{ for }1\leq{k}\leq{n-1},\\
K_{2k+1}&=2^{2k}(p^{k+1}+p^{k}) \,\text{ for }1\leq{k}<n-1.
\end{align*}
We now calculate the last continuant $K_{2n-1}$, which gives the determinant of $B_{int}$. From the last recursion we obtain %
\begin{equation*}
K_{2n-1}=(p+1)K_{2n-2}-4K_{2n-3}.
\end{equation*} 
Using the formulas above, we then find
\begin{align*}
K_{2n-1}&=(p+1)K_{2n-2}-4K_{2n-3}\\
&=2^{2n-2}p^{n}+2^{2n-2}p^{n-1}-2^{2n-2}(p^{n-1}+p^{n-2})\\
&=2^{2n-2}(p^{n}-p^{n-2}).
\end{align*}
The determinants of the corresponding submatrices in %
$A_{int}$ are then given by $\mathcal{O}^{m}K_{m}$.

We now row-reduce $A_{int}$ over $\Q$ using the following general pattern: we first swap rows $1$ and $2$. We then subtract a suitable multiple of row $1$ from row $2$. We then swap rows $2$ and $3$, subtract a suitable multiple of row $2$ from $3$, and so on. Note that this automatically makes the matrix upper-triangular, with diagonal entries $2\mathcal{O}$ up till the last diagonal entry. Using the determinantal formulas found above, we directly find that the last diagonal entry is %
$\pm{}\mathcal{O}(p^{n}-p^{n-2})$. %

To determine the structure over $\Z[1/2]$, we now show that the multiples used in the row reductions are defined over $\Z[1/2]$. We first note that the determinants $\pm\mathcal{O}^{m}K_{m}$ of the $m\times{m}$-submatrices in $A_{int}$ are unchanged by the row reductions up to the part where rows $m$ and $m+1$ are swapped. This implies that the last diagonal term of this reduced $m\times{m}$-submatrix is
$\pm 2^{t}K_{m}\mathcal{O}$ for some $t$ (the factor $2^{t}$ comes from the remaining diagonal factors $2\mathcal{O}$). Since these are a multiple of $2\mathcal{O}$ in $\Z[1/2]$, we conclude. 
\end{proof}

We can also deduce a general formula for %
the geometric Tamagawa numbers of $X_{0}(N)$ from the proof of \cref{thm:MainThmComponentGroup}. This includes the $2$-adic factors.  %

\begin{corollary}\label{cor:GeometricTamagawaNumbers}
Let $\Psi_{N}$ and $p$ be as in \cref{thm:MainThmComponentGroup} with %
$$c:=2^{4n-3}(p^{n}-p^{n-2})p^{(n-2)(2n-1)}=|\mathrm{det}(A_{int})|,$$ and let $s=|\mathcal{S}_{M}|=u+r_{0}+r_{1728}$ be the number of supersingular $j$-invariants of $X_{0}(M)$ over $\mathcal{S}$. If $u>1$, then 
\begin{equation*}
|\Psi_{N}|=c^{s-1}\cdot (3/4)^{(u-2)(2n-1)}\cdot (5/4)^{r_{1728}(2n-1)}\cdot (7/4)^{r_{0}(2n-1)}.
\end{equation*}
If $u=1$, then 
\begin{equation*}
|\Psi_{N}|=c^{s-1}\cdot (5/4)^{r_{1728}(2n-1)}\cdot (7/4)^{r_{0}(2n-1)}.
\end{equation*}
\end{corollary}
\begin{proof}
This follows from our calculations in the proof of \cref{thm:MainThmComponentGroup} on the determinant of $A_{int}$ and the structure of the reduced monodromy matrix. 
\end{proof}

\subsection{Future directions}\label{sec:FutureDirections}

We conclude the paper with an overview of possible future directions. In \cref{sec:CoveringsModularCurves}, we showed how the canonical supersingular tree together with the $\PSL_{2}(\hat{\Z})$-monodromy labeling defined in \cref{def:PSL2Labeling} gives the pruned skeleton of any modular curve. 
Rather than restricting to the canonical supersingular tree $\mathcal{T}_{can}$, 
we can also consider the more general chain of trees 
\begin{equation*}
\mathcal{T}_{can}\subset \mathcal{T}_{CM}\subset \mathcal{T}_{f-CM}\subset X(1)^{\an},
\end{equation*} 
where $\mathcal{T}_{CM}$ is the tree generated by all elliptic curves $E/\mathbb{C}_{p}$ with complex multiplication, and $\mathcal{T}_{f-CM}$ is the tree generated by all elliptic curves $E/\C_{p}$ with formal complex multiplication in the sense of \cite{Gross86}. By \cite[Corollary 1.4]{HMRL2021}, the space $\mathcal{T}_{CM}$ is dense in $\mathcal{T}_{f-CM}$. This implies that %
we can obtain the structure of the \emph{full minimal skeleton} of any modular curve by finding the  subdivision of $\mathcal{T}_{CM}$ (including the metric structure) induced by Weinstein's type-$2$ points $[x,n]$ (see \cref{sec:WildTower} for the notation) with their associated $\PSL_{2}(\hat{\Z})$-labeling. 

In \cref{sec:CompX0N}, we used our reconstruction result to find the structure of the group of connected components of $X_{0}(N)$. In future work, we will study the action of the Hecke operators on these component groups. We expect this to be completely combinatorial (being dictated by the various natural relations between the Hasse invariants of quotients, see \cite{Buzzard2003}), so that we can speak of \emph{tropical Hecke operators}. In \cite[Proposition 3.14]{Ribet1990} for $n=1$ and \cite[Th\'{e}or\`{e}me 1]{Edixhoven1991} for general $n$, it is shown that the action of the Hecke operators on the component groups for $X_{0}(p^{n}M)$ is Eisenstein: we have the identity $T_{\ell}=\ell+1$ for $\ell\neq{p}$. Note that the component group here is taken over $\Q_{p}^{\unr}$, and it usually does not behave well under base change. Using our description for the pruned skeleton, it might be possible to obtain similar results for the Hecke operators in the semistable case as well.

We also mention here that a great deal is known about the component group of the N\'{e}ron model of $J_{0}(N)$ over the field of $p$-adic numbers $\Q_{p}$ (rather than Krir's extension) by the results in \cite{Lorenzini1995}. It would be interesting to compare the two results and see if we can further bound the minimal extension over which semistable reduction is attained. 

To apply the results in this paper to other modular curves $X_{H}$, one needs a suitable extension over which $X_{H}$ attains semistable reduction (more precisely, only the ramification degree over $\Q^{\unr}_{p}$ is needed). For $X_{0}(N)$, such an extension is given by Krir's theorem in terms of local class field theory. The author is however unaware of similar extensions for other modular curves. %

Two functors in our list that are missing are the ones associated to the non-split Cartan and its normalizer. There are no theoretical obstructions to applying the ideas presented here to these subgroups, and we expect the calculations to follow a similar pattern to the one given in \cref{sec:CosetSchemes}. Here one should replace %
$\mathbb{P}^{1}_{\Z}(\Z/p^{n}\Z)$ with $\mathbb{P}^{1}_{\Z}(W(\mathbb{F}_{p^2})/(p^{n}))$, where $W(\mathbb{F}_{p^2})$ is the ring of Witt vectors associated to $\mathbb{F}_{p^{2}}$.  %

Finally, we would like to mention possible applications to the theory of $p$-adic heights and descent problems. Once we know the action of the Hecke operators on the homology of the pruned skeleton and the cohomology of the residue curves (which in principle follow from Weinstein's recipe), %
we can recover the measures $\mu_{F}$ considered in \cite[Section 12.1.1]{BD2019}. This in turn allows us to obtain the possible $p$-adic heights away from $p$,  
see the discussion in \cite[Section 3.1]{BDMT2023}.

\bibliographystyle{alpha}
\bibliography{main}{}

\end{document}